\documentclass[dvipsnames, reqno]{amsart}

\usepackage{amsfonts,amsthm,amssymb,amsmath,stackrel, latexsym,mathabx}

\usepackage{mathrsfs}
\usepackage{tikz-cd}

\usepackage{newcent}
\usepackage{upgreek}

\usepackage{xcolor}
\usepackage{graphicx}
\usepackage[all]{xy}

\usepackage{microtype}
\usepackage{xcolor}
\usepackage{graphicx}
\usepackage[all]{xy}
\usepackage[colorlinks=true,citecolor=MidnightBlue,linkcolor=MidnightBlue,urlcolor=MidnightBlue] 
           {hyperref}  

           \usepackage[inline]{enumitem}
\usepackage{a4wide}
\usepackage{stmaryrd}
\usepackage{mathtools,
  thmtools}

\makeatletter
\DeclareRobustCommand*\cal{\@fontswitch\relax\mathcal}
\makeatother

\makeatletter
\@namedef{subjclassname@2020}{%
  \textup{2020} Mathematics Subject Classification}
\makeatother

\numberwithin{equation}{section}

\usepackage{mathtools}
	
\usepackage[toc,page]{appendix}

\usepackage[bbgreekl]{mathbbol}
\usepackage{bbm}
	
\usepackage[normalem]{ulem}

\usepackage{scalerel,stackengine}  
	
\usepackage[textsize=scriptsize]{todonotes}

\def\longleftrightarrows{\buildrel
  {\displaystyle{\relbar\joinrel\longrightarrow}} \over
  {\longleftarrow\joinrel\relbar}}

\newcommand{\compactlist}[1]{\setlength{\itemsep}{0pt} \setlength{\parskip}{0pt} \setlength{\leftskip}{-1.#1em}}

\allowdisplaybreaks
	
\theoremstyle{plain}
\newtheorem{theorem}{Theorem}[section]
\newtheorem{lemma}[theorem]{Lemma}
\newtheorem{proposition}[theorem]{Proposition}
\newtheorem{corollary}[theorem]{Corollary}

\declaretheorem[name=Theorem,
  numbered=no]{theorem*}
\theoremstyle{definition}
\newtheorem{definition}[theorem]{Definition}
\theoremstyle{remark}
\newtheorem{remark}[theorem]{Remark}
\declaretheorem[name=Remark, numbered=no]{remark*}

\declaretheorem[name=Example,
  ,sibling=theorem]{example}
	
\DeclareRobustCommand{\SkipTocEntry}[5]{}

\def\kasten#1{\mathop{\mkern0.5\thinmuskip
		\vbox{\hrule
			\hbox{\vrule
				\hskip#1
				\vrule height#1 width 0pt
				\vrule}%
			\hrule}%
		\mkern0.5\thinmuskip}}

\newcommand{\bx}{{\kasten{6pt}}}

\newcommand{\Tor}{{\rm Tor}}
\newcommand{\Ext}{{\rm Ext}}
\newcommand{{\bull}}{{\scriptscriptstyle{\bullet}}}

\newcommand{{\op}}{{{\rm op}}}
\newcommand{{\co}}{{{\rm co}}}

\newcommand{\ga}{\alpha}
\newcommand{\gb}{\beta}  
\newcommand{\gG}{\Omega}
\newcommand{\gd}{\delta}
\newcommand{\gD}{\Delta}
\newcommand{\gve}{\varepsilon}
  
\newcommand{\gvf}{\varphi}  

\newcommand{\gl}{\lambda}

\newcommand{\go}{\omega}
\newcommand{\gO}{\Omega}

\newcommand{\gs}{\sigma}

\newcommand{{\Aop}}{{A^{\rm op}}}
\newcommand{{\Rop}}{{R^{\rm op}}}
\newcommand{{\Uop}}{{U^{\rm op}}}
	
\newcommand{{\Aope}}{({A^{\rm op}})^{\rm e}}
\newcommand{{\Ae}}{{A^{\rm e}}}
\newcommand{{\RRe}}{{R^{\rm {\mkern 1mu} e}}}

\newcommand{{\coop}}{{{\rm coop}}}

\newcommand{{\id}}{{{\mathrm{id}}}}

\newcommand{{\dd}}{{{\mathrm{d}}}}
	
\newcommand{\lact}{\smalltriangleright}                  
\newcommand{\ract}{\smalltriangleleft}
\newcommand{\blact}{\blacktriangleright}
\newcommand{\bract}{\blacktriangleleft}
	
\newcommand\pig[1]{\scalerel*[5pt]{\big#1}{%
	\ensurestackMath{\addstackgap[1.5pt]{\big#1}}}}

\newcommand{\cG}{\mathcal{G}}

\newcommand{\cH}{\mathcal{H}}
\newcommand{\cL}{\mathcal{L}}
\newcommand{\cM}{\mathcal{M}}

\newcommand{\C}{\mathbb{C}}

\newcommand{\Sum}{\textstyle\sum\limits}

\newcommand{\due}[3]{{}_{{#2 }} {#1}_{{ #3}}\,}

\newcommand{\smap}{{\raisebox{0.3pt}{${{\scriptscriptstyle {[+]}}}$}}}
\newcommand{\smam}{{\raisebox{0.3pt}{${{\scriptscriptstyle {[-]}}}$}}}

\title{Takeuchi-Schneider equivalence and calculi for homogeneous spaces of Hopf algebroids}

\author[N.\ Kowalzig]{Niels Kowalzig}
\address{\hskip -7pt N.K.: Dipartimento di Matematica, Universit\`a di Roma Tor Vergata, Via della Ricerca Scientifica~1, 00133 Roma, Italy}
\email{niels.kowalzig@uniroma2.it}
	
\author[Th.\ Weber]{Thomas Weber}
\address{\hskip -7pt Th.W.: Mathematical Institute of Charles University, Sokolovsk\'a 49/83, 186 75 Praha 8,
  Czech Republic
}
\email{thomas.weber@matfyz.cuni.cz}



\begin{document}

\begin{abstract}
We develop a notion of covariant differential calculus for Hopf algebroids. As a byproduct, we prove analogues of the fundamental theorem of Hopf modules and a Takeuchi-Schneider equivalence in the realm of Hopf algebroids. The resulting categorical equivalences enable us to classify covariant calculi on Hopf algebroids and, more in general, covariant calculi on quantum homogeneous spaces in this context, in terms of substructures of the augmentation ideal. This generalises the well-known classification results of Woronowicz and Hermisson. A particular focus is given on examples, including covariant calculi on the Ehresmann-Schauenburg Hopf algebroid of a faithfully flat Hopf-Galois extension, and covariant calculi on scalar extension Hopf algebroids, as well as homogeneous space variants of the latter.
\end{abstract}

\maketitle

\tableofcontents

\section{Introduction}

In noncommutative geometry, 
Hopf algebras have been proven to be appropriate generalisations of groups, in particular as a way to describe the concept of symmetry.
For instance, quantum groups 
provide deformations of classical groups which are relevant in mathematical physics since they correspond to solutions of the quantum Yang-Baxter equation. This leads, among various other applications, to the construction of topological invariants of 3-manifolds \cite{FaTa} and to solutions in the quantum inverse scattering method \cite{ReTu}. Moreover, Hopf algebras play the r\^ole of noncommutative structure groups in the geometric interpretation of Hopf-Galois extensions as noncommutative principal bundles \cite{Schneider}.

However, it turns out that, in certain situations, the Hopf algebraic approach is too restrictive, in the sense that one necessitates symmetries with a possibly noncommutative base algebra rather than a base field. This occurs in the study of principal bundles with groupoid symmetries \cite{Mrcun}, when considering solutions of the quantum dynamical Yang-Baxter equation \cite{EtVa},
when trying to interpret
non-integral values of the quantum dimensions in low-dimensional quantum
field theories \cite{BoeSzl:WHAIIRTDATMT}, or still when considering differential operators as universal enveloping algebras of Lie-Rinehart algebras \cite{KowKra:DAPIACT}. An algebraic notion governing the previously mentioned examples is provided by that of a {\em Hopf algebroid} in its various guises, that is, with \cite{Boe:HA} or without \cite{Schau:DADOQGHA} antipodes.
As a groupoid encodes generalised symmetries as an extension of the concept of groups, a Hopf algebroid enlarges the concept of symmetries to generalised symmetries in noncommutative geometry.
A Hopf algebroid has, to start with, an underlying structure of a {\em bialgebroid}, analogously to the fact that a Hopf algebra is, in particular, a bialgebra. Roughly speaking, a bialgebroid
can be seen as the generalisation of a bialgebra over a noncommutative base ring, which, however,  due to the noncommutativity cannot be described
as a compatible monoid and comonoid in some monoidal category (of bimodules) 
but rather as a monoid and a comonoid in different ones. This  entails a remarkable amount of technical complications,  of which we give a short account in \S\ref{heiszheiszheisz} and in general refer to \cite{Boe:HA} for more details; for example,
the fact that one now has to distinguish between {\em left} and {\em right} bialgebroids, or, analogously to the groupoid case, the presence of {\em source and target} maps, or still
the already mentioned difficulty of generalising the notion of antipode. Nevertheless, an increasing number of concepts have been carried over from Hopf algebras to Hopf algebroids resp.\ bialgebroids; so as, for example, the theory of Galois extensions as a successful application \cite{Kad:GTFBDTANHS, BalSzl:FGEONCB, BoeBrz:PTAE}.

As a convention throughout this paper, for the notion of Hopf algebroid (more precisely, of {\em right} Hopf algebroid over a {\em left} bialgebroid) we adopt the perspective of \cite{Schau:DADOQGHA}, that is, we do not require the existence of (whatever kind of) an antipode but solely the invertibility of a certain Hopf-Galois map.
  
On the level of Hopf algebras, there is a rich theory of differential calculi, initiated by Woronowicz in the seminal paper \cite{Woronowicz1989} and developed over the last decades, see, {\em e.g.}, the recent monograph \cite{BeggsMajid}. The idea is to understand noncommutative geometry through differential forms and model the latter algebraically as differential graded algebras which are generated in degree zero. It is effectively sufficient to understand the first degree of a differential cal\-cu\-lus as there is a canonical maximal prolongation \cite{Schau96}, and all possible extensions are quotients of the latter. Moreover, every algebra admits a universal first order differential cal\-cu\-lus \cite{Woronowicz1989} as the kernel of its multiplication and all possible first order differential calculi are again quotients of the latter. In case of a Hopf algebra $H$, this characterisation can be further simplified when dealing with covariant calculi, {\em i.e.}, first order differential calculi such that the bimodule $\Omega(H)$ of one-forms is a Hopf module and the differential $\mathrm{d}\colon H\to\Omega(H)$ colinear. Via the fundamental theorem of Hopf modules, see \cite[Thm.~4.1.1]{Swe:HA} or  \cite[Thm.~1.9.4]{MontBook},
we obtain $\Omega(H)\cong H\otimes\Lambda$, where $\Lambda$ denotes the subspace of one-forms which are invariant under the coaction, and observe furthermore that there is a surjective right $H$-linear map
\begin{equation*}
  \varpi\colon H^+\to\Lambda,
  \quad
	h\mapsto S(h_{(1)})\mathrm{d}(h_{(2)}),
\end{equation*}
that is, the {\em Maurer-Cartan form}, where $H^+$ denotes the kernel of the counit and $S$ the antipode. Thus, $I:=\ker\varpi\subseteq H^+$ is a right $H$-ideal and it turns out that there is a one-to-one correspondence between left covariant first order differential calculi on $H$ and right $H$-ideals of $H^+$. This is the classification theorem of Woronowicz \cite{Woronowicz1989}, which has been used to construct numerous covariant calculi, most famously on matrix quantum groups \cite{Jurco}.

There is a generalisation of the previous classification result proven by Hermisson \cite{Hermisson}, which applies to principal homogeneous spaces. The latter appear as the coinvariants of Hopf algebra surjections $\pi\colon A\to H$, where one assumes that $B:=A^{\mathrm{co}H}\subseteq A$ is a faithfully flat Hopf-Galois extension. Using the Takeuchi-Schneider equivalence \cite{Schneider,Tak:Rel}, it follows that any left $A$-covariant first order differential calculus $\Omega(B)$ on $B$ decomposes as 
\begin{equation*}
	\Omega(B)\cong A\bx_H\Lambda,
\end{equation*}
where $\bx_H$ denotes the cotensor product and $\Lambda$ denotes the subspace of left $A$-coinvariant forms on $B$. Similar to the approach before, one obtains an ideal in $B^+:=B\cap A^+$ as the kernel of a Maurer-Cartan form and then this gives a one-to-one correspondence between left $A$-covariant first order differential calculi on $B$ and subspaces of $B^+$ that are closed under the left $B$-action and left $H$-coaction. This result was used by Heckenberger and Kolb \cite{HeckKolb} to prove that there exists an essentially unique covariant differential calculus of classical dimension on every irreducible quantum flag manifold. The simplest example is the well-known differential calculus on the Podle\'s sphere \cite{Podles}, arising as the quotient of the noncommutative Hopf fibration. Many of these differential calculi have been proven to admit further interesting structures, such as Levi-Civita connections \cite{BGKO} or noncommutative K\"ahler structures \cite{Reamonn}.

The natural question arises if the previously described theory of covariant calculi generalises to the level of Hopf algebroids and if a similar characterisation strategy can be applied. It is the main purpose of this paper to provide a notion of covariant first order differential calculus for Hopf algebroids, discuss quantum homogeneous spaces in this context, and to prove according classification results in the spirit of Woronowicz and Hermisson as described before. In doing so, we mimic the previous strategy in \S\ref{sec:HopfModules} and \S\ref{sec:CovBim} by first discussing Hopf modules for (left) bialgebroids $(\mathcal{H},R)$. A {\em (right-left) Hopf module} $M$ is a right $\mathcal{H}$-module with a compatible left $\mathcal{H}$-comodule structure such that the underlying induced right $R$-actions coincide, and such that the natural left $R$-action commutes with the right $\cH$-action, see Definitions \ref{genzano}. Similarly, one defines a left-left variant as well as {\em covariant bimodules} as an intersection between these two and commuting left and right $\cH$-actions,
see Definition \ref{laparolacheciunisce}. The corresponding fundamental theorem, proven in Theorems \ref{thm:HopfModules} and \ref{rhodia}, is compactly summarised below (see the main text for details and notation).

\begin{theorem*}
  If a left bialgebroid 
 $(\cH, R)$ is also a right Hopf algebroid such that $\cH_\ract$ is $R$-flat,
 then the categories ${}_\mathcal{H}^\cH \!\mathcal{M}^{\phantom{}}_\cH$ of covariant bimodules and ${}_\mathcal{H} \cM$ of left $\mathcal{H}$-modules are equivalent through the adjunction
$$
{}^{\mathrm{co}\mathcal{H}\!}(-) \colon {}_\mathcal{H}^\cH \!\mathcal{M}^{\phantom{}}_\cH
\quad \raisebox{-3pt}{$\longleftrightarrows$} \quad {}_\mathcal{H} \cM \ \colon \!  \cH_\ract \otimes_R -, 
$$
where ${}^{\mathrm{co}\mathcal{H}\!}(-)$ denotes the functor of coinvariants and $\cH_\ract \otimes_R -$ is the tensor product functor.
\end{theorem*} 

Let us underline again that for proving this fundamental theorem the invertibility of a canonical Galois map is sufficient and no antipode is needed.

In \S\ref{sec:Wor}, we define covariant first order differential calculi on right Hopf algebroids as pairs $(\Omega,\mathrm{d})$, where $\Omega$ is a Hopf module, the differential $\mathrm{d}\colon\mathcal{H}\to\Omega$ is $R$-bilinear with respect to the $R$-actions induced by the source maps, left $\mathcal{H}$-colinear, satisfies the Leibniz rule, and, finally, its image
generates $\Omega$. These conditions do not imply that $\mathrm{d}$ vanishes on $R$ since, in general, it is not linear with respect to the target maps. After the construction of a universal calculus in this setting and giving the definition of a Maurer-Cartan form in terms of the translation map, we prove the following analogue  for Hopf algebroids of Woronowicz's classification theorem, see Theorem~\ref{thm:Wor}.

\begin{theorem*}
If $\mathcal{H}_\ract$ is $R$-flat, then left covariant calculi on $\mathcal{H}$ are in bijective correspondence with left ideals in the kernel $\mathcal{H}^+$ of the counit.
\end{theorem*}

The main tool to prove this result is the Hopf algebroid version of the fundamental theorem, together with the Maurer-Cartan form
$$
\varpi\colon\mathcal{H}^+\to{}^{\mathrm{co}\mathcal{H}}\Omega, \quad
h \mapsto 
\dd (h_{\smap })h_{\smam }
$$
defined in terms of the translation map $\ga^{-1}(1 \otimes_R h) =: h_{\smap } \otimes_R h_{\smam }$, where $\ga^{-1}$ is the inverse of a Galois map $\ga\colon	\mathcal{H}  \otimes_R \mathcal{H} \to \mathcal{H}  \otimes_R \mathcal{H}$ as in \eqref{eq:HG}.

In \S\ref{sec:TakSchMain}, we introduce quantum homogeneous space versions for Hopf algebroids and prove the
corresponding {\em Takeuchi-Schneider correspondence theorem} in \S\ref{sec:TakSchSub}. In more detail, given a Hopf algebroid surjection $\pi\colon(\mathcal{G},R)\to(\mathcal{H},R)$ such that the subalgebra $B:=\mathcal{G}^{\mathrm{co}\mathcal{H}}$ of coinvariants under the induced $\mathcal{H}$-coaction yields a faithfully flat Hopf-Galois extension $B\subseteq\mathcal{G}$, together with $R$-flatness conditions on $\mathcal{G}$, we call $B$ a {\em Hopf algebroid principal homogeneous space}, see Definition \ref{def:PHS}. In this setup, we define appropriate categories ${}^\mathcal{G}_B\mathcal{M}^{\phantom{}}_B$ and ${}^\mathcal{H}_B\mathcal{M}$ of Hopf modules and prove their Takeuchi-Schneider equivalence (see Proposition \ref{prop:Adj} and Theorem \ref{Kuchen}):

\begin{theorem*}
For every Hopf algebroid principal homogeneous space $B=\mathcal{G}^{\mathrm{co}\mathcal{H}}$ there is a categorical equivalence ${}^{\mkern 3mu \mathcal{G}}_B\mathcal{M}^{\phantom{\cG}}_B \cong {}^\mathcal{H}_B\mathcal{M}$ induced by the adjunction
$$
\Phi \colon {}^{\mkern 3mu\mathcal{G}}_B\mathcal{M}^{\phantom{\cG}}_B
\quad \raisebox{-3pt}{$\longleftrightarrows$} \quad {}^\mathcal{H}_B\mathcal{M} \ \colon \!  \Psi,
$$
where $\Phi$ and $\Psi$ are the Takeuchi functors as defined in
\S\ref{sec:TakSchSub}.
\end{theorem*}

In \S\ref{sec:calcPHS}, the Takeuchi-Schneider equivalence enables us to obtain yet another result, that is, the {\em Hermisson classification theorem} for Hopf algebroid principal homogeneous spaces, see Theorem \ref{thm:herm}:

\begin{theorem*}
Left $\mathcal{G}$-covariant calculi on Hopf algebroid principal homogeneous spaces are in bijective correspondence with subobjects $I\subseteq(B\cap\mathcal{G}^+)$ which are closed under the left $B$-action and left $\mathcal{H}$-coaction.
\end{theorem*}

Since in this article we introduce the new concept of differential calculus on Hopf algebroids, a major emphasis is set on examples. Our main classes of examples are given by Ehresmann-Schauenburg Hopf algebroids (\S\ref{gewittrigheute}), scalar extension Hopf algebroids (\S\ref{adapter}), and homogeneous space variants of scalar extension Hopf algebroids (\S\ref{sec:HomogeneousScalarExt}). Covariant calculi on all of the previous examples are constructed in \S\ref{sec:EScalc}, \S\ref{sec:scalar-calc}, and \S\ref{sec:QHScalc}, respectively. In more detail, we construct the Ehresmann-Schauenburg Hopf algebroid (see \cite{Schau98}) from a faithfully flat Hopf-Galois extension $B = {}^{\mathrm{co}H}\! A \subseteq A$, in the usual Hopf algebraic sense, as the space $\mathcal{H}_\mathrm{ES}:= {}^{\mathrm{co}H} \! (A\otimes A)$ of left coinvariants of the diagonal coaction  as a Hopf algebroid over $B$. It turns out (Theorem \ref{prop:ES}) that left covariant first order differential calculi on $\mathcal{H}_\mathrm{ES}$ are in bijective correspondence with left $H$-covariant calculi on $A$. 
Secondly, the scalar extension Hopf algebroid \cite{BrzMil:BBAD} is constructed from a braided commutative algebra $A$ in the category of Yetter-Drinfel'd modules of a Hopf algebra $H$ as the smash product $\mathcal{H}_\#:=A\#H$, which is a (right) Hopf algebroid over $A$. Given appropriate ideals $I_A\subseteq A$ and $I_H\subseteq H^+$, one constructs left covariant calculi on $\mathcal{H}_\#$ as the quotient $\mathcal{H}_\#\otimes_A(\mathcal{H}_\#^+/I_A\#I_H)$, see Proposition \ref{prop:smashCalc}. 
Finally, starting from a principal homogeneous space $\pi\colon G\to H$ in the usual Hopf algebraic sense and a braided commutative algebra $A$ in the category ${}_H\mathcal{YD}^G$ of $H$-$G$-Yetter-Drinfel'd modules, one obtains a Hopf algebroid principal homogeneous space from the Hopf algebroid projection $A\#G\to A\#H$, see Proposition \ref{prop:smashhomogeneous}. Left $A\#G$-covariant calculi on $A\otimes G^{\mathrm{co}H}$ are then constructed in Proposition \ref{prop:CalcSmashHom}.
In particular, this procedure can be applied to the Heckenberger-Kolb calculi \cite{HeckKolb} on irreducible quantum flag manifolds to construct covariant calculi on the corresponding Hopf algebroid principal homogeneous spaces.

    Let us conclude by a remark on a different notion of noncommutative differential calculi.
    A {\em Tamarkin-Tsygan (homotopy) calculus} \cite{TamTsy:NCDCHBVAAFC} essentially starts from a mixed complex $(\gO_{-\bullet}, B, b)$, where $\gO_{-\bullet}$ plays the r\^ole of the module of differential forms with differential $B$,
    along with a (homotopy) Gerstenhaber algebra $V^\bullet$ that is seen as the space of polyvector fields. The objective in this context is quite different from what we have in mind for this article and rather focusses on the construction of a contraction $\iota$ and a Lie derivative $\cL$ by which $V^\bullet$ acts on $\gO_{-\bullet}$ in a homotopical way as a graded module resp.\ as a graded Lie algebra. In the context of bialgebroids and Hopf algebroids, this has been dealt with extensively in \cite{KowKra:BVSOEAT}, and from a more general perspective in \cite[\S6.3]{Kow:GABVSOMOO}.

\addtocontents{toc}{\SkipTocEntry}
 \subsection*{Notation}
Throughout this paper, we fix a commutative unital ring $\Bbbk$ and an associative unital $\Bbbk$-algebra $R$. The category of $\Bbbk$-modules is denoted by ${}_\Bbbk\mathcal{M}$, while the categories of left $R$-modules, right $R$-modules, and $R$-bimodules are denoted by ${}_R\mathcal{M}$, $\mathcal{M}_R$, and ${}_R\mathcal{M}_R$, respectively, and likewise for any other ring. We use the unadorned symbol $\otimes$ for the tensor product of $\Bbbk$-modules, while $\otimes_R$ denotes the $R$-balanced tensor product.
Finally, by $\RRe:=R\otimes R^\mathrm{op}$ we mean the enveloping algebra of $R$.
Actions on a module, be it left or right, are most of the time denoted just by juxtaposition.

\addtocontents{toc}{\SkipTocEntry}
\subsection*{Acknowledgements}

We would like to thank Gabriella B\"ohm, Sophie Chemla, Laiachi El Kaoutit, Xiao Han, Ulrich Kr\"ahmer, Paolo Saracco, Peter Schauenburg, and the referee for helpful discussions and precious comments. This work was partially supported by the MIUR Excellence Department Project MatMod@TOV (CUP:E83C23000330006). Also partially supported by the Swedish Research Council under grant no.~2021-06594, while the authors were in residence at the Mittag-Leffler Institute in Djursholm, Sweden, during spring 2025.
The first author is a member of the {\em Gruppo Nazionale per le Strutture Algebriche, Geometriche e le loro
Applicazioni} (GNSAGA-INdAM) and last but not least wants to thank 
the Institut de Math\'ematiques de Jussieu -- Paris
Rive Gauche, where part of this work has been achieved, for hospitality and support.
The second author is supported by the GA\v{C}R PIF 24-11324I.

\section{Hopf-Galois extensions, Hopf algebroids, and Hopf-Galois comodules}

This section unites some well-known material on Hopf-Galois extensions and quantum homogeneous spaces in the classical, {\em i.e.}, Hopf algebraic case with our main objects of interest, bialgebroids resp.\ Hopf algebroids, and presents several first results when applying these concepts to the Ehresmann-Schauenburg bialgebroid, the scalar extension bialgebroid, {\em etc.} We also discuss here the respective notion of Hopf modules along with (what we call) {\em Hopf-Galois} comodules, a corresponding fundamental theorem, and covariant bimodules, all in the realm of bialgebroids resp.\ Hopf algebroids.

\subsection{Hopf-Galois extensions and quantum homogeneous spaces for Hopf algebras}
\label{fuellfederhalter}

Let $H$ be a Hopf algebra over a field $\Bbbk$ with invertible antipode $S$.
Recall the standard con\-struction of a  \textit{Hopf-Galois extension} $B\subseteq A$, where $A$ is a right $H$-comodule algebra with right $H$-coaction $\rho_A\colon A\to A\otimes H$ such that
\begin{equation}
	\label{berlinale}
	\chi\colon A\otimes_BA\to A\otimes H,
	\quad
	a\otimes_Ba' \mapsto aa'_{(0)} \otimes a'_{(1)}
\end{equation}
is a bijection, and where $B:=A^{\mathrm{co}H}$ denotes the subalgebra of right $H$-coinvariant elements.
A Hopf-Galois extension $B\subseteq A$ is called \textit{faithfully flat} if $A$ is a faithfully flat left $B$-module, or, equivalently, a faithfully flat right $B$-module, according to \cite[Thm.~1]{Schneider}. All Hopf-Galois extensions we consider will be assumed faithfully flat.
Its \emph{translation map} \cite{Brz96} is defined by 
\begin{equation}
	\label{berlinale2}
\tau:=\chi^{-1}\circ(\eta_A\otimes H)\colon H \to A \otimes_B A, \quad
h \mapsto h^{\langle 1\rangle} \otimes_B h^{\langle 2\rangle},
\end{equation}
for which we adopt the given Sweedler-type notation and where $\eta_A\colon\Bbbk\to A$ denotes the unit of $A$.
The image of $\tau$ lands in the $B$-bimodule centre of $A \otimes_B A$, that is,
$
b h^{\langle 1\rangle} \otimes_B h^{\langle 2\rangle} = h^{\langle 1\rangle} \otimes_B h^{\langle 2\rangle}b
$ for $b\in B$ and $h\in H$.
With this, let us list a couple of further standard identities mentioned in \cite[Rem.~3.4]{Schneider2}:
\begin{equation}
	\label{allthat}
	\begin{array}{rcl}
		h^{\langle 1\rangle } {h^{\langle 2\rangle }}_{\!\!(0)} \otimes  {h^{\langle 2\rangle }}_{\!\!(1)} = 1_A \otimes h, &&  a_{(0)} {a_{(1)}}^{\!\!\langle 1\rangle } \otimes_B
		{a_{(1)}}^{\!\!\langle 2\rangle } = 1_A \otimes_B a,
		\\[1mm]
		h^{\langle 1\rangle } h^{\langle 2\rangle } = \gve(h)1_A, &&
		(hg)^{\langle 1\rangle } \otimes_B (hg)^{\langle 2\rangle }
		= g^{\langle 1 \rangle }h^{\langle 1\rangle } \otimes_B h^{\langle 2\rangle } g^{\langle 2\rangle },
		\\[1mm]
		h^{\langle 1\rangle } \otimes_B  {h^{\langle 2\rangle }}_{\!\!(0)} \otimes  {h^{\langle 2\rangle }}_{\!\!(1)} & = &  {h_{(1)}}^{\!\!\langle 1\rangle } \otimes_B  {h_{(1)}}^{\!\!\langle 2\rangle } \otimes  h_{(2)} ,
		\\[1mm]
		{h_{(2)}}^{\!\!\langle 1\rangle } \otimes_B  {h_{(2)}}^{\!\!\langle 2\rangle } \otimes  Sh_{(1)} & = &  {h^{\langle 1\rangle }}_{\!\!(0)} \otimes_B  h^{\langle 2\rangle } \otimes  {h^{\langle 1 \rangle}}_{\!\!(1)} ,
		\\[1mm]
		{h_{(1)}}^{\!\!\langle 1\rangle } \otimes_B  {h_{(1)}}^{\!\!\langle 2\rangle } {h_{(2)}}^{\!\!\langle 1\rangle } \otimes_B   {h_{(2)}}^{\!\!\langle 2\rangle } & = &  {h^{\langle 1\rangle }} \otimes_B 1 \otimes_B  {h^{\langle 2 \rangle}},
	\end{array}
\end{equation}
for all $h,g\in H$ and $a\in A$.

\begin{remark}
  \label{rem:S-1Galois}
  Since the antipode of $H$ is invertible,
in case of a Hopf-Galois extension as above,
  a {\em second} Galois map
	\begin{equation}
          \label{Galois2}
		\tilde\chi\colon A\otimes_BA\to A\otimes H,\quad a\otimes_Ba' \mapsto a_{(0)}a'\otimes a_{(1)}
	\end{equation}
	is invertible (see, {\em e.g.}, \cite[p.~372]{MontSurvey}), and it is rather {\em this} map that we are going to generalise to (left) bialgebroids later on.
        Its inverse can be described explicitly in terms of the translation map $\tau$ of $\chi$ and the inverse of the antipode by
	\begin{equation}
          \label{chi'inv}
		\tilde\chi^{-1}\colon A\otimes H\to A\otimes_BA,\quad
		a\otimes h \mapsto \tau'(h)a : = S^{-1}(h)^{\langle 1\rangle}\otimes_BS^{-1}(h)^{\langle 2\rangle}a,
	\end{equation}
	where $\tau'=\tau\circ S^{-1}\colon H\to A\otimes_BA$ denotes the translation map corresponding to $\tilde\chi$. In fact, \eqref{chi'inv} is the inverse of $\tilde\chi$ as can be easily seen as follows:
one has
	\begin{eqnarray*}
		\tilde\chi^{-1}(\tilde\chi(a\otimes_Ba'))
		&=&
                \tilde\chi^{-1}(a_{(0)}a'\otimes a_{(1)})
                %\\
		%&
                =
                S^{-1}(a_{(1)})^{\langle 1\rangle}\otimes_BS^{-1}(a_{(1)})^{\langle 2\rangle}a_{(0)}a'
                \\
                &\overset{(*)}{=}&
                S^{-1}(a_{(1)})^{\langle 1\rangle}S^{-1}(a_{(1)})^{\langle 2\rangle}a_{(0)} \otimes_B a' 
		\overset{\eqref{allthat}}{=} \, a\otimes_Ba'
	\end{eqnarray*}
	for all $a,a'\in A$ and $h\in H$, where in $(*)$ we used that $S^{-1}(a_{(1)})^{\langle 1\rangle}\otimes_BS^{-1}(a_{(1)})^{\langle 2\rangle}a_{(0)}\in A\otimes_BB$. The latter holds since $A$ is flat as a left $B$-module and
	\begin{align*}
		S^{-1}(a_{(1)})^{\langle 1\rangle}\otimes_B\rho_A(S^{-1}(a_{(1)})^{\langle 2\rangle}a_{(0)})
		&=S^{-1}(a_{(2)})^{\langle 1\rangle}\otimes_BS^{-1}(a_{(2)})^{\langle 2\rangle}{}_{(0)}a_{(0)}\otimes S^{-1}(a_{(2)})^{\langle 2\rangle}{}_{(1)}a_{(1)}
                \\
		&=S^{-1}(a_{(2)})_{(1)}{}^{\!\!\langle 1\rangle}\otimes_BS^{-1}(a_{(2)})_{(1)}{}^{\!\!\langle 2\rangle}a_{(0)}\otimes S^{-1}(a_{(2)})_{(2)}a_{(1)}
                \\
		&=S^{-1}(a_{(3)})^{\langle 1\rangle}\otimes_BS^{-1}(a_{(3)})^{\langle 2\rangle}a_{(0)}\otimes S^{-1}(a_{(2)})a_{(1)}
                \\
		&=S^{-1}(a_{(1)})^{\langle 1\rangle}\otimes_BS^{-1}(a_{(1)})^{\langle 2\rangle}a_{(0)}\otimes 1,
	\end{align*}
	where we employed Eqs.~\eqref{allthat} and
the customary antipode properties.
Even simpler is to show that $\tilde\chi \circ \tilde\chi^{-1} = \id$. 
\end{remark}

If one starts from a Hopf algebra surjection $\pi \colon G \twoheadrightarrow H$, setting $A := G$ with right $H$-coaction given by $\rho_G \colon G \to G \otimes H, \ g \mapsto g_{(0)} \otimes g_{(1)} := g_{(1)} \otimes \pi(g_{(2)})$, we call $B:=A^{\mathrm{co}H}$ a \emph{principal (quantum)} \emph{homogeneous space} if $B\subseteq A$ is a faithfully flat Hopf-Galois extension.

In \S\ref{schreibtinte}, we are going to generalise this notion to the realm of (right) Hopf algebroids. 

\subsection{Bialgebroids and Hopf algebroids}
\label{heiszheiszheisz}
Mainly to fix the notation, let us
recall that a \textit{left bialgebroid} \cite{Tak:GOAOAA} is a sextuple $(\mathcal{H},R, s,t,\Delta,\varepsilon)$, where
\begin{enumerate}
\compactlist{99}
\item
the triple  $(\mathcal{H},s,t)$ is an $\RRe$-ring.
In particular, $\mathcal{H}$ has two commuting $R$-bimodule structures
$$
r\lact h\ract r'
:=s(r)t(r')h,\qquad\qquad
r\blact h\bract  r'
:=ht(r)s(r'); 
$$
\item
the triple  $(\cH,\Delta,\varepsilon)$ is an $R$-coring with respect to the actions $\lact, \ract$ and coproduct $\Delta\colon\mathcal{H}\to\mathcal{H}_\ract\times_R\due {\mathcal{H}} \lact {}$, where
  \begin{equation}
  \label{takeuchiyippieh}
  \mathcal{H}_\ract\times_R\due {\mathcal{H}} \lact {}=
  \big\{\Sum\nolimits_i h^i\otimes_Rh_i\in\mathcal{H}_\ract\otimes_R\due {\mathcal{H}} \lact {} \mid 
r\blact h^i\otimes_R h_i=h^i\otimes_Rh_i\bract r,\ \forall r\in R \big\}
\end{equation}
is the {\em Takeuchi} subspace, and in the quotient $\mathcal{H}_\ract \otimes_R \due \cH \lact {}$ we identify $h\ract r\otimes h'$ with $h\otimes r\lact h'$.
In particular, $\gD$ is $R$-{\em tetralinear}, that is, 
$$
\Delta(r\lact h\ract r')
=r\lact h_{(1)}\otimes_R h_{(2)}\ract r',\qquad\qquad
\Delta(r\blact h\bract r')
=h_{(1)}\bract  r'\otimes_Rr\blact h_{(2)},
$$
where we wrote $\Delta h = h_{(1)}\otimes_Rh_{(2)}$;
\item
the counit  $\varepsilon\colon\mathcal{H}\to R$ is an $R$-bilinear generalised left character, by which we mean
\begin{equation}
  \label{counityippieh}
\varepsilon(r\lact h\ract r')
=r\varepsilon(h)r',\qquad\qquad
\varepsilon(hh') =
\varepsilon(h\bract  \varepsilon(h'))
=\varepsilon(\varepsilon(h')\blact h).
\end{equation}
\end{enumerate}
We will often denote a left bialgebroid compactly by $(\cH, R)$ or even $\cH$ only.

A {\em module}, be it left or right, over a left bialgebroid $(\cH, R)$ is simply a module over the underlying ring $\cH$. The category ${}_\cH \cM$ of left modules over a left bialgebroid is monoidal with unit $R$, while the category $\cM_\cH$ of right modules, in striking contrast to the bialgebra case, is not; nevertheless, one still has two forgetful functors ${}_\cH \cM \to {}_\RRe \cM$ resp.\ $\cM_\cH \to {}_\RRe \cM$ with respect to which we write
\begin{equation}
  \label{soedermalm}
r \lact m \ract r' := s(r)t(r')m, \qquad r \blact n \bract r' := nt(r)s(r')
  \end{equation}
for $r,r'\in R$, $m\in M$, and $n\in N$, where $M$ is a left $\mathcal{H}$-module and $N$ is a right $\mathcal{H}$-module.
\noindent The notion of {\em comodules} over a left bialgebroid will be discussed at the beginning of \S\ref{mittagleffler}.

We call a left bialgebroid $(\mathcal{H},R,s,t,\Delta,\varepsilon)$ a \textit{right Hopf algebroid} if the canonical map
\begin{equation}
  \label{eq:HG}
\ga\colon	\mathcal{H}_\bract  \otimes_R{}_\lact  \mathcal{H}\to\mathcal{H}_\ract  \otimes_R{}_\lact  \mathcal{H},\quad
	h\otimes_Rh'\mapsto h_{(1)}h'\otimes_Rh_{(2)}
\end{equation}
is bijective. In this case, we adopt yet another Sweedler-type notation and write 
\begin{equation}
  \label{toujours}
\ga^{-1}(1 \otimes_R h) =: h_{\smap } \otimes_R h_{\smam },
\end{equation}
for the {\em translation map} such that, in total,
\begin{equation*}
  \ga^{-1}\colon	
  \mathcal{H}_\ract  \otimes_R{}_\lact  \mathcal{H}
  \to
  \mathcal{H}_\bract  \otimes_R{}_\lact  \mathcal{H}, \quad
	h\otimes_Rh'\mapsto h'_{\smap } \otimes_R h'_{\smam }h.
\end{equation*}
A left bialgebroid is, on the other hand, called 
a \textit{left Hopf algebroid} if, in turn, 
$$
\beta \colon
{}_\blact \mathcal{H}\otimes_{R^\mathrm{op}}\mathcal{H}_\ract  \to \mathcal{H}_\ract  \otimes_R{}_\lact \mathcal{H},\quad
h\otimes_{R^\mathrm{op}}h'\mapsto h_{(1)}\otimes_Rh_{(2)}h'
$$
is a bijective map. For some categorical implications of this definition, see \cite{BruLacVir:HMOMC, Schau:DADOQGHA}. Since in this article we are mainly concerned with right Hopf algebroids, for simplicity we sometimes refer to them as \textit{Hopf algebroids} only.
If a left bialgebroid $(\mathcal{H},R, s, t, \gD, \gve)$ is right Hopf, then one verifies that for the translation map the identities
\begin{eqnarray}
	\label{Tch1}
	h_{\smap } \otimes_R  h_{\smam } & \in
	& \mathcal{H}_\bract\times_R\due {\mathcal{H}} \lact {},  \\
	\label{Tch2}
	h_{\smap (1)} h_{\smam } \otimes_R h_{\smap (2)}  &=& 1 \otimes_R h \quad \in \mathcal{H}_{\ract} \! \otimes_R \! {}_\lact \mathcal{H},  \\
	\label{Tch3}
      h_{(2)\smap } \otimes_R  h_{(2)\smam }h_{(1)}   &=& h \otimes_R 1
	%h_{(2)\smam }h_{(1)} \otimes_R h_{(2)\smap }  &=& 1 \otimes_R h
        \quad \in {\mathcal{H}}_{\bract} \!
	\otimes_R \! \due {\mathcal{H}} \lact {},  \\
	\label{Tch4}
	h_{\smap (1)} \otimes_R h_{\smam } \otimes_R h_{\smap (2)} &=& h_{(1)\smap } \otimes_R
	h_{(1)\smam } \otimes_R  h_{(2)},  \\
	\label{Tch5}
	h_{\smap \smap } \otimes_R  h_{\smap \smam } \otimes_R h_{\smam } &=&
	h_{\smap } \otimes_R h_{\smam (1)} \otimes_R h_{\smam (2)},  \\
	\label{Tch6}
	(hg)_{\smap } \otimes_R (hg)_{\smam } &=& h_{\smap }g_{\smap }
	\otimes_R g_{\smam }h_{\smam },  \\
	\label{Tch7}
	h_{\smap }h_{\smam } &=& t \varepsilon (h),  \\
	\label{Tch8}
	h_{\smap } \bract \varepsilon(h_{\smam })  &=&  h,  \\
	\label{Tch9}
	(s (r) t (r'))_{\smap } \otimes_R (s (r) t (r') )_{\smam }
	&=& t(r') \otimes_R t(r),
\end{eqnarray}
hold for all $h,g\in\mathcal{H}$ and $r,r'\in R$,
where in  \eqref{Tch1} we denoted the Takeuchi subspace
\begin{equation*}
  \label{petrarca2}
	\mathcal{H}_\bract\times_R\due {\mathcal{H}} \lact {}    \coloneqq 
	\big\{ {\textstyle \sum_i} h_i \otimes  g_i \in {\mathcal{H}}_{\bract}  \otimes_R \!  \due {\mathcal{H}} \lact {} \mid \, {\textstyle \sum_i} r \lact h_i \otimes g_i = {\textstyle \sum_i} h_i \otimes g_i \bract r,  \ \forall r \in R  \big\}.
\end{equation*}
See, for example, \cite[Prop.~4.2]{BoeSzl:HAWBAAIAD} for further information in this direction on bialgebroids and Hopf algebroids.

\subsection{The Ehresmann-Schauenburg Hopf algebroid}
\label{gewittrigheute}
Given a faithfully flat Hopf-Galois extension $B:=A^{\mathrm{co}H}\subseteq A$ in the sense of \S\ref{fuellfederhalter}, one can construct a {\em right} Hopf algebroid over $B$ on the space $\mathcal{H}:=(A\otimes A)^{\mathrm{co}H}$ of right $H$-coinvariants under the diagonal $H$-coaction if the antipode $S$ of $H$ is involutive. Instead of assuming $S^2=\mathrm{id}$, one can consider a version of the Ehresmann-Schauenburg bialgebroid based on \textit{left} coinvariants. We show in Remark \ref{leftpeft} that this becomes a right Hopf algebroid without the need of additional assumptions.

The underlying left bialgebroid structure was given in \cite[Thm.~6.3]{Schau98},
and known as the \textit{Ehresmann-Schauenburg bialgebroid}, see, for example, \cite[\S34.14]{BrzWis:CAC},
while in \cite[\S4.2]{HanMajid2} it is shown that this is a {\em left} Hopf algebroid in the sense mentioned above. Here, we briefly recall its construction and show afterwards that it is also a {\em right} Hopf algebroid if $S^2 = \id$. More precisely,
$$
\mathcal{H} = (A\otimes A)^{\mathrm{co}H} := \{a\otimes a'\in A\otimes A~|~a_{(0)}\otimes a'_{(0)}\otimes a_{(1)}a'_{(1)}=a\otimes a'\otimes 1\}
$$ 
becomes a left bialgebroid over $B$ with source and target maps
\begin{equation*}
	s,t\colon B\to\mathcal{H},\qquad s(b):=b\otimes 1,\qquad t(b):=1\otimes b,
\end{equation*}
with algebra structure
\begin{equation*}
	(a\otimes a')(c\otimes c'):=ac\otimes c'a',\qquad
	1_\mathcal{H}:=1\otimes 1,
\end{equation*}
and $B$-coring structure
\begin{equation}
	\label{blauetinte}
	\begin{split}
	  \Delta\colon\mathcal{H}&\to\mathcal{H}\otimes_B\mathcal{H},
          \\
		\varepsilon\colon\mathcal{H}&\to B,
	\end{split}\quad
	\begin{split}
		a\otimes a'&\mapsto a_{(0)}\otimes\tau(a_{(1)})\otimes a'
		=(a_{(0)}\otimes a_{(1)}{}^{\!\!\langle 1\rangle})\otimes_B(a_{(1)}{}^{\!\!\langle 2\rangle}\otimes a'),
                \\
		a\otimes a'&\mapsto aa'.
	\end{split}
\end{equation}
This left bialgebroid can be promoted to a right Hopf algebroid in the following way:

\begin{lemma}
	\label{quarkspeise}
If $S^2 = \id$ holds for the antipode of the Hopf algebra $H$, then the Ehresmann-Schauenburg bialgebroid $\cH$ as defined above is a right Hopf algebroid, with translation map \eqref{toujours}  given by   
	\begin{equation}
		\label{internetgehtnich}
		(a \otimes a')_\smap \otimes_B (a \otimes a')_\smam
		:= (a'_{(1)}{}^{\!\!\langle 1 \rangle} \otimes a'_{(0)} ) \otimes_B (a'_{(1)}{}^{\!\!\langle 2 \rangle} \otimes a)  
	\end{equation}
	for an element $a \otimes a' \in \cH$, as a relation in $\cH_\bract \otimes_B \due \cH \lact {}$.
\end{lemma}

\reversemarginpar

\begin{proof}
Let us show first that, in spite of its appearance, the map \eqref{internetgehtnich} {\em is} well-defined, indeed, in the sense that
      $
      (a \otimes a')_\smap \otimes_B (a \otimes a')_\smam \in (A\otimes A)^{\mathrm{co}H} \otimes_B (A\otimes A)^{\mathrm{co}H},
$      
      that is, that the tensor factors are still coinvariant. To this end, let us first prove the equivalence, in case $S^2 = \id$,
      \begin{equation}
        \label{flipped1}
a_{(0)}\otimes a'_{(0)}\otimes a_{(1)}a'_{(1)}=a\otimes a'\otimes 1 \ \iff \ a_{(0)}\otimes a'_{(0)}\otimes a'_{(1)}a_{(1)}=a\otimes a'\otimes 1,
      \end{equation}
      that is, the $H$-components in the codiagonal coaction can be flipped.
      Said equivalently,
      $$
a \otimes a' \in (A\otimes A)^{\mathrm{co}H} \ \iff \ a' \otimes a \in (A\otimes A)^{\mathrm{co}H}
      $$
if $S^2 = \id$.
This results as follows: applying $S$ to the last factor on both sides of the identity on the left hand side in \eqref{flipped1} and subsequently coacting on the respective second tensor factor yields
            $
      a_{(0)}  \otimes a'_{(0)} \otimes a'_{(1)} \otimes S(a'_{(2)}) S(a_{(1)}) = a \otimes a'_{(0)} \otimes a'_{(1)} \otimes 1,
      $
and multiplying the respective last two tensor factors results into:
      \begin{equation}
        \label{flipped2}
a_{(0)} \otimes a' \otimes S(a_{(1)})  = a \otimes a'_{(0)} \otimes a'_{(1)},
      \end{equation}
      as already observed in \cite[Eq.~(3.1)]{Schau98}. Therefore,
      \begin{eqnarray*}
        a \otimes a' \in (A\otimes A)^{\mathrm{co}H}
        &
\stackrel{\eqref{flipped2}}{\iff} 
        &
a_{(0)} \otimes a' \otimes S(a_{(1)})  = a \otimes a'_{(0)} \otimes a'_{(1)}
\\[.1mm]
        &
\stackrel{\scriptscriptstyle S^2 = \id}{\iff} 
        &
a_{(0)} \otimes a' \otimes a_{(1)}  = a \otimes a'_{(0)} \otimes S(a'_{(1)})
\\
        &
\stackrel{}{\iff} 
        &
a'_{(0)} \otimes a \otimes S(a'_{(1)})
=
a' \otimes a_{(0)} \otimes a_{(1)} 
\\
        &
\stackrel{\eqref{flipped2}}{\iff} 
        &
 a' \otimes a \in (A\otimes A)^{\mathrm{co}H},
      \end{eqnarray*}
      as desired. Denoting this tensor flip by $\gs$, observe that $\gs((a \otimes a') \ract b) = \gs(a \otimes a'b) = a'b \otimes a = (a' \otimes a) \bract b$ for all $a \otimes a' \in (A\otimes A)^{\mathrm{co}H}$ and $b \in B$, that is, $\gs$ is right $B$-linear in the indicated sense.     
      Using this, the translation map in 
      \eqref{internetgehtnich} can then be written as the composition
      $$
      (A\otimes A)^{\mathrm{co}H} \xrightarrow{\ \gs \ } (A\otimes A)^{\mathrm{co}H}
      \xrightarrow{\ \gD \ } (A\otimes A)^{\mathrm{co}H} \otimes_B (A\otimes A)^{\mathrm{co}H} \xrightarrow{\, \gs \otimes_B  (A\otimes A)^{\mathrm{co}H}} (A\otimes A)^{\mathrm{co}H} \otimes_B (A\otimes A)^{\mathrm{co}H}  
      $$
      that is to say,
$$
(a \otimes a')_\smap \otimes_B (a \otimes a')_\smam \in (A\otimes A)^{\mathrm{co}H} \otimes_B (A\otimes A)^{\mathrm{co}H}
        $$
        if
        $a \otimes a' \in (A\otimes A)^{\mathrm{co}H}$,
        as claimed.
        Also observe that $s(b)_\smap \otimes_B s(b)_\smam = (1 \otimes 1) \otimes_B (1 \otimes b) = (1 \otimes 1) \otimes_B t(b)$ for $b \in B$ yields the well-definedness over $\otimes_B$ of (what is going to be) the inverse
$$
\ga^{-1} \colon        \cH_\ract \otimes_B \due \cH \lact {} \to
\cH_\bract \otimes_B \due \cH \lact {}, \quad
        (c \otimes c') \otimes_B (a \otimes a') \mapsto (a \otimes a')_\smap \otimes_B (a \otimes a')_\smam (c \otimes c')
       $$
of the Hopf-Galois map  as in \eqref{eq:HG}.
   Next, by using the relations for $\tau$ listed in Eqs.~\eqref{allthat}, let us verify the identity \eqref{Tch2}. Indeed,
	\begin{eqnarray*}
			(a \otimes a')_{\smap(1)} (a \otimes a')_\smam \otimes_B (a \otimes a')_{\smap(2)}
			&\overset{\eqref{internetgehtnich}}{=}&
			(a'_{(1)}{}^{\!\!\langle 1 \rangle} \otimes a'_{(0)} )_{(1)} (a'_{(1)}{}^{\!\!\langle 2 \rangle} \otimes a)  \otimes_B (a'_{(1)}{}^{\!\!\langle 1 \rangle} \otimes a'_{(0)} )_{(2)}
			\\
&\overset{\eqref{blauetinte}}{=}&
			(a'_{(1)}{}^{\!\!\langle 1 \rangle}{}_{(0)} \otimes a'_{(1)}{}^{\!\!\langle 1 \rangle}{}_{(1)}{}^{\!\!\langle 1 \rangle} ) (a'_{(1)}{}^{\!\!\langle 2 \rangle} \otimes a)  \otimes_B (  a'_{(1)}{}^{\!\!\langle 1 \rangle}{}_{(1)}{}^{\!\!\langle 2 \rangle} \otimes a'_{(0)} )
			\\
&\overset{\eqref{allthat}}{=}&
			\big(a'_{(2)}{}^{\!\!\langle 1 \rangle} \otimes S(a'_{(1)}){}^{\langle 1 \rangle}\big) \big(a'_{(2)}{}^{\!\!\langle 2 \rangle} \otimes a\big)  \otimes_B \big(  S(a'_{(1)}){}^{\langle 2 \rangle} \otimes a'_{(0)} \big)
			\\
&\overset{\eqref{allthat}}{=}&
			\big(1 \otimes a S(a'_{(1)}){}^{\langle 1 \rangle}\big) \otimes_B \big(  S(a'_{(1)}){}^{\langle 2 \rangle} \otimes a'_{(0)} \big)
			\\
                        &\overset{\eqref{flipped2}}{=}&
                        \big(1 \otimes a_{(0)} a_{(1)}{}^{\langle 1 \rangle}\big) \otimes_B \big(  a_{(1)}{}^{\langle 2 \rangle} \otimes a' \big)
	                \\
                        &\overset{\eqref{allthat}}{=}&
			(1 \otimes 1) \otimes_B (a \otimes a' )
			\end{eqnarray*}
	in $\cH_\ract \otimes_B \due \cH \lact {}$, using also $S^2 = \id$ in the penultimate step, and which is \eqref{Tch2}.
    Eq.~\eqref{Tch3} can be proven along the same lines:
    \begin{eqnarray*}        
			& &
			(a \otimes a')_{(2)\smap} \otimes_B (a \otimes a')_{(2)\smam} (a \otimes a')_{(1)}
			\\
              &\overset{\eqref{blauetinte}}{=}&
			(a_{(1)}{}^{\!\!\langle 2 \rangle} \otimes a')_{\smap} \otimes_B (a_{(1)}{}^{\!\!\langle 2 \rangle} \otimes a')_{\smam} (a_{(0)} \otimes a_{(1)}{}^{\!\!\langle 1 \rangle} )
			\\
                   &\overset{\eqref{internetgehtnich}}{=}&
			(a'_{(1)}{}^{\!\!\langle 1 \rangle} \otimes a'_{(0)}) \otimes_B (a'_{(1)}{}^{\!\!\langle 2 \rangle} \otimes a_{(1)}{}^{\!\!\langle 2 \rangle} ) (a_{(0)} \otimes a_{(1)}{}^{\!\!\langle 1 \rangle} )
			\\
                   &\overset{\eqref{allthat}}{=}&
			(a'_{(1)}{}^{\!\!\langle 1 \rangle} \otimes a'_{(0)}) \otimes_B (a'_{(1)}{}^{\!\!\langle 2 \rangle} a \otimes 1) 
			\\
                   &\overset{\eqref{flipped1}}{=}&
			(a \otimes a') \otimes_B (1 \otimes 1) 
     \end{eqnarray*}
	in $\cH_\bract \otimes_B \due \cH \lact {}$.
        Here, the last step is justified by the equivalence \eqref{flipped1} that holds for coinvariant elements $a \otimes a'$, by which one shows \begin{eqnarray*}        
    (a'_{(1)}{}^{\!\!\langle 1 \rangle} \otimes a'_{(0)}) \otimes_B \rho_A(a'_{(1)}{}^{\!\!\langle 2 \rangle} a)
   & = &
    (a'_{(1)}{}^{\!\!\langle 1 \rangle} \otimes a'_{(0)}) \otimes_B (a'_{(1)}{}^{\!\!\langle 2 \rangle}\!{}_{(0)} a_{(0)} \otimes a'_{(1)}{}^{\!\!\langle 2 \rangle}\!{}_{(1)} a_{(1)})
    \\
   & \overset{\eqref{allthat}}{=} &
    (a'_{(1)}{}^{\!\!\langle 1 \rangle} \otimes a'_{(0)}) \otimes_B (a'_{(1)}{}^{\!\!\langle 2 \rangle} a_{(0)} \otimes a'_{(2)} a_{(1)})
\\
    & \overset{\eqref{flipped1}}{=} &
    (a'_{(1)}{}^{\!\!\langle 1 \rangle} \otimes a'_{(0)}) \otimes_B (a'_{(1)}{}^{\!\!\langle 2 \rangle} a \otimes 1),
     \end{eqnarray*}
and hence $(a'_{(1)}{}^{\!\!\langle 1 \rangle} \otimes a'_{(0)}) \otimes_B (a'_{(1)}{}^{\!\!\langle 2 \rangle} a \otimes 1) \in \cH \otimes_B s(B)$.
Summing up, this shows the invertibility of the Hopf-Galois map \eqref{eq:HG} by means of \eqref{internetgehtnich}.
\end{proof}

\begin{remark}
	In \cite[Thm.~4.4]{HanMajid2}, it was shown that $\cH$ is also a left Hopf algebroid; hence, by the lemma above, it is left and right Hopf if $S^2 = \id$. In this specific situation, this is equivalent to being a {\em full} Hopf algebroid (see \cite{BoeSzl:HAWBAAIAD} for details of this notion) in the sense that it admits (sort of) an antipode, which here is simply the tensor flip. This applies, for example, to the case of the noncommutative Hopf fibration \cite[\S5.1]{HanLandi}, which we revisit in Example \ref{ex:HopfFibration} right below.
\end{remark}

    \begin{remark}
      \label{leftpeft}
Although the construction in Lemma \ref{quarkspeise} is sufficient for our examples, one can also renounce on the condition $S^2 = \id$ by starting from a mirrored version of the Ehresmann-Schauenburg bialgebroid defined by means of
  {\em left} coinvariants. Let us indicate the crucial steps, skipping a few details and proofs. 
 Here, $H$ is a Hopf algebra (not necessarily with involutive antipode), $A$ is a left $H$-comodule algebra, $B := {}^{\co H}\! A$ is, accordingly, the subalgebra of left $H$-coinvariants,   and one asks 
 $$
\gamma \colon  A \otimes_B A \to H \otimes A, \quad a \otimes a' \mapsto a_{(-1)} \otimes a_{(0)} a',
 $$
   to be bijective, which leads to a translation map 
 \begin{equation}
	\label{berlinale3}
\gamma^{-1} \circ (H \otimes \eta_A) \colon H \to A \otimes_B A, \quad h \mapsto h^{[1]} \otimes_B h^{[2]}.
   \end{equation}
Similarly to \eqref{allthat}, one can prove that, among others, the identities
\begin{equation}
	\label{allthat2}
	\begin{array}{rcl}
	  {h^{[ 1] }}_{\! (-1)}  \otimes  {h^{[ 1] }}_{\! (0)} {h^{[ 2] }} = h \otimes 1_A,
          &&
{a_{(-1)}}^{\! [ 1] } \otimes_B   {a_{(-1)}}^{\! [ 2 ] } a_{(0)} 
		 = a \otimes_B 1_A,
		\\[1mm]
		h^{[ 1] } h^{[ 2] } = \gve(h)1_A, &&
		(hg)^{[ 1] } \otimes_B (hg)^{[ 2] }
		= h^{[ 1 ] }g^{[ 1] } \otimes_B g^{[ 2] } h^{[ 2] },
		\\[1mm]
 {h^{[ 1] }}_{\!(-1)} \otimes  {h^{[ 1] }}_{\!(0)} \otimes_B h^{[2]}
                  & = &
h_{(1)}  \otimes
 {h_{(2)}}^{\! [ 1] } \otimes_B  {h_{(2)}}^{\! [ 2] },
		\\[1mm]
	        Sh_{(2)}
                \otimes
	  {h_{(1)}}^{\! [ 1] } \otimes_B  {h_{(1)}}^{\! [ 2] }
          & = &
          {h^{[ 2] }}_{\! (-1)} \otimes h^{[ 1] } \otimes_B  {h^{[ 2 ]}}_{\! (0)} ,
	\end{array}
\end{equation}
hold for all $h,g\in H$ and $a\in A$. If, in addition, $A$ is faithfully flat as a right $B$-module, this allows to define an Ehresmann-Schauenburg bialgebroid of {\em left} coinvariants, which is a left bialgebroid over $B = {}^{\co H}\! A$: define
$$
\cH = {}^{\mathrm{co}H} \! (A\otimes A)
:= \{a\otimes a'\in A\otimes A \mid a_{(-1)}a'_{(-1)} \otimes  a_{(0)}\otimes a'_{(0)} = 1 \otimes a\otimes a' \},
$$
with source and target map 
$s(b):=b\otimes 1$ resp.\ $t(b):=1\otimes b$,
multiplication
$
	(a\otimes a')(c\otimes c') := ac\otimes c'a'
$,
unit
$1_\mathcal{H}:=1\otimes 1$,
and finally $B$-coring structure
\begin{equation}
	\label{blauetinte2}
	\begin{split}
	  \Delta\colon\mathcal{H}&\to\mathcal{H}\otimes_B\mathcal{H},
          \\
		\varepsilon\colon\mathcal{H}&\to B,
	\end{split}
        \quad
	\begin{split}
		a\otimes a'&\mapsto 
                \big (a \otimes a'_{(-1)}{}^{\! [1]} \big) \otimes_B
                \big( a'_{(-1)}{}^{\! [2]} \otimes a'_{(0)} \big),
                \\
		a\otimes a'&\mapsto aa'.
	\end{split}
\end{equation}
This version of an Ehresmann-Schauenburg bialgebroid can, without any further assumptions, be turned into a right Hopf algebroid by defining
\begin{equation}
  \label{internetgehtnich2}
(a\otimes a')_\smap \otimes_B  (a\otimes a')_\smam
:= 
                \big (a_{(-1)}{}^{\! [1]} \otimes a'\big) \otimes_B
                \big(a_{(-1)}{}^{\! [2]} \otimes a_{(0)} \big).
                \end{equation}
                We skip the verification of this fact since the necessary computations are similar to those in the proof of Lemma \ref{quarkspeise}. However, observe that as a crucial step for verifying \eqref{Tch3}, in the spirit of \eqref{flipped2} one derives again an equation of the type
                \begin{equation}
                  \label{again}
  S(a'_{(-1)}) \otimes a \otimes a'_{(0)} = a_{(-1)} \otimes a_{(0)} \otimes a'
\end{equation}
for $a \otimes a' \in {}^{\co H}\!(A \otimes A)$, from which one obtains
                $
S(a'_{(-1)})^{[1]} \otimes S(a'_{(-1)})^{[2]} a \otimes a'_{(0)} = a \otimes 1 \otimes a'
                $
                with the help of Eqs.~\eqref{allthat2}, which is needed during the verification.

                Interestingly enough,
                in case $S^2 = \id$, one has an isomorphism
                $$
\Xi \colon {}^{\co H}\! (A \otimes A) \to (A \otimes A)^{\co H}, \quad a \otimes a' \mapsto a \otimes a',
                $$
of left bialgebroids resp.\ right Hopf algebroids, which connects this construction to the situation in Lemma \ref{quarkspeise}. More in detail, this works as follows: it is well-known that defining
$m_{(0)} \otimes m_{(1)} := m_{(0)} \otimes S^{-1}(m_{(-1)})$ for a left $H$-comodule $M$
yields a monoidal functor $F \colon {}^H \! \cM \to \cM^{H^\op}$, where by the latter we mean the category of right $H$-comodules but with monoidal structure given by $(m \otimes n)_{(0)} \otimes (m \otimes n)_{(1)} := m_{(0)} \otimes n_{(0)} \otimes n_{(1)} m_{(1)}$ for two right $H$-comodules $M, N$.
Hence, the functor $F$ induces an isomorphism ${}^{\co H}\! (A \otimes A) \cong (A \otimes A)^{\co H^\op}$, while in case $S^2 = \id$ the identity \eqref{flipped1} yields the isomorphism $(A \otimes A)^{\co H^\op} \cong (A \otimes A)^{\co H}$.
Now, the compatibility of $\Xi$ above with the bialgebroid structures \eqref{blauetinte} resp.\ \eqref{blauetinte2} is obvious for the source, target, and counit maps. As for the coproducts, let us first state that for $h \in H$,
\begin{equation}
  \label{glaciere}
h^{\langle 1 \rangle} \otimes_B h^{\langle 2 \rangle} = h^{[1]} \otimes_B h^{[2]} 
\end{equation}
with respect to the two translation maps from \eqref{berlinale2} and \eqref{berlinale3} if the right $H$-coaction arises from the left one via $S^{-1}$ as above. For this, by definition it is enough to verify the two
 equations in the first line of \eqref{allthat}; for example,
 $$
 h^{\langle 1\rangle } {h^{\langle 2\rangle }}_{\!\!(0)} \otimes  {h^{\langle 2\rangle }}_{\!\!(1)}
 \, = \,
 h^{[ 1] } {h^{[ 2 ]}}_{\!(0)} \otimes  S^{-1}({h^{[ 2] }}_{\!(-1)})
 \, \stackrel{\eqref{allthat2}}{=} \,
    {h_{(1)}}^{\![ 1] }
 {h_{(1)}}^{\![ 2] } \otimes  S\big(S^{-1}(h_{(2)})\big) 
 \, \stackrel{\eqref{allthat2}}{=} \,
 1_A  \otimes  h, 
 $$
 which is the first one, and the second one follows analogously.
 Hence, suppressing $\Xi$ in notation, one has, starting from the coproduct \eqref{blauetinte2}, 
 \begin{equation*}
   \begin{split}
 \big (a \otimes a'_{(-1)}{}^{\! [1]} \big) \otimes_B
 \big( a'_{(-1)}{}^{\! [2]} \otimes a'_{(0)} \big)
 \, & \stackrel{\eqref{again}}{=} \,
 \big (a_{(0)} \otimes S^{-1}(a_{(-1)})^{ [1]} \big) \otimes_B
 \big(S^{-1}(a_{(-1)})^{ [2]} \otimes a' \big)
\\
\, & \stackrel{\eqref{glaciere}}{=} \,
 \big (a_{(0)} \otimes S^{-1}(a_{(-1)})^{ \langle 1 \rangle} \big) \otimes_B
 \big(S^{-1}(a_{(-1)})^{\langle 2 \rangle } \otimes a' \big)
\\
\, & \stackrel{\phantom{\eqref{glaciere}}}{=} \,
 \big (a_{(0)} \otimes {a_{(1)}}^{\!\! \langle 1 \rangle} \big) \otimes_B
 \big( {a_{(1)}}^{\!\! \langle 2 \rangle } \otimes a' \big),
   \end{split}
 \end{equation*}
 which is the coproduct given in \eqref{blauetinte}.
This proves the compatibility of $\Xi$ with the two coproducts and hence that it is a left bialgebroid isomorphism. A similarly simple argument (using $S^2 = \id$ again) shows that $\Xi$ also intertwines the two translation maps \eqref{internetgehtnich} and 
\eqref{internetgehtnich2}, which finally proves that it is an isomorphism of right Hopf algebroids.
\end{remark}

\begin{example}[The Ehresmann-Schauenburg Hopf algebroid of the Hopf fibration]
  \label{ex:HopfFibration}
	Fix a non-zero complex number $q\in\mathbb{C}$ which is not a root of unity, and let $A :=\mathcal{O}_q(\mathrm{SL}_2(\mathbb{C}))$ be the free (associative unital) algebra generated by $a,b,c,d$ modulo the Manin relations
	\begin{equation}
          \label{Manin}
		\begin{split}
			ab&=qba,\\
			bc&=cb,
		\end{split}\qquad
		\begin{split}
			bd&=qdb,\\
			ac&=qca,
		\end{split}\qquad
		\begin{split}
			cd&=qdc,\\
			ad-da&=(q-q^{-1})bc,
		\end{split}
	\end{equation}
	and the quantum determinant relation
	\begin{equation}
          \label{Manin2}
		ad-q bc=1 = da - q^{-1} bc.
	\end{equation}
	Denote by $H=\mathcal{O}(U(1))=\mathbb{C}[t,t^{-1}]$ the algebra of rational complex polynomials in one variable $t$ with inverse $t^{-1}$. Then $H$ is a Hopf algebra with $t$ seen as a grouplike element, while $A$ can be given a right $H$-comodule algebra structure with coaction determined on generators by
	$$
	\begin{pmatrix}
		a & b\\
		c & d
	\end{pmatrix}\mapsto
	\begin{pmatrix}
		a & b\\
		c & d
	\end{pmatrix}\otimes
	\begin{pmatrix}
		t & 0\\
		0 & t
	\end{pmatrix},
	$$
        and
        extended to all of $A$ as an algebra morphism.
        It is well-known (see \cite[Ex.~6.26]{BJM} or \cite[\S3.3]{HanLandi} for the precise setting we are following here) that the right $H$-coinvariants $B =A^{\mathrm{co}H}$ coincide with the Podle\'s sphere and that $B\subseteq A$ is a faithfully flat Hopf-Galois extension.
For later use, 
        denote the generators of $B=\mathcal{O}_q(S^2)$ by
	\begin{equation*}
		B_0:=-q^{-1}bc,\qquad\qquad
		B_+:=cd,\qquad\qquad
		B_-:=-q^{-1}ab,
	\end{equation*}
which can be seen to satisfy the relations
        \reversemarginpar
	\begin{equation*}
		\begin{split}
			B_-B_0
			&=q^2B_0B_-,\\
			B_+B_0
			&=q^{-2}B_0B_+,
		\end{split}\qquad\qquad
		\begin{split}
			B_-B_+
			&=q^2B_0(1-q^2B_0),\\
			B_+B_-
			&=B_0(1-B_0).
		\end{split}
	\end{equation*}
Let us now	consider the Ehresmann-Schauenburg bialgebroid $\mathcal{H}=(A\otimes A)^{\mathrm{co}\mathcal{O}(U(1))}$.
	According to \cite[\S5.1]{HanLandi}, this is a left bialgebroid over $B$ with generators
	\begin{equation}
		\label{napoleon}
		\begin{split}
			\ga & = a \otimes d,\\
			\tilde\ga & = -q^{-1} a \otimes b,
		\end{split}\qquad
		\begin{split}
			\gb &= -q^{-1} b \otimes c,\\
			\tilde\gb &= d \otimes c,
		\end{split}\qquad
		\begin{split}
			\gamma &= -q^{-1} c \otimes b,\\
			\tilde\gamma &= c \otimes d,
		\end{split}\qquad
		\begin{split}
			\gd &= d \otimes a,\\
			\tilde\gd &= -q^{-1} b \otimes a,
		\end{split}
	\end{equation}
	satisfying certain relations resulting from Eqs.~\eqref{Manin},
and coproduct given on generators by
	\begin{equation*}
		\begin{split}
			\Delta \alpha
			&=\alpha\otimes_B\alpha+\tilde{\alpha}\otimes_B\tilde{\gamma},\\
			\Delta \beta
			&=q^2\beta\otimes_B\beta+\tilde{\delta}\otimes_B\tilde{\beta},\\
			\Delta\gamma
			&=\gamma\otimes_B\gamma+\tilde{\gamma}\otimes_B\tilde{\alpha},\\
			\Delta\delta
			&=\delta\otimes_B\delta+q^2\tilde{\beta}\otimes_B\tilde{\delta},
		\end{split}
		\qquad\qquad
		\begin{split}
			\Delta\tilde\alpha
			&=\alpha\otimes_B\tilde\alpha+\tilde{\alpha}\otimes_B\gamma,\\
			\Delta\tilde\beta
			&=q^2\tilde\beta\otimes_B\beta+\delta\otimes_B\tilde{\beta},\\
			\Delta\tilde\gamma
			&=\gamma\otimes_B\tilde\gamma+\tilde{\gamma}\otimes_B\alpha,\\
			\Delta\tilde\delta
			&=\tilde\delta\otimes_B\delta+q^2\beta\otimes_B\tilde{\delta},
		\end{split}
	\end{equation*}
	while the counit reads
	\begin{equation*}
		\begin{split}
			\varepsilon(\alpha)
			&=1-q^2B_0,
			\\
			\varepsilon(\tilde{\alpha})
			&=B_-,
		\end{split}
		\qquad
		\begin{split}
			\varepsilon(\beta)
			&=B_0,\\
			\varepsilon(\tilde{\beta})
			&=q^{-1}B_+,
		\end{split}
		\qquad
		\begin{split}
			\varepsilon(\gamma)
			&=B_0,\\
			\varepsilon(\tilde{\gamma})
			&=B_+,
		\end{split}
		\qquad
		\begin{split}
			\varepsilon(\delta)
			&=1-B_0,
			\\
			\varepsilon(\tilde{\delta})
			&=q^{-1}B_-.
		\end{split}
	\end{equation*}
\end{example}

\noindent These relations allow to explicitly write down the right Hopf structure on $\cH$:

\begin{lemma}
	The Ehresmann-Schauenburg bialgebroid 
	$\mathcal{H}=(A\otimes A)^{\mathrm{co}\mathcal{O}(U(1))}$ of the Hopf-Galois extension $\mathcal{O}_q(S^2)=\mathcal{O}_q(\mathrm{SL}_2(\mathbb{C}))^{\mathrm{co}\mathcal{O}(U(1))}\subseteq\mathcal{O}_q(\mathrm{SL}_2(\mathbb{C}))$
	is a right Hopf algebroid over $B$.
	Explicitly, the translation map \eqref{toujours} is given on generators as:
	\begin{equation*}
		\begin{array}{rclcrcl}
			\ga_{\smap } \otimes_B \ga_{\smam }
			&=&
			\ga \otimes_B \gd + q^2 \tilde\gamma  \otimes_B \tilde\gd, 
			&&
			\tilde \ga_{\smap } \otimes_B \tilde \ga_{\smam }
			&=&
			\tilde \ga    \otimes_B  \gd +  q^2 \gamma \otimes_B \tilde \gd,
			\\[1mm]
			\gb_{\smap } \otimes_B \gb_{\smam }
			&=&
			\gb     \otimes_B  \gamma +   \tilde \gb \otimes_B \tilde \ga,
			&&
			\tilde        \gb_{\smap } \otimes_B \tilde \gb_{\smam }
			&=&
			\gb    \otimes_B  \tilde \gamma +   \tilde \gb \otimes_B \ga,
			\\[1mm]
			\gamma_{\smap } \otimes_B  \gamma_{\smam }
			&=&
			q^2 \gamma       \otimes_B  \gb +   \tilde \ga \otimes_B \tilde \gb,
			&&
			\tilde        \gamma_{\smap } \otimes_B \tilde \gamma_{\smam }
			&=&
			q^2 \tilde \gamma         \otimes_B  \gb +  \ga \otimes_B \tilde \gb,
			\\[1mm]
			\gd_{\smap } \otimes_B   \gd_{\smam }
			&=&
			\gd        \otimes_B  \ga +   \tilde \gd \otimes_B \tilde \gamma,
			&&
			\tilde   \gd_{\smap } \otimes_B \tilde \gd_{\smam }
			&=&
			\gd    \otimes_B  \tilde \ga +   \tilde \gd \otimes_B \gamma.
		\end{array}
	\end{equation*}
\end{lemma}

\begin{proof}
	This is a special instance of Lemma \ref{quarkspeise}.
	Alternatively, this follows by the fact that, in this case, $\cH$ is even a {\em full} Hopf algebroid in the sense of \cite[Def.~4.1]{BoeSzl:HAWBAAIAD}, that is, equipped with an invertible (and in this case involutive) antipode $S$. As such, it is in particular a right Hopf algebroid as proven in {\em op.~cit.}, Prop.~4.2, while the translation map~\eqref{toujours}
        is obtained via
	\begin{equation}
		\label{rothko}
		h_{\smap } \otimes_B h_{\smam } := S \big((S^{-1} h)_{(1)}\big) \otimes_B (S^{-1}h)_{(2)}
		= S \big((S h)_{(1)}\big) \otimes_B (S h)_{(2)}
	\end{equation}
	for any $h \in \cH$. This already concludes the proof. Nevertheless, let us exemplify this on the generator $\ga$ and show that this indeed inverts the Hopf-Galois map \eqref{eq:HG} by proving the relation \eqref{Tch2}. To start with, from the coproduct and \eqref{rothko}, we obtain
	\begin{equation*}
		\begin{split}
			\ga_{\smap } \otimes_B \ga_{\smam }
			&
			= S \big((S \ga)_{(1)}\big) \otimes_B (S \ga )_{(2)}
			= S \gd_{(1)}  \otimes_B \gd_{(2)}
			=
			S \gd \otimes_B \gd + S(q^2 \tilde\gb)  \otimes_B \tilde\gd 
			\\
			&
			=
			\ga \otimes_B \gd + q^2 \tilde\gamma  \otimes_B \tilde\gd. 
		\end{split}
	\end{equation*}
	With this, one computes:
	\begin{equation*}
		\begin{split}
			\ga_{\smap (1)} \ga_{\smam } \otimes_B \ga_{\smap (2)}
			&
			=
			\ga_{(1)}\gd  \otimes_B \ga_{(2)} + q^2 \tilde \gamma_{(1)} \tilde \gd  \otimes_B \tilde \gamma_{(2)} 
			\\
			&
			=
			\ga \gd  \otimes_B \ga + \tilde \ga \gd  \otimes_B \tilde \gamma
			+ q^2 \tilde \gamma \tilde \gd  \otimes_B \ga + q^2 \gamma \tilde \gd  \otimes_B \tilde \gamma 
			\\
			&
			=
			(\ga \gd  + q^2 \tilde \gamma \tilde \gd) \otimes_B \ga
			+ (\tilde \ga \gd  + q^2 \gamma \tilde \gd )\otimes_B \tilde \gamma
			\\
			&
			=
			(1 \otimes ad) \otimes_B (a \otimes d)
			-  q^{-1} (1 \otimes ab) \otimes_B (c \otimes d).
		\end{split}
	\end{equation*}
	In the last step, we used the expressions \eqref{napoleon} that give for the first tensor factor in the first summand, 
	\begin{equation*}
		\begin{split}
			\ga \gd  + q^2 \tilde \gamma \tilde \gd
			& =
			(a \otimes d)(d \otimes a) + q^2 (c \otimes d)(-q^{-1} b \otimes a)
			=
			ad \otimes ad  - q cb \otimes ad = 1 \otimes ad
		\end{split}
	\end{equation*}
	by means of Eqs.~\eqref{Manin} \& \eqref{Manin2}.
        A similar computation yields
	$
	\tilde \ga \gd  + q^2 \gamma \tilde \gd = -q^{-1} (1 \otimes ab) 
	$
	for the second summand above. Considering now the relevant tensor product, that is, $\cH_\ract \otimes_B \due \cH \lact {}$, 
	\begin{equation*}
		\begin{split}
			\ga_{\smap (1)} \ga_{\smam } \otimes_B \ga_{\smap (2)}
			&
			=
			(1 \otimes ad) \otimes_B (a \otimes d)
			-  q^{-1} (1 \otimes ab) \otimes_B (c \otimes d).
			\\
			&
			=
			(1 \otimes 1) \otimes_B (ada \otimes d)
			-  q^{-1} (1 \otimes 1) \otimes_B (abc \otimes d).
			\\
			&
			=
			(1 \otimes 1) \otimes_B (a(da - q^{-1} bc)  \otimes d)
			\\
			&
			=
			1 \otimes_B \ga,
		\end{split}
	\end{equation*}
	where we still used Eq.~\eqref{Manin2},
        and which was to show.
	Verifying the relation \eqref{Tch3} as well as the respective computations for all other generators is similar.
\end{proof}

\subsection{Yetter-Drinfel'd modules and the scalar extension Hopf algebroid}
\label{adapter}
Let us recall the notion of scalar extension Hopf algebroid from \cite[\S4]{BrzMil:BBAD}. For a $\Bbbk$-Hopf algebra $H$ with invertible antipode $S$, the category ${}_H\mathcal{YD}^H$ of (left-right) \textit{Yetter-Drinfel'd modules} consists of left $H$-modules $M$ which are at the same time right $H$-comodules, subject to the compatibility condition
\begin{equation*}
	h_{(1)} m_{(0)}\otimes h_{(1)}m_{(1)}
	=(h_{(2)} m)_{(0)}\otimes(h_{(2)} m)_{(1)}h_{(1)}
\end{equation*}
for all $m\in M$ and $h\in H$. Morphisms in this category are left $H$-linear and right $H$-colinear. It is well-known that ${}_H\mathcal{YD}^H$ is braided monoidal (and equivalent to the {\em right weak} monoidal centre of ${}_H \cM$)
via the customary diagonal action on the tensor product, while
the braiding $\sigma^\mathcal{YD}$ of ${}_H\mathcal{YD}^H$ is given on elements $m\in M$ and $n\in N$ for $M,N\in{}_H\mathcal{YD}^H$ by
$$
\sigma^\mathcal{YD}_{M,N}(m\otimes n)
:=n_{(0)}\otimes n_{(1)}  m.
$$
A \textit{braided commutative algebra} in ${}_H\mathcal{YD}^H$ is a $\Bbbk$-algebra $(A,m,\eta)$ such that $A$ is an object and $m\colon A\otimes A\to A$ and $\eta\colon\Bbbk\to A$ are morphisms in ${}_H\mathcal{YD}^H$, and such that $m\circ\sigma^\mathcal{YD}_{A,A}=m$.

For every braided commutative algebra $A$ in ${}_H\mathcal{YD}^H$ we obtain a (full) Hopf algebroid $\mathcal{H}:=A\#H$ over $A$, the \textit{scalar extension Hopf algebroid} \cite[Thm.~4.1]{BrzMil:BBAD}. 
As such, it is automatically a right Hopf algebroid, and from the (full) antipode formula given in {\em loc.~cit.} one could deduce the explicit form of the translation map as in \eqref{toujours} by deploying \cite[Prop.~4.2 \& Ex.~4.14]{BoeSzl:HAWBAAIAD}, but it is technically simpler to just give a candidate and then prove that it does the job, see below.
As a $\Bbbk$-algebra, $\mathcal{H}$ is the smash product algebra $A\#H$, that is, the $\Bbbk$-module $A\otimes H$ (one usually writes $a\#h$ instead of $a\otimes h$ in this case) endowed with the product
\begin{equation*}
  m_\#\colon A\#H\otimes A\#H\to A\#H,
  \quad
	(a\#h)\otimes(b\#g)\mapsto(a\#h) (b\#g):=a(h_{(1)} b)\#h_{(2)}g
\end{equation*}
and unit $1\# 1$. Source and target $s,t \colon A \to \mathcal{H}$ of are defined by
\begin{equation}
	\label{ratagnan}
	s(a):=a\#1,
        \qquad\qquad
	t(a):=a_{(0)}\#a_{(1)},
\end{equation}
while the comultiplication $\Delta\colon\mathcal{H}\to\mathcal{H}\otimes_A\mathcal{H}$ and counit $\varepsilon\colon\mathcal{H}\to A$ are given as
\begin{equation*}
	\Delta(a\#h):=(a\#h_{(1)})\otimes_A(1\#h_{(2)}),\qquad\qquad
	\varepsilon(a\#h):=\varepsilon_H(h)a
\end{equation*}
for all $a\in A$ and $h\in H$. Finally, the Hopf-Galois map
$$
\alpha((a\#h)\otimes_A(b\#g))
=(a\#h_{(1)}) (b\#g)\otimes_A(1\#h_{(2)})
=(a(h_{(1)} b)\#h_{(2)}g)\otimes_A(1\#h_{(3)})
$$
is invertible with inverse determined by the translation map
\begin{equation*}
	(a\#h)_{\smap }\otimes_A(a\#h)_{\smam }
	=(a_{(0)} \#h_{(2)})\otimes_A \big( 1 \#S^{-1}(h_{(1)}) a_{(1)}\big),
\end{equation*}
which is the map \eqref{toujours} that defines right Hopf algebroids (over a left bialgebroid).
This is, in fact, the case since Eqs.~\eqref{Tch2} \& \eqref{Tch3} hold; the first is easy to see by explicitly using the target map \eqref{ratagnan} in the tensor product,
\begin{align*}
	(a\#h)_{\smap (1)} (a\#h)_{\smam } \otimes_A (a\#h)_{\smap (2)}
	&=
	(a_{(0)} \# h_{(2)})\big(1 \# S^{-1}(h_{(1)})a_{(1)}\big) \otimes_A (1 \# h_{(3)})
	\\
	&=
	(a_{(0)} \# a_{(1)}\big) \otimes_A (1 \# h)
	\\
	&=
	(1 \# 1) \otimes_A (a\#h),
\end{align*}
while the second one follows from 
\begin{align*}
	(a\#h)_{(2)\smap } \otimes_A (a\#h)_{(2)\smam }(a\#h)_{(1)}
	&=
	(1\#h_{(3)})  \otimes_A \big(1 \# S^{-1}(h_{(2)})\big) (a \# h_{(1)})
	\\
	&=
	(1\#h_{(2)})  \otimes_A \big(S^{-1}(h_{(1)})(a) \# 1\big)
	\\
	&=
	(1\#h_{(2)})\big(S^{-1}(h_{(1)})(a) \# 1\big) \otimes_A (1 \# 1)
	\\
	&=
	(a\#h) \otimes_A (1 \# 1).
\end{align*}

\begin{example}
There is a construction of scalar extension bialgebroids based on coquasitriangular Hopf algebras, see, for example, \cite[Ex.~4.2]{BrzMil:BBAD}. We recall that a {\em coquasitriangular bialgebra} is a $\Bbbk$-bialgebra $H$ together with a convolution invertible map $\mathcal{R}\colon H\otimes H\to\Bbbk$, the {\em universal $\mathcal{R}$-form}, which makes $H$ quasi-commutative, that is, $\mathcal{R}(h_{(1)}\otimes h'_{(1)})h_{(2)}h'_{(2)}=h'_{(1)}h_{(1)}\mathcal{R}(h_{(2)}\otimes h'_{(2)})$ for all $h,h'\in H$, and which is subject to the hexagon relations
\begin{equation*}
	\mathcal{R}(hh'\otimes h'')
	=\mathcal{R}(h\otimes h''_{(1)})\mathcal{R}(h'\otimes h''_{(2)}),\qquad\qquad
	\mathcal{R}(h\otimes h'h'')
	=\mathcal{R}(h_{(2)}\otimes h')\mathcal{R}(h_{(1)}\otimes h'').
\end{equation*}
We call $H$ {\em cotriangular} if, in addition, the convolution inverse of $\mathcal{R}$ is given by $\mathcal{R}^{-1}=\mathcal{R}^{\mathrm{op}}$. It is well-known that the monoidal category $\mathcal{M}^H$ of right $H$-comodules is braided if $H$ is coquasitriangular with respective braiding
\begin{equation*}
	\sigma^\mathcal{R}_{M,N}\colon M\otimes N\to N\otimes M,\quad
	m\otimes n \mapsto n_{(0)}\otimes m_{(0)}\mathcal{R}(m_{(1)}\otimes n_{(1)})
\end{equation*}
for $M,N \in \cM^H$,
where $\mathcal{R}$ is the corresponding universal $\mathcal{R}$-form of $H$. Finally, the braided monoidal category $\mathcal{M}^H$ is symmetric if and only if $H$ is cotriangular.

For a coquasitriangular bialgebra 
$(H,\mathcal{R})$ and a braided commutative algebra $A$ in $\mathcal{M}^H$, the braided commutativity explicitly reads 
$$
ba=a_{(0)}b_{(0)}\mathcal{R}(b_{(1)}\otimes a_{(1)})
$$
for all $a,b\in A$. We can endow $A$ with the left $H$-action
\begin{equation*}
  \cdot\colon H\otimes A\to A,
  \quad
	h\cdot a:=a_{(0)}\mathcal{R}(a_{(1)}\otimes h),
\end{equation*}
and one easily verifies that in this way $A$ becomes a braided commutative algebra in ${}_H\mathcal{YD}^H$. In particular, if $H$ is coquasitriangular and $A$ braided commutative as before, then we obtain a scalar extension bialgebroid $A\#H$ from the previous considerations. A class of particular interest is given by the FRT bialgebras \cite{FRT}, see \cite[Prop.\ 4.3]{BrzMil:BBAD}. 
This kind of con\-struc\-tion on scalar extension Hopf algebroids will give one of the main classes of examples for differential calculi in \S\ref{sec:scalar-calc}.
\end{example}

\subsection{Comodules and Hopf-Galois comodules}\label{mittagleffler}
Given a left bialgebroid $(\mathcal{H},R, s, t, \Delta,\varepsilon)$, recall that a \textit{left $\mathcal{H}$-comodule} is a pair $(M,\lambda_M)$, where
$M\in{}_R\mathcal{M}$ and $\lambda_M\colon M\to \cH_\ract \otimes_RM, \ m \mapsto m_{(-1)}\otimes_Rm_{(0)}$ is a
left $R$-linear left coaction over the underlying coring $\cH$, which generates a right $R$-action $m r:=\varepsilon(m_{(-1)}\bract r) m_{(0)}$ to which we refer as the {\em induced} right $R$-action on $M$, and with respect to which the coaction becomes $R$-linear as well. One therefore has
\begin{equation}
	\label{ha}
	\lambda_M(r m r')
	=r\lact m_{(-1)}\bract r'\otimes_Rm_{(0)},\qquad\qquad
	m_{(-1)}\otimes_Rm_{(0)} r=r\blact m_{(-1)}\otimes_Rm_{(0)},
\end{equation}
where the second identity reflects the fact that
$\lambda_M$ corestricts to the Takeuchi subspace
\begin{equation*}
\cH \times_R M = \big\{ \textstyle\sum_i h^i \otimes m^i \in \cH_\ract \otimes_R M \mid  \sum_i r \blact h^i \otimes m^i = \sum_i h^i \otimes m^ir, \ \forall r \in R  \, \big\},
\end{equation*}
similar to the coproduct in $\cH$ itself.
Let us denote by
${}^\mathcal{H}\!\mathcal{M}$
the category of left $\mathcal{H}$-comodules, while
\begin{equation}
	\label{waterman}
	{}^{\mathrm{co}\mathcal{H}\!}M:=\big\{ m\in M \mid \lambda_M(m)=1\otimes_Rm \big\}
\end{equation}
denotes the $R$-bimodule of left $\mathcal{H}$-{\em coinvariant} elements.

{\em Right} $\cH$-comodules are defined analogously, 
the full details of which we omit, forming a category $\cM^\cH$:
one starts from a right $R$-module $N$ equipped with a map $\rho_N \colon N \to N \otimes_R \due \cH \lact {}$, denoted $n \mapsto n_{(0)} \otimes_R n_{(1)}$, which induces a left $R$-action $rn := n_{(0)} \gve(r \blact n_{(1)})$ such that $\rho_N$ becomes $R$-bilinear with respect to $\blact, \ract$, and similarly $\mathrm{im}(\rho_N) \subseteq N \times_R \cH$ defined analogously.

    The {\em cotensor product} $N \bx_\cH M$ of a right resp.\ left $\cH$-comodule is, as always, given as the equaliser of $\rho_N \otimes_R M$ and $N \otimes_R \gl_M$, that is, the kernel of the difference map.
    Observe that
    \begin{equation}
      \label{callisto}
    N \bx_\cH M \subseteq N \times_R M,
    \end{equation}
    where 
$
N \times_R M = \big\{ \textstyle\sum_i n^i \otimes m^i \in N \otimes_R M \mid  \sum_i r n^i \otimes m^i = \sum_i n^i \otimes m^ir, \ \forall r \in R  \, \big\},
$
analogously to the above. Indeed, for $n \otimes_R  m \in N \bx_\cH M$, we have
$$
rn \otimes_R m = n_{(0)} \gve(r \blact n_{(1)}) \otimes_R m =
n_{(0)} \otimes_R \gve(r \blact n_{(1)}) m = n \otimes_R \gve(m_{(-1)} \bract r) m_{(0)} = n \otimes_R mr. 
$$
This way, one actually deals with the two maps $\rho_N \times_R M$ and $N \times_R \gl_M$ with, a priori, images in $(N \times_R \cH) \times_R M$ resp.\
$N \times_R (\cH \times_R M)$. These are, in general, not isomorphic but can be both mapped to $N \times_R \cH \times_R M$, where they can be compared (as one does for giving a sense to the coassociativity of a bialgebroid coproduct). See \cite{Tak:GOAOAA} for all technical details, in particular on the nonassociativity of the product $\times_R$.

\begin{definition}
	\label{minestra}
        The \textit{Galois} map
for a left comodule $M$ over a left bialgebroid $(\cH, R)$ 
        is defined as the map
\begin{equation}
	\label{duracell}
	\alpha_M \colon M\otimes_R\mathcal{H}\to \mathcal{H}\otimes_RM, \quad m\otimes_Rh\mapsto m_{(-1)}h\otimes_Rm_{(0)}.
\end{equation}
If it  is invertible, we
	call
	$M$ a {\em Hopf-Galois comodule} and 
	write
	$$
	m_{\smap }\otimes_R m_{\smam }:= \alpha_M^{-1}(1_\cH\otimes_R m)
	$$
	for the {\em translation map} $M \to M \otimes_R \cH$ on $M$.
\end{definition}

In such a case, {\em i.e.}, if $M$ is a Hopf-Galois comodule, the following identities  taken from \cite[Prop.~3.0.6]{Che:ITFLHLB}
are satisfied:
\begin{align}
	m_{\smap }\otimes_R m_{\smam }&\in M \times_R\mathcal{H},\label{Mch1}
	\\
        \label{Mch2}
	m_{\smap (-1)} m_{\smam }\otimes_Rm_{\smap (0)}
	&=1\otimes_Rm,
	\\
        \label{Mch3}
	m_{(0)\smap }\otimes_R m_{(0)\smam } m_{(-1)}
	&= m \otimes_R 1,
	\\
\label{Mch4}
        m_{\smap \smap }\otimes_R m_{\smap \smam }\otimes_R m_{\smam }
	&=m_{\smap }\otimes_Rm_{\smam (1)}\otimes_Rm_{\smam (2)},
        	\\
\label{Mch5}
        (rmr')_{\smap }\otimes_{R}(rmr')_{\smam }
	&=m_{\smap }\otimes_{R} r \blact m_{\smam } \ract r',
		\\
\label{Mch6}
        m_{\smap } \varepsilon(m_{\smam })
	&=m,
        \end{align}
for $r,r' \in R$ and $m \in M$,
where in \eqref{Mch1} we mean the Takeuchi subspace
$$
M \times_R \cH = \big\{ \textstyle\sum_i m^i \otimes h^i \in M \otimes_R \due \cH \lact {} \mid  \sum_i rm^i \otimes h^i = \sum_i m^i \otimes h^i \bract r, \ \forall r \in R  \, \big\}.
$$
It follows from \eqref{Mch4} and \eqref{Mch6} that $\rho_M := \ga^{-1}_M (1_\cH \otimes_R -)$, that is, the map
\begin{equation}
	\label{fabercastell}
	\rho_M \colon M \to M \otimes_R \cH, \quad m \mapsto m_{\smap } \otimes_R m_{\smam }
\end{equation}
turns the left $\cH$-comodule $M$ one started with into a right one.
The invertibility of the Galois map of a left $\mathcal{H}$-comodule is automatic if $\mathcal{H}$ is a right Hopf algebroid,
as we will show now.
The following is an enhancement of Theorem 3.0.4 in \cite{Che:ITFLHLB}.

\begin{lemma}
    \label{lem:GalInv}
    A left bialgebroid $(\cH, R)$ is a right Hopf algebroid
        if and only if
	any left $\cH$-comodule is a Hopf-Galois comodule, that is to say, has a bijective Galois map as in Eq.~\eqref{duracell}.
\end{lemma}

\begin{proof}
  Assume first that $(\cH, R)$ is a right Hopf algebroid and $M \in {}^\cH \! \cM$. Observe then that both domain and codomain of the map $\ga$ in \eqref{eq:HG} are right $\cH$-comodules: the first by using the coproduct on the first factor along with the right $R$-action
  $(\cH_\bract \otimes_R \due \cH \lact {})_R := (\cH_\ract)_\bract \otimes_R \due \cH \lact {}$, that is, by means of the map $h \otimes_R h' \mapsto (h_{(1)} \otimes_R h') \otimes_R h_{(2)}$;
  the second by using the coproduct on the second factor and right $R$-action
  $(\cH_\ract \otimes_R \due \cH \lact {})_R := \cH_\ract \otimes_R \due {(\cH_\ract)} \lact {}$, that is, via the map $h \otimes_R h' \mapsto (h \otimes_R h'_{(1)} ) \otimes_R h'_{(2)}$.
With this, let us show that one obtains a commutative diagram
\begin{equation}
  \label{fai}
		\begin{tikzcd}
(\cH_\bract \otimes_R \due \cH \lact {}) \bx_\cH M
\arrow{rr}{\ga \kasten{5pt}_\cH M}
\arrow{d}[swap]{\phi}
& &
(\cH_\ract \otimes_R \due \cH \lact {}) \bx_\cH M 
                  \arrow{d}{\psi}
                   \\
M \otimes_R \due \cH \lact {}
 \arrow{rr}{\ga_M}
 & & 
\cH_\ract \otimes_R M,
		\end{tikzcd}
\end{equation}
where the vertical arrows are isomorphisms given by
$$
\phi \colon (\cH_\bract \otimes_R \due \cH \lact {}) \bx_\cH M \to M \otimes_R \due \cH \lact {}, \quad (h \otimes_R h') \otimes_R m \mapsto \gve(h)m \otimes_R h'
$$
with inverse $ m \otimes_R h \mapsto (m_{(-1)} \otimes_R h) \otimes_R m_{(0)} $, as well as
$$
\psi \colon (\cH_\ract \otimes_R \due \cH \lact {}) \bx_\cH M \to \cH_\ract \otimes_R M, \quad (h \otimes_R h') \otimes_R m \mapsto h \otimes_R \gve(h')m,
$$
with inverse $h \otimes_R m \mapsto (h \otimes_R m_{(-1)}) \otimes_R m_{(0)}$. While it is obvious that $\psi$ is well-defined, this is less evident for $\phi$; after a little thought, this becomes clear from the fact that the cotensor product $N \bx_\cH M$ of any two (right and left) comodules lives in the Takeuchi subspace $N \times_R M$, see \eqref{callisto}, and hence, for
$
(h \otimes_R h') \otimes_R m
\in (\cH_\bract \otimes_R \due \cH \lact {}) \bx_\cH M,
$
one has $(r \blact h \otimes_R h') \otimes_R m = (h \otimes_R h') \otimes_R mr$, and therefore
$$
\gve(h\bract r)m \otimes_R h' = \gve(r \blact h)m \otimes_R h'
= \gve(h)mr \otimes_R h' = \gve(h)m \otimes_R r \lact h', 
$$
which makes $\phi$ well-defined. Next, a direct computation yields
\begin{equation*}
  \begin{split}
    \big(\psi \circ (\ga \bx_\cH M) \circ \phi^{-1}\big)(m \otimes_R h)
    &=
    \big(\psi \circ (\ga \bx_\cH M)\big)\big((m_{(-1)} \otimes_R h) \otimes_R m_{(0)}\big)
    \\
    &=
    \psi\big((m_{(-2)}h \otimes_R m_{(-1)} )\otimes_R m_{(0)}\big) 
  \\
    &=
  m_{(-1)}h \otimes_R m_{(0)}
  \\
    &=
 \ga_M(m \otimes_R h), 
  \end{split}
\end{equation*}
and thus \eqref{fai} commutes. 
Assuming now that $\ga$ is invertible, then so is
$\ga \bx_\cH M$
  for any left $\cH$-comodule $M$, and the commutativity of the above diagram yields the invertibility of $\ga_M$ as well, which was to prove.
  Moreover, we can explicitly compute
\begin{equation*}
  \begin{split}
\ga_M^{-1}(h \otimes_R m) 
&=
\big(\phi \circ (\ga \bx_\cH M)^{-1} \circ \psi^{-1} \big)(h \otimes_R m)
\\
&=
\big(\phi \circ (\ga^{-1} \bx_\cH M) \big)\big((h \otimes_R m_{(-1)}) \otimes_R m_{(0)} \big)
    \\
    &=
\phi\big((m_{(-1)\smap} \otimes_R m_{(-1)\smam}h) \otimes_R m_{(0)} \big)
   \\
    &=
\varepsilon(m_{(-1)\smap }) m_{(0)} \otimes_R m_{(-1)\smam } h
  \end{split}
\end{equation*}
  for the inverse of the Hopf-Galois map, which amounts to saying that
\begin{equation}
	\label{fabercastell0}
	m_{\smap } \otimes_R m_{\smam } :=  \varepsilon(m_{(-1)\smap }) m_{(0)} \otimes_R m_{(-1)\smam }
\end{equation}
for the right coaction \eqref{fabercastell} on $M$ in this case. This expression was already found in \cite[Thm.~4.1.1]{CheGavKow:DFOLHA} in a different way, 
but subject to a projectivity assumption.

The opposite implication is obvious: if $\ga_M$ is invertible for any $M \in {}^\cH \! \cM$, then it is so, in particular, for the $\cH$-comodule given by $(\cH, \gD)$ itself. This amounts to the invertibility of the map \eqref{eq:HG} as a defining property of right Hopf algebroids.
\end{proof}

\subsection{Hopf modules}
\label{sec:HopfModules}
Having introduced the category ${}^\mathcal{H}\!\mathcal{M}$ of left comodules over left bialgebroids, it is natural to study those left comodules which, in addition, admit a compatible (left or right) $\mathcal{H}$-action: the so-called {\em Hopf modules}, and related to it, {\em covariant bimodules}.
The second part of the subsequent definition of two types of Hopf modules
appeared essentially in \cite{Boe:ITFHA}:

\begin{definition}
  \label{genzano}
  Let $(\mathcal{H}, R)$ be a left bialgebroid and let $M$ be a left $\cH$-comodule with its {\em natural} left resp.\ {\em induced} right $R$-action written
  $
  r \otimes m \otimes  r' \mapsto rmr'
  $
  for $r,r' \in R$ and $m \in M$.
  \begin{enumerate}
\compactlist{99}
\item
  \label{genzano1}
  If $M$ is, in addition, a right $\cH$-module such that
  \begin{enumerate}
    \compactlist{99}
       \item
the induced right $R$-action on $M$ as a left $\cH$-comodule {\em coincides} with the right $R$-action induced by the right $\cH$-action, that is, $m r= m\bract r$;
\item
  the natural left $R$-action on $M$ as a left $\cH$-comodule {\em commutes} with the right $\cH$-action on $M$;
\item
  the equation
  $\lambda_M(m h)=\lambda_M(m)\Delta(h)$ holds as a product in $\cH \times_R M$,
  \end{enumerate}
then
we call $M$ a \textit{right-left Hopf module} over $\cH$, and
write $M\in{}^\mathcal{H}\!\mathcal{M}_\mathcal{H}$.
\item
  \label{genzano2}
If $M$ is, in addition, a left $\cH$-module such that
\begin{enumerate}
  \compactlist{99}
\item
  the natural left $R$-action on $M$ as a left $\cH$-comodule {\em coincides} with the left $R$-action induced by the left $\cH$-action, that is, 
  $r m = r \lact m$;
\item
  the induced right $R$-action on $M$ as a left $\cH$-comodule {\em commutes} with the left $\cH$-action on $M$;
\item
  the equation
  $\lambda_M(hm)=\gD(h) \lambda_M(m)$ holds as a product in $\cH \times_R M$, 
\end{enumerate}
then
  we call $M$ a \textit{left-left Hopf module} over $\cH$, and 
write $M\in{}^\mathcal{H}_\mathcal{H}\mathcal{M}$.
\end{enumerate}
  \end{definition}

\begin{remark}
  \label{roslagen}
  Condition (b) in part (ii) is actually redundant as this follows from (c), in striking contrast to the respective requirements in part (i), where it is needed to guarantee well-definedness. We wrote it this way to have a symmetric picture; however, this reveals  a certain structural difference between left-left and right-left Hopf modules in the bialgebroid case: while ${}_\cH \mathcal{M}$ for a left bialgebroid is a monoidal category, $\cM_\cH$ is not. This allows to define ${}^\mathcal{H}_\mathcal{H}\mathcal{M}$ as the category of left $\cH$-comodules in ${}_\mathcal{H}\mathcal{M}$, while additional compatibility conditions have to be enforced in order to define ${}^\mathcal{H}\!\mathcal{M}_\mathcal{H}$.
    \end{remark}

If a (say) left-left Hopf module is, in addition, a Hopf-Galois comodule in the sense of Definition \ref{minestra} that the corresponding Galois map is invertible, then we can add a technical statement that will be needed later on:

\reversemarginpar

\begin{lemma}
\label{minestrone}
Let $(\cH, R)$ be a right Hopf algebroid
and $M\in {}^\cH\hskip -2pt\cM_\cH$ be a left-right Hopf module. Then 
\begin{equation*}
\label{scalea1}
  \rho_M(mh)
= (mh)_{\smap } \otimes_R (mh)_{\smam }
  =  m_{\smap } h_{\smap } \otimes_R h_{\smam } m_{\smam }  
  \end{equation*}
holds 
with respect to the right $\cH$-coaction
\eqref{fabercastell0} for all $m \in M$, $h \in \cH$.
 Likewise, if $M\in {}^\cH_\mathcal{H}\mathcal{M}$ is a left-left Hopf module, then
\begin{equation*}
  \label{scalea2}
   \rho_M(hm)
 = (hm)_{\smap } \otimes_R (hm)_{\smam }
   = h_{\smap } m_{\smap } \otimes_R m_{\smam } h_{\smam } 
   \end{equation*}
holds for all $m \in M$ and $h \in \cH$.
\end{lemma}

\begin{proof}
Let us only prove the first case, the second one being similar (but easier).
  Both are a direct computation using Definition \ref{genzano} (i) along with Eq.~\eqref{fabercastell0}. Indeed,
  \begin{equation*}
    \begin{split}
      (mh)_{\smap } \otimes_R (mh)_{\smam }
      &=
      \varepsilon\big((mh)_{(-1)\smap }\big)  (mh)_{(0)} \otimes_R (mh)_{(-1)\smam }
      \\
      &=
      \varepsilon(m_{(-1)\smap }h_{(1)\smap })  (m_{(0)}h_{(2)}) \otimes_R h_{(1)\smam } m_{(-1)\smam }
      \\
      &=
      \big(\varepsilon(m_{(-1)\smap }h_{(1)\smap })  m_{(0)}\big) h_{(2)} \otimes_R h_{(1)\smam } m_{(-1)\smam }
 \\
      &=
      \big(\varepsilon(m_{(-1)\smap }h_{\smap(1)})  m_{(0)}\big) h_{\smap(2)} \otimes_R h_{\smam } m_{(-1)\smam }
      \\
      &=
      \pig(\varepsilon\big(\gve(h_{\smap(1)}) \blact m_{(-1)\smap }\big)  m_{(0)}\pig) h_{\smap(2)} \otimes_R h_{\smam } m_{(-1)\smam }
          \end{split}
  \end{equation*}
      \begin{equation*}
    \begin{split}
\phantom{       (mh)_{\smap } \otimes_R (mh)_{\smam }}
      &=
    \big(  \varepsilon(m_{(-1)\smap })  m_{(0)}  \gve(h_{\smap(1)}) \big) h_{\smap(2)} \otimes_R h_{\smam } m_{(-1)\smam }
      \\
      &=
    \big(  \varepsilon(m_{(-1)\smap })  m_{(0)}\big)\big(\gve(h_{\smap(1)}) \lact h_{\smap(2)}\big) \otimes_R h_{\smam } m_{(-1)\smam }
    \\
      &=
     m_{\smap }  h_{\smap } \otimes_R h_{\smam } m_{\smam },
          \end{split}
  \end{equation*}
  where in the second step we used
both  Definition \ref{genzano}, part\,(i)\,(c), and \eqref{Tch6}, in step three Definition \ref{genzano}, part\,(i)\,(b), 
then Eq.~\eqref{Tch4} in the fourth step,
the properties \eqref{counityippieh}
  of a bialgebroid counit in the fifth,
the 
Takeuchi property \eqref{ha} for a left $\cH$-comodule
along with \eqref{Tch9} in the sixth,
again 
Definition \ref{genzano}, part\,(i)\,(b), in the seventh, and finally
 counitality of $\cH$ along with \eqref{fabercastell0} to conclude in the eighth.
 \end{proof}

For (right-left) Hopf modules with invertible Galois map one can prove a  \textit{Fundamental Theorem}, which generalises the classical fundamental theorem of Hopf modules \cite[Thm.~4.1.1]{Swe:HA} for  Hopf algebras.
The following is a slight extension and reformulation of \cite[Thm.~5.0.1]{Che:ITFLHLB}, which is why we omit a full proof here:

\begin{theorem}
  \label{thm:HopfModules}
  Let $(\cH,R)$ be a left bialgebroid such that $\cH_\ract$ is flat
  over $R$, and
  let $M \in {}^\cH \!\mathcal{M}_\cH$ be a right-left Hopf module that is simultaneously a Hopf-Galois comodule over $\cH$.
  Then
\begin{enumerate}
 \compactlist{50}
\item
  one has $m_{\smap }m_{\smam }\in {}^{\mathrm{co}\mathcal{H}\!}M$ for all $m\in M$;
\vskip 1.2pt
\item
  the module of coinvariants  ${}^{\mathrm{co}\mathcal{H}\!}M$ is a left $R$-module by means of $r\blact {m}= {m}t(r)$ for $m \in {}^{\mathrm{co}\mathcal{H}\!}M$ and $r \in R$;
  \vskip 1.2pt
\item if $(\mathcal{H},R)$ is a right Hopf algebroid, then
the functors
  $$
  {}^{\mathrm{co}\mathcal{H}\!}(-) \colon {}^\cH \!\mathcal{M}_\cH
\quad \raisebox{-3pt}{$\longleftrightarrows$} \quad {}_R \cM \ \colon \!  \cH_\ract \otimes_R - 
  $$
are mutually adjoint and  establish an equivalence of categories. In particular, the counit of the adjunction is given by the isomorphism 
  $$
\xi \colon  \mathcal{H}_\ract  \otimes_R{}^{\mathrm{co}\mathcal{H}\!}M \to  M, \quad h\otimes_R {m}\mapsto {m}h, 
$$
in the category ${}^\mathcal{H}\!\mathcal{M}_\mathcal{H}$ of right-left Hopf modules, with inverse
$$
M \to \mathcal{H}_\ract  \otimes_R {}^{\mathrm{co}\mathcal{H}\!}M, \quad
m\mapsto m_{(-1)}\otimes_R m_{(0)\smap } m_{(0)\smam },
$$
while the
unit results as the isomorphism
$
\eta \colon V \to {}^{\mathrm{co}\mathcal{H}} (\mathcal{H}_\ract \otimes_RV),
\
v\mapsto 1_\cH \otimes_Rv,
$
with inverse
$
h \otimes_R v \mapsto \varepsilon(h) v.
$
\end{enumerate}
\end{theorem}

\begin{proof}
We only add to {\em loc.~cit.}\ a discussion of the unit: the right-left Hopf module structure on $\cH_\ract \otimes_R V$ for a left $R$-module $V$ arises from right multiplication on the first tensor factor $\cH$, along with 
    $\gD \otimes_R V$ as  left $\cH$-coaction, which has image in
    $\cH \times_R (\cH_\ract \otimes_R M)$, and where the Takeuchi product on the right factor is defined with respect to the source maps on $\cH$ in $\cH \otimes_R M$. The condition in Definition \ref{genzano}\,(i)\,(c) is then easily verified.
    Finally, while obviously $\eta^{-1} \circ \eta = \id_V$, one has
    $(\eta \circ \eta^{-1})(h \otimes_R v)
= 1_\cH \otimes_R \gve(h)v = 1_\cH \ract \gve(h) \otimes_R v = h_{(1)} \ract \gve(h_{(2)}) \otimes_R v = h \otimes_R v$, using $h \otimes_R v \in {}^{\mathrm{co}\cH} (\cH_\ract \otimes_R V)$ in the third step.
  \end{proof}

\noindent
There is an analogue of the above theorem for the category ${}^\mathcal{H}_\mathcal{H}\mathcal{M}$, which follows from homological arguments rather than from explicit maps.
A version of this fundamental theorem for full Hopf algebroids, {\em i.e.}, those with a sort of antipode, has been given in \cite[Thm.~4.2]{Boe:ITFHA}, and a very general form of such a statement has been proven in \cite[Thm.~5.6]{Brz:TSOC}, see also \cite[\S28.19]{BrzWis:CAC} in the framework of corings. An even further reaching generalisation to the realm of bimonads appears in \cite[Thm.~6.11]{BruLacVir:HMOMC}, see also \cite[Thm.~5.3.1]{AguCha:GHMFB} for later development.

\subsection{Covariant bimodules}\label{sec:CovBim}

Next, let us discuss covariant bimodules: they will be of partic\-ular interest since they will model the module of one-forms for right Hopf algebroids later~on.

\begin{definition}
\label{laparolacheciunisce}
  Let $(\mathcal{H}, R)$ be a left bialgebroid. An object $M$ in ${}^\mathcal{H}_\mathcal{H}\mathcal{M} \cap {}^\mathcal{H}\!\mathcal{M}^{\phantom{\mathcal{H}}}_\mathcal{H}$ is called a
  \textit{left covariant bimodule} over $\cH$ if
  \begin{enumerate}
  \compactlist{50}
\item
  $r m r'= r \lact m \bract r'$ for all $r, r' \in R$;
\item
  the left and right $\mathcal{H}$-action commute.
\end{enumerate}
In this situation, we write $M\in{}^\mathcal{H}_\mathcal{H}\mathcal{M}^{\phantom{\mathcal{H}}}_\mathcal{H}$.
\end{definition}

\noindent In particular, for every object $M$ in ${}^\mathcal{H}_\mathcal{H}\mathcal{M}^{\phantom{\mathcal{H}}}_\mathcal{H}$, the identity
$$
\lambda_M(h m h')=\Delta(h)\lambda_M(m)\Delta(h')
$$
holds as an equation in $\cH \times_R M$.
Observe that part (ii) in this definition automatically resolves the issue raised in Remark~\ref{roslagen} resp.\ Definition \ref{genzano}\,(i)\,(b) of commuting actions of $R$ and $\cH$.

{\em Right} $\mathcal{H}$-coactions and the corresponding {\em right} covariant bimodules are defined similarly.
Interestingly enough, Theorem \ref{thm:HopfModules} can be adapted to the situation of Definition \ref{laparolacheciunisce}, that is, to objects in ${}^\mathcal{H}_\mathcal{H}\mathcal{M}^{\phantom{\mathcal{H}}}_\mathcal{H}$. This is essentially due to the fact that the module of coinvariants for such an $M$ becomes a left $\cH$-module by means of an adjoint action, which turns out to be well-defined on $ {}^{\mathrm{co}\mathcal{H}\!}M$ while this is not so on $M$ itself. More precisely, we have:

\begin{theorem}
  \label{rhodia}
  Let $\cH$ be a right Hopf algebroid such that $\cH_\ract$ is flat
  over $R$, and
  let $M\in {}^\mathcal{H}_\mathcal{H}\mathcal{M}^{\phantom{^\mathcal{H}}}_\mathcal{H}$ be a left covariant $\cH$-bimodule.
  \begin{enumerate}
    \compactlist{99}
    \item
  Then the module $ {}^{\mathrm{co}\mathcal{H}\!}M$ of coinvariants becomes a left
  $\RRe$-module by means of 
  $$
  \RRe \otimes {}^{\mathrm{co}\mathcal{H}\!}M \to {}^{\mathrm{co}\mathcal{H}\!}M, \quad (r,r') \otimes m \mapsto  r' \blact m \ract r,
  $$
  which extends to a well-defined left $\cH$-action
  \begin{equation}
    \label{sogei1}
\due \cH \blact \bract \otimes_{\RRe} {}^{\mathrm{co}\mathcal{H}\!}M \to {}^{\mathrm{co}\mathcal{H}\!}M, \quad h \otimes_{\RRe} m \mapsto h_{\smap } m h_{\smam }.
  \end{equation}
\item
  The categorical equivalence from Theorem \ref{thm:HopfModules}\,(iii) by means of the adjoint functors carries over to this case with respect to the categories ${}^\mathcal{H}_\mathcal{H}\mathcal{M}^{\phantom{}}_\mathcal{H}$ and ${}_\mathcal{H}\mathcal{M}$. In particular, the
  map
   \begin{equation}
    \label{sogei3}
\xi \colon  \mathcal{H}_\ract  \otimes_R {}^{\mathrm{co}\mathcal{H}\!}M \to  M, \quad h\otimes_R {m}\mapsto {m}h, 
\end{equation}
  from Theorem \ref{thm:HopfModules} can be promoted to an isomorphism in 
  the category ${}^\mathcal{H}_\mathcal{H}\mathcal{M}^{\phantom{^\mathcal{H}}}_\mathcal{H}$ of left covariant bimodules over $\cH$,
with inverse
  $
m\mapsto m_{(-1)}\otimes_R m_{(0)\smap } m_{(0)\smam }
$.
\end{enumerate}
\end{theorem}
\begin{proof}
  Let us first show that the map \eqref{sogei1} is well-defined indeed over the Sweedler presentation \eqref{Tch1}: for $m \in {}^{\mathrm{co}\mathcal{H}\!}M \subset M$ and $r \in R$, one has, using the right $R$-action on $M$,
  \begin{equation}
    \label{gaudia}
\gl_M(mr) = \gl_M(m \bract r) = m_{(-1)} s(r) \otimes_R m_{(0)} = s(r) \otimes_R m 
  \end{equation}
since $M \in {}_\mathcal{H}\mathcal{M}_\mathcal{H}$ or directly from the $R$-linearity \eqref{ha} of $\gl_M$; the last step follows from $m$ being coinvariant, {\em cf.}\ Eq.\ \eqref{waterman}. Applying $\varepsilon \otimes_R M$ to both sides and identifying $R \otimes_R M \cong M$ by means of $r \otimes_R m \mapsto rm$ yields $mr = rm$ and since $M \in {}^\mathcal{H}_\mathcal{H}\mathcal{M}^{\phantom{^\mathcal{H}}}_\mathcal{H}$, Definition \ref{laparolacheciunisce} then implies
$$
mr = m \bract r = r \lact m = rm
$$
for
$m \in {}^{\mathrm{co}\mathcal{H}\!}M$,
which gives the well-definedness of the map \eqref{sogei1}. Observe that Eq.~\eqref{gaudia} also implies that $mr = rm \notin {}^{\mathrm{co}\mathcal{H}\!}M$ for $m \in {}^{\mathrm{co}\mathcal{H}\!}M$. Next, we have to show that the map \eqref{sogei1} lands in ${}^{\mathrm{co}\mathcal{H}\!}M$. Indeed, for $m \in {}^{\mathrm{co}\mathcal{H}\!}M$ and $h \in \cH$, we have by using the fact that $M$ is a left covariant $\cH$-bimodule:
\begin{equation*}
\begin{split}
  \gl_M(h_{\smap } m h_{\smam }) &= h_{\smap (1)} m_{(-1)} h_{\smam (1)} \otimes_R h_{\smap (2)} m_{(0)} h_{\smam (2)}
  \\
&  = h_{\smap (1)} h_{\smam (1)} \otimes_R h_{\smap (2)} m h_{\smam (2)}
  \\
&  = h_{\smap \smap (1)} h_{\smap \smam } \otimes_R h_{\smap \smap (2)} m h_{\smam }
  \\
&  = 1 \otimes_R h_{\smap } m h_{\smam },
\end{split}
  \end{equation*}
using the identities \eqref{Tch5} and \eqref{Tch2} in the third and fourth step.
That the so-defined map \eqref{sogei1} then defines a left $\cH$-action is immediate from \eqref{Tch6} and the statement regarding the $R$-bimodule structure follows from \eqref{Tch9}.
Thanks to Theorem \ref{thm:HopfModules}, the only things left to show are that
$\cH \otimes_R {}^{\mathrm{co}\mathcal{H}\!}M$ is a left-left Hopf module, {\em i.e.}, an object in ${}^\mathcal{H}_\mathcal{H}\mathcal{M}$, and that the isomorphism $\xi$ from \eqref{sogei3} is a morphism of left $\cH$-modules.
As for the first, the left $\cH$-coaction in question is simply given by
$
\gl_{\cH \otimes_R {}^{\mathrm{co}\mathcal{H}\!}M} := \gD \otimes_R {}^{\mathrm{co}\mathcal{H}\!}M,
$
while the left $\cH$-action is by the diagonal action on the tensor product,
that is, 
$
h(g \otimes_R m) :=
h_{(1)}g \otimes_R h_{(2)\smap } m h_{(2)\smam }.
$
From this, the requirements in Definition \ref{genzano}\,(ii) for a left-left Hopf module structure on
$
\cH \otimes_R {}^{\mathrm{co}\mathcal{H}\!}M
$
are immediate. As for the left $\cH$-linearity of $\xi$ in \eqref{sogei3}, we simply compute
$$
\xi\big(h(g \otimes_R m)\big) = 
\xi\big(h_{(1)}g \otimes_R h_{(2)\smap } m h_{(2)\smam }\big)
=  h_{(2)\smap } m h_{(2)\smam } h_{(1)}g
= hmg = h\xi(g \otimes_R m),
$$
where we used \eqref{Tch3} in the third step.
\end{proof}

With all the technical preparations at hand, in the next section we can fully dedicate our attention to one of our main objects of study, that is, {\em noncommutative differential calculi}.

\section{Covariant calculi on Hopf algebroids}

\subsection{First order differential calculi on bialgebroid comodule algebras}
In this section, we introduce differential structures in the bialgebroid context. This is done in two steps: to start with, on $R$-rings $B$ we define first order differential calculi as $B$-bimodules $\Omega$, the modules of \emph{one-forms}, together with a derivation $\mathrm{d}\colon B\to\Omega$, the \emph{differential}, such that $\mathrm{d}(B)$ generates $\Omega$ as a $B$-module. Apart from working over $\Bbbk$-algebras $R$ rather than fields, this is very much in line with the current literature. However, in a second step
we consider left $\mathcal{H}$-comodule algebras $B$ over left bialgebroids $(\mathcal{H}, R)$ and
introduce first order differential calculi on the former by requiring additional compatibility with the coactions.
As a technical feature, we have to assume that the $R$-actions on $\Omega$ coincide with the ones induced by the $\mathcal{H}$-coaction. Let us therefore define first a differential structure on a ring:

\begin{definition}
  \label{FODC}
Let $B$ be an $R$-ring. Given a $B$-bimodule $\Omega$ and an $R$-bilinear map $\mathrm{d}\colon B\to\Omega$, we call the pair $(\Omega,\mathrm{d})$ a \emph{first order differential calculus on $B$} if
\begin{enumerate}
  \compactlist{50}
\item
the {\em Leibniz rule}
$$
\dd  \circ m_B = R_B \circ (\dd  \otimes_R B) + L_B \circ (B \otimes_R \dd )
$$
holds on $B \otimes_R B$, where $m_B$ denotes the ($R$-balanced) multiplication on $B$, whereas $R_B$ and $L_B$ the right resp.\ left $B$-actions on $\gG$.
On elements, this simply reads as
\begin{equation}
	\label{wasserrutsche}
	\dd (bb')=(\dd b)b'+b(\dd b')
\end{equation}
for all $b,b'\in B$;
\item
the map $B\otimes_R B\to\Omega$,\, $b\otimes_R b'\mapsto b\hspace{1pt}\dd b'$ is surjective, compactly formulated as $\Omega=B\hspace{1pt}\dd B$.
\end{enumerate}
A \textit{morphism} $(\Omega,\dd )\to(\Omega',\dd ')$
of first order differential calculi is a morphism $\phi\colon\Omega\to\Omega'$ of $B$-bimodules such that $\dd =\dd '\circ\phi$. Then, \textit{quotients} of first order differential calculi correspond to surjective morphisms.
\end{definition}

It remains to incorporate the possible coaction of a bialgebroid.
Recall that, given a left bialgebroid $(\mathcal{H},R)$, a \emph{left $\mathcal{H}$-comodule algebra} is a left $\mathcal{H}$-comodule $B$ with an $R$-ring structure $\iota\colon R\to B$ such that $\lambda_B\colon B\to\mathcal{H}\otimes_RB$ is an algebra morphism.
If we denote the respective $R$-actions on $B$ from left and right simply by juxtaposition as we usually do for the multiplication in $B$, a little thought reveals that by $R$-bilinearity of $\gl_B$ the useful properties
$$
(rb)b' = r(bb'), \qquad b(b'r) = (bb')r, \qquad (br)b'=b(rb'),
$$
hold for all $r \in R$ and $b, b' \in B$.
For example, the last identity in this line follows from
\begin{eqnarray*}
  (br)b'
  = (\gve \otimes_R B) \circ \gl_B ( (br)b' )
  &=& (\gve \otimes_R B)\big(\big((br)b'\big)_{(-1)} \otimes_R \big((br)b'\big)_{(0)} \big)
  \\
  &=&
  (\gve \otimes_R B)\big((br)_{(-1)} b'_{(-1)} \otimes_R (br)_{(0)} b'_{(0)}\big)
\\
  &\overset{\eqref{ha}}{=}&
  (\gve \otimes_R B)\big(b_{(-1)} s(r) b'_{(-1)} \otimes_R b_{(0)} b'_{(0)}\big)
\\
  &\overset{\eqref{ha}}{=}&
  (\gve \otimes_R B)\big(b_{(-1)} (r b')_{(-1)} \otimes_R b_{(0)} (rb')_{(0)}\big)
\\
&=& 
  (\gve \otimes_R B)\big(\big(b(r b')\big)_{(-1)} \otimes_R \big(b(rb')\big)_{(0)}\big)
\\
&=& 
b(rb'),
\end{eqnarray*}
where we suppressed the canonical isomorphism $R \otimes_R B \cong B$ to express counitality and used the comodule algebra property in step two and five.
This leads to a natural generalisation of Definition \ref{laparolacheciunisce}.

\begin{definition}
Let $(\mathcal{H}, R)$ be a left bialgebroid and $B$ a left $\mathcal{H}$-comodule algebra. An object $M$ in ${}^\mathcal{H}\!\mathcal{M}\cap{}_B\mathcal{M}_B$ with coaction $\lambda_M\colon M\to\mathcal{H}\otimes_RM$ is called a \textit{left $\mathcal{H}$-covariant $B$-bimodule} if
\begin{enumerate}
\compactlist{99}
\item
$r m r'=\iota(r) m \iota(r')$;
\item
$\lambda_M(b m b')=\lambda_B(b)\lambda_M(m)\lambda_B(b')$.
\end{enumerate}
{\em Morphisms} of left $\mathcal{H}$-covariant $B$-bimodules are $B$-bilinear maps which are compatible with the left $\mathcal{H}$-coactions.
The corresponding category is denoted by ${}_B^\mathcal{H}\mathcal{M}^{\phantom{}}_B$ of which there is a subcategory ${}_B^\mathcal{H}\mathcal{M}$ defined the obvious way, where we do not require the existence of a right $B$-action.
\end{definition}

This turns out to be the appropriate category for one-forms of a covariant first order differential calculus in the bialgebroid framework.

  \begin{definition}
    \label{covFODC}
For a left bialgebroid $(\mathcal{H},R)$ and a left $\mathcal{H}$-comodule algebra $B$, we call a first order differential calculus $(\Omega,\mathrm{d})$ on the $R$-ring $B$ \emph{left $\mathcal{H}$-covariant} if
\begin{enumerate}
\compactlist{50}
\item $\Omega$ is an object in
${}_B^{\cH}\cM^{\phantom{\cG}}_{\mkern -2mu B}$;
\item $\dd \colon B\to\Omega$ is a left $\mathcal{H}$-colinear map.
\end{enumerate}
A \textit{morphism}
$(\Omega,\dd )\to(\Omega',\dd ')$
of left $\mathcal{H}$-covariant first order differential calculi  is a morphism $\phi\colon\Omega\to\Omega'$ in ${}_B^{\mkern 0 mu \cH}\cM^{\phantom{\cH}}_{\mkern -2mu B}$ such that $\dd =\dd '\circ\phi$, while \textit{quotients} of left $\mathcal{H}$-covariant first order differential calculi correspond to surjective morphisms.
  \end{definition}

  For every $R$-ring $B$ there exists a first order differential calculus, the so-called \emph{universal calculus}, with the universal property that all first order differential calculi on $B$ can be realised as quotients of it. If $B$ is a left $\mathcal{H}$-comodule algebra for a left bialgebroid $(\mathcal{H},R)$,
where $\cH_\ract$ is a flat right $R$-module,
  then its universal calculus is left $\mathcal{H}$-covariant and admits the universal property for left $\mathcal{H}$-covariant first order differential calculi on $B$. This follows in complete analogy to the bialgebra context, see \cite[Prop.~1.1]{Woronowicz1989}.

  \begin{example}[Universal calculus]
    \label{freibadoderbadesee}
    For every $R$-ring $B$, the pair $(\gO^u_B,\mathrm{d}_u)$ is a first order differential calculus, where $\gO^u_B:=\ker(m_B \colon B\otimes_R B\to B)\subseteq B\otimes_R B$ is the kernel of the multiplication, naturally a  $B$-bimodule via $b(b^i\otimes_R \tilde b^i)b':=bb^i\otimes_R \tilde b^ib'$ and $R$-bimodule structure induced by the forgetful functor, while the differential reads
$$
\dd _u\colon B\to\gO^u_B,\quad b \mapsto 1\otimes_R b-b\otimes_R 1.
$$
In case $B$ is a left $\cH$-comodule algebra and, as above, $\cH_\ract$ a flat right $R$-module, $\gO^u_B$ is an object in ${}_B^{\cH}\cM^{\phantom{\cH}}_{\mkern -2mu B}$	with respect to the diagonal left $\cH$-coaction
$
\gO^u_B \to \cH \otimes_R \gO^u_B, \
b^i\otimes_R \tilde b^i\mapsto b^i_{(-1)}\tilde b^i_{(-1)}\otimes_R (b^i_{(0)}\otimes_R \tilde b^i_{(0)})$, while $\mathrm{d}_u$ turns out to be left $\mathcal{H}$-colinear.
Thus, in this case $(\gO^u_B,\mathrm{d}_u)$ is a left $\mathcal{H}$-covariant calculus.
\end{example}

As mentioned, this universal first order differential calculus admits a universal property:

\begin{proposition}
  \label{schweineschmalz}
Every first order differential calculus $(\Omega,\dd )$ on an $R$-ring $B$ is a quotient of the universal calculus of Example \ref{freibadoderbadesee}: there is a subobject $M\subseteq\gO^u_B$ in 
${}_B\cM_{\mkern -2mu B}$
along with an isomorphism
$\phi\colon\gO^u_B/M\xrightarrow{\cong}\Omega$ in
${}_B\cM_{\mkern -2mu B}$
such that
$$
\dd =\phi\circ\mathrm{pr}\circ \dd _u,
$$
where $\mathrm{pr}\colon\gO^u_B\to\gO^u_B/M$ denotes the canonical projection. 
If, in addition, $(\Omega,\mathrm{d})$ is left $\cH$-covariant, then $M$ can be chosen in ${}_B^{\cH}\cM^{\phantom{\cG}}_{\mkern -2mu B}$ and $\phi$ becomes an isomorphism in this category.
\end{proposition}

\begin{proof}
Let $(\Omega, \dd )$ be a first order differential calculus on $B$. Then there is an $R$-bilinear map
\begin{equation*}
	\phi\colon\gO^u_B \to \Omega, \quad b^i\otimes_R \tilde b^i \mapsto b^i \dd \tilde b^i.
\end{equation*}
Clearly, $\phi$ is left $B$-linear but also right $B$-linear since 
$$
\phi\big((b^i\otimes_R \tilde b^i) b'\big)
=\phi(b^i\otimes_R \tilde b^i b')
=b^i\dd (\tilde b^i b')
=b^i \dd(\tilde b^i)b' + b^i \tilde b^i \dd b'
=\phi(b^i\otimes_R \tilde b^i)b'
$$
by the Leibniz rule \eqref{wasserrutsche} and the property $b^i\tilde b^i=0$, by the very definition of $\gO^u_B$.
Let us show that $\phi$ is a surjection. For an arbitrary element $b^i \dd b_i \in \Omega$ with $b^i, b_i\in B$, we have
$
b^i\otimes_R b_i - b^i b_i \otimes_R 1 \in\gO^u_B
$
as well as
$
\phi(b^i\otimes_R b_i- b^ib_i \otimes_R 1 ) = b^i \dd b_i,
$
that is, $\phi\colon\gO^u_B\to\Omega$ is a surjection of $B$-bimodules.
Moreover, it intertwines the differentials since
$$
\phi(\dd_ub)
=\phi (1 \otimes_R b - b \otimes_R 1)
=
\dd b+0
$$
for all $b\in B$. Thus, $(\Omega, \dd)$ is a quotient of $(\gO^u_B, \dd_u)$.

\noindent In the left $\mathcal{H}$-covariant case, $\phi$ is left $\cH$-colinear, as simply seen by
\begin{align*}
	\gl_\Omega \big(\phi(b^i\otimes_R \tilde b^i)\big)
	=\gl_\Omega (b^i \dd\tilde b^i)
	=b^i_{(-1)}\tilde b^i_{(-1)}\otimes_R b^i_{(0)} \dd b^i_{(0)}
	=(\cH \otimes_R \phi) \big(\gl_{\gO^u_B}(b^i\otimes_R \tilde b^i)\big),
\end{align*}
using the left $\cH$-covariance of $(\Omega,\dd)$.
\end{proof}

Note that Proposition \ref{schweineschmalz} recovers the universal property of the universal calculus of $\Bbbk$-algebras $A$ and $H$-comodule algebras for a $\Bbbk$-Hopf algebra $H$ as proven in \cite[Prop.~1.1]{Woronowicz1989}.

\subsection{Woronowicz classification for Hopf algebroids}\label{sec:Wor}
In this section, we classify covariant calculi on Hopf algebroids, using an analogue of the fundamental theorem of Hopf modules adapted to Hopf algebroids.
To start with, we obtain differential structures on left bialgebroids $(\mathcal{H},R)$ from Definitions \ref{FODC} and \ref{covFODC} by considering the case $B=\mathcal{H}$ with coaction given by the coproduct; given the importance of this definition we spell it out in detail. To this end,
recall the notation from Eqs.~\eqref{soedermalm}, which expresses the presence of forgetful functors from left resp.\ right $\cH$-modules to $\RRe$-modules.

\begin{definition}
  \label{def:FODC}
We will refer to a pair $(\Omega,\dd )$ as \textit{first order differential calculus} on a left bialgebroid $(\cH, R)$ if
	\begin{enumerate}
		\compactlist{50}
		\item
		$\Omega\in{}_\mathcal{H}\mathcal{M}_\mathcal{H}$;
		\item
		$\dd \colon \due {\mathcal{H}} \lact \bract \to \due \Omega \lact \bract$ is $\RRe$-linear such that the Leibniz rule
		\begin{equation*}
			\dd \circ m_\mathcal{H}
			=R_\mathcal{H}\circ(\dd \otimes_R\mathcal{H})
			+L_\mathcal{H}\circ(\mathcal{H}\otimes_R\dd )
		\end{equation*}
		holds as an equation of $\RRe$-linear maps
                $\cH_\bract \otimes_R \due \cH \lact {} \to \Omega$;
	      \item the map
                $\cH_\bract \otimes_R \due \cH \lact {}
                \to \Omega$, $h\otimes_Rg\mapsto h {\mkern 1mu} \dd g$ is surjective.
	\end{enumerate}
	If there is, in addition, a left $\mathcal{H}$-coaction $\lambda_\Omega\colon \due \Omega \lact {} \to \mathcal{H}_\ract \otimes_R \due \Omega \lact {}$ on $\Omega$ such that $\Omega\in{}_\mathcal{H}^\mathcal{H}\mathcal{M}^{\phantom{\cH}}_\mathcal{H}$ and
	\begin{equation}
          \label{dleftcol}
		\lambda_\Omega\circ\dd 
		=(\mathcal{H}\otimes_R\dd )\circ\Delta
                	\end{equation}
holds as maps $\mathcal{H}\to\mathcal{H}\otimes_R\Omega$,
        the first order differential calculus on $\mathcal{H}$ will be called \textit{left covariant}.
	A \textit{morphism} $F\colon(\Omega,\dd )\to(\Omega',\dd ')$ of first order differential calculi on $\mathcal{H}$
        is a morphism $F\colon\Omega\to\Omega'$ in ${}_\mathcal{H}\mathcal{M}_\mathcal{H}$ such that $F\circ\dd =\dd '$. Such a morphism is called a morphism of covariant first order differential calculi if it is, in addition, left $\mathcal{H}$-colinear. 
\end{definition}

\begin{remark}
Observe that the $\RRe$-linearity in Definition \ref{def:FODC} refers to the left resp.\ right $R$-action coming from the source map only, while in general $\dd$ is not $\RRe$-linear with respect to the left and right $R$-actions arising from the target map. As a consequence, one has $\dd(s(r)) = 0$ in contrast to $\dd(t(r)) \neq 0$, which means that in general the differential $\dd$ is not trivial on the base algebra $R$. 
  \end{remark}

If $\gO$ is a covariant calculus and hence, in particular, a left $\cH$-comodule, as in \eqref{duracell} we can define a corresponding Galois map by
\begin{equation}
  \label{alphaOmega}
	\alpha_\Omega \colon \Omega\otimes_R\mathcal{H}\to \mathcal{H}\otimes_R\Omega, \quad \omega\otimes_R h' \mapsto \omega_{(-1)} h' \otimes_R \omega_{(0)}.
\end{equation}
As in \S\ref{mittagleffler}, if this map invertible, let us write
$\omega_{\smap }\otimes_R \omega_{\smam }:= \alpha_\Omega^{-1}(1\otimes_R \omega)$ for the translation map and therefore
$(\alpha_\Omega)^{-1}(h' \otimes_R\omega)=\omega_{\smap }\otimes_R\omega_{\smam }h'$;
or, by using the surjectivity of Definition \ref{def:FODC} (iii) and therefore writing $\go = h {\mkern 1mu} \dd g$, one can write $(\alpha_\Omega)^{-1}(h' \otimes_R  h {\mkern 1mu} \dd g)
=(h {\mkern 1mu} \dd g)_{\smap } \otimes_R (h {\mkern 1mu} \dd g)_{\smam }h'$.

Let us assume now that the left bialgebroid $(\cH, R)$ happens to be, in addition, a right Hopf algebroid. Then from 
Lemma \ref{minestrone}, along with Eqs.~\eqref{fabercastell0} and \eqref{dleftcol},
it follows 
that 
\begin{equation*}
  (h {\mkern 1mu} \dd g)_{\smap }\otimes_R(h {\mkern 1mu} \dd g)_{\smam }
  =h_{\smap }\dd g_{\smap } \otimes_Rg_{\smam }h_{\smam },
\end{equation*}
which we use to define the \textit{Maurer-Cartan form} on a covariant calculus by
\begin{equation}
  \label{MaurerCartan}
\varpi\colon\mathcal{H}^+\to{}^{\mathrm{co}\mathcal{H}}\Omega, \quad
  h \mapsto (\dd h)_{\smap } (\dd h)_{\smam }
	=\dd (h_{\smap }) \mkern 2mu h_{\smam }
\end{equation}
for all $h\in\mathcal{H}^+:=\ker\varepsilon\subset\mathcal{H}$, where
as in \eqref{waterman}
\begin{equation*}
  {}^{\mathrm{co}\mathcal{H}}\Omega
  := \big\{\omega\in\Omega \mid \lambda_\Omega(\omega)=1\otimes_R\omega\big\}
\end{equation*}
denotes the $R$-bimodule of (left) coinvariants.
By Theorem \ref{thm:HopfModules}\,(i), the image of $\varpi$ is in ${}^{\mathrm{co}\mathcal{H}}\Omega$, indeed.
Note that
\begin{align*}
	(\dd h)_{\smap } (\dd h)_{\smam }
	=\dd (h_{\smap }) \mkern 2mu h_{\smam }
	=\dd (h_{\smap }h_{\smam })-h_{\smap }\dd h_{\smam }
	=\dd (t\varepsilon(h))-h_{\smap }\dd h_{\smam }
\end{align*}
for all $h\in\mathcal{H}$.

\begin{lemma}
	\label{lem:MaurerCartan}
	The Maurer-Cartan form of a left covariant first order differential calculus $(\Omega,\dd )$ on a right Hopf algebroid $(\cH, R)$ is well-defined, left $\mathcal{H}$-linear, and surjective. Thus, $I:=\ker\varpi\subseteq\mathcal{H}^+$ is a left $\mathcal{H}$-ideal such that ${}^{\mathrm{co}\mathcal{H}}\Omega\cong\mathcal{H}^+/I$.
        \end{lemma}

\begin{proof}
	The kernel $I=\ker\varpi$ is a left ideal in $\mathcal{H}$ since for $h\in \mathcal{H}$ and $h'\in I$, we have
	\begin{align*}
		\varpi(hh')
		&=\dd ((hh')_{\smap }) \mkern 2mu (hh')_{\smam }
                \\
		&=\dd (h_{\smap }h'_{\smap }) \mkern 2mu (h'_{\smam }h_{\smam })
                \\
		&=\dd (h_{\smap }) \mkern 2mu (h'_{\smap }h'_{\smam }h_{\smam })
		+ h_{\smap } \dd (h'_{\smap }) \mkern 2mu (h'_{\smam }h_{\smam })
                \\
		&=\dd (h_{\smap })  (t\varepsilon(h') h_{\smam })
		+h_{\smap } \varpi(h')  h_{\smam }\\
		&=0+0.
	\end{align*}
	Furthermore, $\varpi$ is surjective because
	for every $\omega\in{}^{\mathrm{co} \cH}\Omega$, 
	we have $\omega=\omega_{\smap } \omega_{\smam }$ since
	$$
	\omega_{\smap }\omega_{\smam }
	= \omega_{(0)\smap } \mkern 2mu (\omega_{(0)\smam }\omega_{(-1)})
	=\omega,
	$$
	where we used the coinvariance $1\otimes_R\omega=\omega_{(-1)}\otimes_R\omega_{(0)}$
	along with the property \eqref{Mch3} of Hopf-Galois comodules.
	Then, for $\omega=h\dd g \in{}^{\mathrm{co}\mathcal{H}}\Omega$ we have
	\begin{align*}
		\omega=\omega_{\smap }  \omega_{\smam }
		=h_{\smap }\dd (g_{\smap })  \mkern 2mu (g_{\smam }h_{\smam })
		=h_{\smap } \mkern 2mu \varpi(g-s\varepsilon(g)) \mkern 2mu h_{\smam }
		=\varpi\big(h(g-s\varepsilon(g))\big),
	\end{align*}
	where in the third equation we used that
	$
        \dd (s\varepsilon(g)_{\smap }) \mkern 2mu
        s\varepsilon(g)_{\smam }
	=\dd (1) \mkern 2mu t\varepsilon(g)=0
        $
        by \eqref{Tch9}, 
	and that $g-s\varepsilon(g)\in\ker\varepsilon$, which is
	the domain of $\varpi$.
	It follows that ${}^{\mathrm{co}\mathcal{H}}\Omega\cong\mathcal{H}^+/I$ as left
	$\mathcal{H}$-modules.
\end{proof}

Invoking Lemma \ref{lem:MaurerCartan} together with Theorem \ref{thm:HopfModules}, given a left covariant first order differential calculus $(\Omega,\dd )$ on a right Hopf algebroid $\mathcal{H}$, where ${\mathcal{H}}_\ract$ is $R$-flat, it follows that there is a left $\mathcal{H}$-ideal $I =\ker\varpi\subseteq\mathcal{H}^+$ such that the left covariant $\mathcal{H}$-bimodule
\begin{equation*}
	\mathcal{H}\otimes_R{}^{\mathrm{co}\mathcal{H}}\Omega
	\cong\mathcal{H}\otimes_R\mathcal{H}^+/I
\end{equation*}
is isomorphic to $\Omega$.
Moreover, the $R$-bilinear map $\dd '\colon\mathcal{H}\to\mathcal{H}\otimes_R\mathcal{H}^+/I$ defined by
\begin{equation*}
	\dd ' h :=(\mathcal{H}\otimes_R\pi)(\Delta h - h\otimes_R1),
\end{equation*}
where $\pi\colon\mathcal{H}^+\to\mathcal{H}^+/I$ denotes the projection, turns $(\mathcal{H}\otimes_R{}^{\mathrm{co}\mathcal{H}}\Omega,\dd ')$ into a left covariant first order differential calculus on $\mathcal{H}$, which is isomorphic to $(\Omega,\dd )$ as a calculus.

The following classification theorem states that {\em all} left covariant first order differential calculi on $\mathcal{H}$ arise this way.

\begin{theorem}
  \label{thm:Wor}
	Given a right Hopf algebroid $\mathcal{H}$ such that ${\mathcal{H}}_\ract$ is $R$-flat, there is a bijective correspondence
	\begin{equation}
          \label{WorCor}
		\begin{Bmatrix}
		  \text{left covariant first order}
                  \\
			\text{differential calculi on }\mathcal{H}
		\end{Bmatrix}\xleftrightarrow{1:1}
		\begin{Bmatrix}
		  \text{left ideals}
                  \\
			I\subseteq\mathcal{H}^+
		\end{Bmatrix}.
	\end{equation}
\end{theorem}

\begin{proof}
	Given a left covariant first order differential calculus $(\Omega,\dd )$ on a right Hopf algebroid $(\mathcal{H}, R)$, we obtain a left $\mathcal{H}$-ideal $I=\ker\varpi\subseteq\mathcal{H}^+$ according to Lemma \ref{lem:MaurerCartan}. On the other hand, given a left $\mathcal{H}$-ideal $I\subseteq\mathcal{H}^+$, consider the $R$-bimodule
	\begin{equation}
          \label{WorOmega}
		\Omega:={\mathcal{H}}_\ract\otimes_R\due {\mathcal{H}^+/I} \lact {}
	\end{equation}
	with $R$-actions $r (h\otimes_R\pi(g)) r':= r \lact h \bract r'\otimes_R\pi(g)$
	for all $r,r'\in R$, $h\in\mathcal{H}$, and $g\in\mathcal{H}^+$.
	Furthermore, consider the $R$-bilinear map $\dd \colon\due {\mathcal{H}} \lact\bract\to\Omega$ defined by
	\begin{equation}
          \label{WorD}
		\dd h:=(\mathcal{H}\otimes_R\pi)(\Delta h-h\otimes_R1),
	\end{equation}
	where $\pi\colon\mathcal{H}^+\to\mathcal{H}^+/I$ denotes the quotient map. Note that the map \eqref{WorD} is well-defined since $\Delta h-h\otimes_R1 \in\mathcal{H}\otimes_R\mathcal{H}^+$, as one easily verifies by applying $\mathcal{H}\otimes_R\varepsilon$ to the latter, using $R$-flatness. Let us prove that $(\Omega,\dd )$ is a left covariant first order differential calculus on $\mathcal{H}$.
        First of all, the module \eqref{WorOmega} is an $\mathcal{H}$-bimodule with respect to the $\mathcal{H}$-actions
	\begin{equation}
          \label{yellowcup}
	  h \mkern 1mu
          (g\otimes_R[g']):=h_{(1)}g\otimes_R[h_{(2)}g'],\qquad\qquad
		(g\otimes_R[g']) \mkern 1mu h:=gh\otimes_R[g']
	\end{equation}
	for all $h,g\in\mathcal{H}$, $g'\in\mathcal{H}^+$, writing
        $[g'] := \pi(g')$.
        Note that the left $\mathcal{H}$-action is well-defined because $I$ is a left $\mathcal{H}$-ideal in $\mathcal{H}^+$. Furthermore, on $\gO$ there is the left $\mathcal{H}$-coaction $\lambda_\Omega:=\Delta\otimes_R\mathcal{H}^+/I\colon\Omega\to\mathcal{H}\otimes_R\Omega$, which is coassociative and counital as $\Delta$ is. Then $\lambda_\Omega$ is compatible with the $\mathcal{H}$-bimodule structure  since
	\begin{align*}
		\lambda_\Omega(h \mkern 1mu (g\otimes_R[g']))
		&=\lambda_\Omega(h_{(1)}g\otimes_R[h_{(2)}g'])\\
		&=h_{(1)}g_{(1)}\otimes_Rh_{(2)}g_{(2)}\otimes_R[h_{(3)}g']\\
		&=h_{(1)}g_{(1)}\otimes_Rh_{(2)} \mkern 1mu (g_{(2)}\otimes_R[g'])
		\, = \, h \mkern 2mu \lambda_\Omega(g\otimes_R[g'])
	\end{align*}
	and
	\begin{align*}
		\lambda_\Omega((g\otimes_R[g']) \mkern 1mu h)
		&=\lambda_\Omega(gh\otimes_R[g'])\\
		&=g_{(1)}h_{(1)}\otimes_Rg_{(2)}h_{(2)}\otimes_R[g']\\
		&=g_{(1)}h_{(1)}\otimes_R(g_{(2)}\otimes_R[g']) \mkern 1mu h_{(2)}
		\, = \, \lambda_\Omega(g\otimes_R[g']) \mkern 1mu h
	\end{align*}
	for all $h,g\in\mathcal{H}$ and $g'\in\mathcal{H}^+$. Thus, $\Omega$ is a left $\mathcal{H}$-covariant $\mathcal{H}$-bimodule. Moreover, $\dd $ satisfies 
	\begin{align*}
		\dd(h) \mkern 1mu g+h \mkern 1mu \dd g
		&=(\mathcal{H}\otimes_R\pi)(\Delta h -h\otimes_R1)\mkern 1mu g
		+h  \mkern 1mu (\mathcal{H}\otimes_R\pi)(\Delta g -g\otimes_R1)
                \\
		&=(\mathcal{H}\otimes_R\pi)(h_{(1)} g\otimes_R h_{(2)} - hg\otimes_R1
		+\Delta(hg)- h_{(1)} g\otimes_R h_{(2)})
                \\
		&=(\mathcal{H}\otimes_R\pi)(\Delta(hg)-hg\otimes_R1)\\
		&=\dd (hg),
	\end{align*}
        that is, the Leibniz rule
	for all $h,g\in\mathcal{H}$. Finally, $(\Omega,\dd )$ is surjective because we can recover any element $h \otimes_R [g]\in\Omega$ with $h\in\mathcal{H}$ and $g\in\mathcal{H}^+$ by
	\begin{align*}
		\dd (g_{\smap }g_{\smam }h)
		-g_{\smap }\dd (g_{\smam }h)
		&=\dd (g_{\smap })  \mkern 1mu (g_{\smam }h)
                \\
		&=(\mathcal{H}\otimes_R\pi)(g_{\smap (1)}g_{\smam }h\otimes_Rg_{\smap (2)}
		-g_{\smap }g_{\smam }h\otimes_R1)\\
		&=(\mathcal{H}\otimes_R\pi)(h\otimes_Rg
		-h\ract \varepsilon(g)\otimes_R1)\\
		&=h\otimes_R[g],
	\end{align*}
	using the Leibniz rule and Eqs.~\eqref{Tch3} \& \eqref{Tch8}.
	It remains to prove that the correspondence \eqref{WorCor} is one-to-one.
	Starting with a left covariant first order differential calculus $(\Omega,\dd )$ on $\mathcal{H}$, we obtain a
	left $\mathcal{H}$-ideal $I=\ker\varpi\subseteq\mathcal{H}^+$ and we construct $(\Omega',\dd ')$ as before. Let us show that $(\Omega,\dd )$ and $(\Omega',\dd ')$ are isomorphic as left covariant first order differential calculi: by the fundamental theorem of Hopf modules, Theorem \ref{rhodia}, there is an isomorphism
	\begin{equation*}
	  \begin{split}	\Xi\colon\Omega\xrightarrow{\cong}\mathcal{H}\otimes_R{}^{\mathrm{co}\mathcal{H}}\Omega,
            \quad
			\omega\mapsto\omega_{(-1)}\otimes_R\omega_{(0)\smap }  \mkern 1mu \omega_{(0)\smam }
		\end{split}
	\end{equation*}
	in ${}^\mathcal{H}_\mathcal{H}\mathcal{M}^{\phantom{\cH}}_\mathcal{H}$.
	This isomorphism intertwines the differentials and is thus an isomorphism of left covariant first order differential calculi since for all $h\in\mathcal{H}$,
	\begin{eqnarray*}
		\Xi(\dd h)
		&= & \dd (h)_{(-1)}\otimes_R\dd (h)_{(0)\smap }  \mkern 1mu \dd (h)_{(0)\smam }
                \\
	        &\overset{\eqref{dleftcol}}{=}&
                h_{(1)}\otimes_R\dd (h_{(2)})_{\smap }  \mkern 1mu \dd (h_{(2)})_{\smam }
                \\
		&\overset{\eqref{MaurerCartan}}{=}&
                h_{(1)}\otimes_R\varpi(h_{(2)}-s \varepsilon(h_{(2)}))
                \\
                &=&                (\mathcal{H}\otimes_R\varpi)(\Delta h-h_{(1)}\otimes_R s\varepsilon(h_{(2)}))
                \\
		&=&
                (\mathcal{H}\otimes_R\varpi)(\Delta h -h\otimes_R1)
                \\
		&=&
                \dd' h,
	\end{eqnarray*}
	        using the Maurer-Cartan form \eqref{MaurerCartan}, which induces an isomorphism ${}^{\mathrm{co}\mathcal{H}}\Omega\cong\mathcal{H}^+/I$.
	        On the other hand, starting from a left $\mathcal{H}$-ideal $I\subseteq\mathcal{H}^+$, constructing $(\Omega,\dd )$ and then $I':=\ker\varpi\subseteq\mathcal{H}^+$, let us show that $I=I'$. Recall that
                $\Omega:=\mathcal{H}\otimes_R\mathcal{H}^+/I$, by construction;
                thus, for an arbitrary $h\in\mathcal{H}^+$, we obtain
	\begin{align*}
		\varpi(h)
		&=\dd (h_{\smap })  \mkern 1mu h_{\smam }\\
		&=(\mathcal{H}\otimes_R\pi)(\Delta h_{\smap } -h_{\smap }\otimes_R1)  \mkern 1mu h_{\smam }\\
		&=(\mathcal{H}\otimes_R\pi)(h_{\smap (1)}h_{\smam }\otimes_Rh_{\smap (2)}-h_{\smap }h_{\smam }\otimes_R1)\\
		&=(\mathcal{H}\otimes_R\pi)(1\otimes_Rh-t \varepsilon(h) \otimes_R1)\\
		&=1\otimes_R\pi(h),
	\end{align*}
	where we used \eqref{yellowcup} in the third equation and \eqref{Tch3} \& \eqref{Tch8} in the fourth. This shows that $I=\ker\pi=\ker\varpi=I'$, which concludes the proof.
\end{proof}

\subsection{The Ehresmann-Schauenburg Hopf algebroid calculus}
\label{sec:EScalc}
Let us now apply our results on calculi to the Ehresmann-Schauenburg Hopf algebroid $(\cH,B)$ based on left coinvariants $\mathcal{H}={}^{\mathrm{co}H}\!(A\otimes A)$ for a left $H$-Galois extension $B={}^{\mathrm{co}H}\!A\subseteq A$, discussed in Remark \ref{leftpeft} in \S\ref{gewittrigheute}.
Since we proved there 
that this constitutes, in particular, a right Hopf algebroid, we can apply Theorem \ref{thm:Wor} to describe the covariant first order differential calculi on $\mathcal{H}$ in terms of left $\mathcal{H}$-ideals $\mathcal{H}^+ =\ker\varepsilon$. Given the particular form of the counit \eqref{blauetinte} in this case, this reveals a connection between the covariant differential calculi on $\mathcal{H}$ and left $H$-covariant first order differential calculi on $A$. As in \S\ref{gewittrigheute}, in the following $H$ is understood as a Hopf algebra over a field $\Bbbk$.

\begin{theorem}
  \label{prop:ES}
	Let $B={}^{\mathrm{co}H}\!A\subseteq A$ be a faithfully flat Hopf-Galois extension and consider the Ehresmann-Schauenburg Hopf algebroid $\mathcal{H}={}^{\mathrm{co}H}\!(A\otimes A)$ over $B$. Then
	\begin{equation*}
		\mathcal{H}^+ = \ker\varepsilon = {}^{\mathrm{co}H}\!(\gO^u_A),
	\end{equation*}
	that is, the kernel of the counit coincides with the left
        $H$-coinvariant elements of the universal first order differential calculus on $A$. Moreover, the correspondence
	\begin{equation}
		\label{1:1ES}
		\begin{Bmatrix}
			\text{left covariant first order}\\
			\text{differential calculi on }\mathcal{H}
		\end{Bmatrix}\xleftrightarrow{1:1}
		\begin{Bmatrix}
		  \text{left $H$-covariant first order}
                  \\
			\text{differential calculi on } A
		\end{Bmatrix}
	\end{equation}
        is one-to-one.
	Explicitly,
	\begin{enumerate}
		\compactlist{99}
		\item
		  given a left covariant first order differential calculus $\gO(\mathcal{H})\cong\mathcal{H}\otimes_B\mathcal{H}^+/I$ on $\mathcal{H}$, with left $\mathcal{H}$-ideal $I\subseteq\mathcal{H}^+$, we obtain a
                  left $H$-covariant first order differential calculus 
		$$
		\gO(A):=\gO^u_A/J
		$$ 
		on $A$, where $J:=AI\subseteq\gO^u_A$;
		\item
		given a left $H$-covariant first order differential calculus $\gO(A) \cong \gO^u_A/J$ on $A$, with an $A$-subbimodule and left $H$-coideal $J\subseteq\gO^u_A$,
		we obtain a left covariant first order differential calculus
		$$
		\gO(\mathcal{H}):=\mathcal{H}\otimes_B\mathcal{H}^+/I
		$$ 
		on $\mathcal{H}$, where $I:=
                {}^{\mathrm{co}H}\!J \subseteq\mathcal{H}^+$.
	\end{enumerate}
\end{theorem}

\begin{proof}
  Recall first that $\mathcal{H}$ is a right Hopf algebroid over $B$ according to Remark \ref{leftpeft}
  with counit
  $\varepsilon=m|_{{}^{{\rm co} H} \! (A \otimes A)} \colon \mathcal{H} \to B$,
where $m\colon A\otimes A\to A$ denotes the product,
  from which we obtain
	$$
	\cH^+ = \ker m|_{{}^{{\rm co} H} \! (A \otimes A)} = {}^{{\rm co} H}\! (\ker m) = {}^{{\rm co} H}\! (\gO^u_A),
	$$
                where the second equation follows from the fact that 
	        $\ker m|_{{}^{{\rm co} H}(A \otimes A)} = {}^{{\rm co} H}\!(A \otimes A)
                \cap \ker m$,
	and the third equation follows by the definition of $\gO^u_A=\ker m$. To prove the one-to-one correspondence \eqref{1:1ES}, let us verify the constructions (i) and (ii) above.
	
	According to Theorem~\ref{thm:Wor},
        for
        a left covariant first order differential calculus $\gO(\mathcal{H})$ on $\mathcal{H}$,
        there is a left $\mathcal{H}$-ideal $I\subseteq\mathcal{H}^+$ such that $\gO(\mathcal{H})\cong\mathcal{H}\otimes_B\mathcal{H}^+/I$ is an isomorphism of calculi. Then the left $A$-module
        $J := AI \subseteq\gO^u_A$,
        generated by $I$ via the left $A$-action on the first tensor factor,
        is an $A$-subbimodule and a left $H$-coideal. In fact, $J\subseteq\gO^u_A=\ker m$ as
        $I\subseteq\mathcal{H}^+ = {}^{\mathrm{co}H}\!(\gO^u_A)
        \subseteq\gO^u_A$,
        and $J=AI$ is by definition closed under the left $A$-action but also closed under the right $A$-action as we show now.
       To this end, recall first (the left $H$-Galois variant of) Schneider's equivalence \cite{Schneider}: given a faithfully flat Hopf-Galois extension $B = {}^{\mathrm{co}H}\!A \subseteq A$,
            where $H$ has an invertible antipode and $\Bbbk$ is a field, the functors ${}^{{\rm co} H}(-)$ and $A\otimes_B-$ provide an equivalence between the categories ${}_A^H\mathcal{M}$ and ${}_B\mathcal{M}$, of which we will need in a moment that the inverse of the counit $\epsilon^{-1}_N \colon N \to A \otimes_B {}^{\co H} \! N$
            of this adjunction is given by $n \mapsto {n_{(-1)}}\!^{[1]} \otimes_B {n_{(-1)}}\!^{[2]} n_{(0)}$.

            Using now the isomorphism $J \cong A\otimes_BI$ from \cite[Thm.~3.17]{AFLW} adapted to left $H$-Galois extensions, the module $A\otimes_BI$ becomes a right $A$-module via
	\begin{equation}
          \label{eq:rightaction}
		(a\otimes_Bx) \mkern 1 mu  c:= ac_{(-1)}\!{}^{[1]}\otimes_Bc_{(-1)}\!{}^{[2]}xc_{(0)}
	\end{equation}
	for $a\otimes_Bx\in A\otimes_BI$ and $c\in A$. Indeed, 
        $ac_{(-1)}\!{}^{[1]}\otimes_Bc_{(-1)}\!{}^{[2]}xc_{(0)}
        \in A\otimes_BI$
        since
$c_{(-1)}\!{}^{[1]}\otimes_Bc_{(-1)}\!{}^{[2]}\otimes c_{(0)} = \epsilon^{-1}_{A \otimes A}(1 \otimes c) \in A\otimes_B \cH$ 
                and $I$ is a left $\mathcal{H}$-ideal.
 	Moreover, it is a right $A$-action since
            \begin{eqnarray*}
		((a\otimes_Bx) c) e
	      & = &
              (ac_{(-1)}\!{}^{[1]}\otimes_Bc_{(-1)}\!{}^{[2]}xc_{(0)}) e
              \\
              &=& ac_{(-1)}\!{}^{[1]}e_{(-1)}\!{}^{[1]}\otimes_Be_{(-1)}\!{}^{[2]}c_{(-1)}\!{}^{[2]}xc_{(0)}e_{(0)}
              \\
              &\overset{\eqref{allthat2}}{=}&
              a(ce)_{(-1)}\!{}^{[1]}\otimes_B(ce)_{(-1)}\!{}^{[2]}x(ce)_{(0)}
              \\
	      &=&
              (a\otimes_Bx) (ce)
	    \end{eqnarray*}
	for all $a\otimes_Bx\in A\otimes_BI$, $c,e\in A$,
	and the above operation is obviously unital. Also note that under the isomorphism $AI \cong A\otimes_BI$,
        the right $A$-action \eqref{eq:rightaction} becomes  the usual right multiplication on the second tensor factor, {\em i.e.}, the right $A$-module structure of $\gO^u_A$.
	It remains to prove that $J$ is a left $H$-coideal. This is the case since, given $a\in A$ and $x\otimes y \in I\subseteq\mathcal{H}^+ = {}^{\mathrm{co}H}\!(\gO^u_A)$, we obtain 
	$\lambda_\otimes(ax\otimes y)
	=a_{(-1)}(x\otimes y)_{(-1)}\otimes a_{(0)}(x\otimes y)_{(0)}
	=a_{(-1)}\otimes a_{(0)}x\otimes y$.
	Thus, $J\cong A\otimes_BI$ is an $A$-subbimodule left  $H$-coideal of $\gO^u_A$ and determines a left  $H$-covariant first order differential calculus $\gO(A):=\gO^u_A/J$ on $A$.
	
	On the other hand, let $\gO(A)\cong\gO^u_A/J$ be a left  $H$-covariant first order differential calculus on $A$, where $J\subseteq\gO^u_A$ is the corresponding $A$-subbimodule and left  $H$-coideal. We claim that $I:= {}^{\mathrm{co}H}\!J$ is a left $\mathcal{H}$-ideal in $\mathcal{H}^+$.
        In fact, $0\to I\to\mathcal{H}^+= {}^{\mathrm{co}H}\!(\gO^u_A)$
        is exact since $0\to J\to\gO^u_A$ is exact and the functor
	\begin{equation*}
	{}^{{\rm co} H}(-) = \Bbbk \bx_H -
	\end{equation*}
	is exact since it is an invertible functor by (the Hopf algebraic version of) the fundamental theorem of Hopf modules.
        Moreover, $\mathcal{H}I\subseteq I$ follows from $(a\otimes c)(x\otimes y)=ax\otimes yc\in I$ for $a\otimes c\in\mathcal{H}$ and $x\otimes y\in I$. To see this, first note that $I = {}^{\mathrm{co}H}\!J
        \subseteq J$
        and thus $ax\otimes yc\in J$ since $J$ is an $A$-bimodule. Secondly,
	\begin{align*}
		(ax\otimes yc)_{(-1)}\otimes(ax\otimes yc)_{(0)}
		&=a_{(-1)}(x\otimes y)_{(-1)}c_{(-1)}\otimes a_{(0)}(x\otimes y)_{(0)}c_{(0)}\\
		&=a_{(-1)}c_{(-1)}\otimes a_{(0)}x\otimes yc_{(0)}\\
		&=1\otimes ax\otimes yc
	    \end{align*}
         	shows that $ax\otimes yc \in {}^{\mathrm{co}H}\!J = I$,
        where we used $x\otimes y\in I$ in the second equation. Therefore, we obtain a left covariant first order differential calculus $\gO(\mathcal{H}) =\mathcal{H}\otimes_B\mathcal{H}^+/I$ via the Woronowicz construction of Theorem \ref{thm:Wor}.

	    The one-to-one correspondence is again a consequence of the left $H$-Galois variant of Schneider's equivalence mentioned above: this manifests here as
$A\otimes_B{{}^{\mathrm{co}H}\!J}\cong J$ and $ {}^{\mathrm{co}H}\!(AI) \cong {}^{\mathrm{co}H}\!(A\otimes_BI) \cong I$ in the categories
        ${}_A^H\mathcal{M}$
        and ${}_B\mathcal{M}$, respectively. From the above consideration we see that
        $A\otimes_B {}^{\mathrm{co}H}\!J
        \cong J$
        and ${}^{\mathrm{co}H}\!(AI) \cong I$ are even isomorphic as $A$-subbimodule left  $H$-coideals and left $\mathcal{H}$-ideals, respectively, which concludes the proof.
\end{proof}

In the following, we exemplify the previous theorem in case of the noncommutative Hopf fibration, as in Example \ref{ex:HopfFibration}. Let us recall ({\em cf.}\ Remark \ref{leftpeft}) that in this case the Ehresmann-Schauenburg Hopf algebroid based on right coinvariants coincides with the counterpart based on left coinvariants used in the above theorem.

\begin{example}
	From Theorem \ref{thm:Wor}, we infer that left covariant first order differential calculi on $\mathcal{H}$ are in one-to-one correspondence with left $\mathcal{H}$-ideals $I\subseteq\mathcal{H}^+$. For example, for $I=0$ we obtain the universal first order differential calculus
	$\Omega=\mathcal{H}\otimes_B\mathcal{H}^+$ with differential $\dd \colon\mathcal{H}\to\Omega$ determined on generators by
	\begin{equation*}
		\begin{split}
			\dd \alpha
			&=\alpha\otimes_B(\alpha-1)+\tilde{\alpha}\otimes_B\tilde{\gamma},\\
			\dd \beta
			&=\beta\otimes_B(q^2\beta-1)+\tilde{\delta}\otimes_B\tilde{\beta},\\
			\dd \gamma
			&=\gamma\otimes_B(\gamma-1)+\tilde{\gamma}\otimes_B\tilde{\gamma},\\
			\dd \delta
			&=\delta\otimes_B(\delta-1)+q^2\tilde{\beta}\otimes_B\tilde{\delta},
		\end{split}\qquad
		\begin{split}
			\dd \tilde\alpha
			&=\tilde{\alpha}\otimes_B(\gamma-1)+\alpha\otimes_B\tilde{\alpha},\\
			\dd \tilde\beta
			&=\tilde\beta\otimes_B(q^2\beta-1)+\delta\otimes_B\tilde{\beta},\\
			\dd \tilde\gamma
			&=\tilde{\gamma}\otimes_B(\alpha-1)+\gamma\otimes_B\tilde{\gamma},\\
			\dd \tilde\delta
			&=\tilde{\delta}\otimes_B(\delta-1)+q^2\beta\otimes_B\tilde{\delta},
		\end{split}
	\end{equation*}
	and extended by the Leibniz rule. Projecting the generators via $\pi_\varepsilon\colon\mathcal{H}\to\mathcal{H}^+$, $h\mapsto h-\varepsilon(h)\lact 1_\mathcal{H}$, the Maurer-Cartan form can be expressed as
		\begin{align*}
			\varpi(\pi_\varepsilon(\alpha))
			&=\varpi(\alpha+q^2B_0-1)\\
			&=\dd (\alpha_{\smap }) \mkern 2mu \alpha_{\smam  }+0\\
			&=\dd (\alpha) \mkern 2mu  \delta+q^2
                        \dd (\tilde{\gamma}) \mkern 2mu  \tilde{\delta}\\
			&=(\alpha\otimes_B(\alpha-1)+\tilde{\alpha}\otimes_B\tilde{\gamma}) \mkern 2mu  \delta
+q^2(\tilde{\gamma}\otimes_B(\alpha-1)+\gamma\otimes_B\tilde{\gamma}) \mkern 2mu  \tilde{\delta}\\
			&=(\alpha\delta+q^2\tilde{\gamma}\tilde{\delta})\otimes_B(\alpha-1)
			+(\tilde{\alpha}\delta+q^2\gamma\tilde{\delta})\otimes_B\tilde{\gamma},
		\end{align*}
		and similarly
		\begin{align*}
			\varpi(\pi_\varepsilon(\beta))
			&=(\beta\gamma+\tilde{\beta}\tilde{\alpha})\otimes_B(q^2\beta-1)
			+(\tilde{\delta}\gamma+\delta\tilde{\alpha})\otimes_B\tilde{\beta},\\
			\varpi(\pi_\varepsilon(\gamma))
			&=(q^2\gamma\beta+\tilde{\alpha}\tilde{\beta})\otimes_B(\gamma-1)
			+q^2\tilde{\gamma}\beta\otimes_B\tilde{\gamma}
			+\alpha\tilde{\beta}\otimes_B\tilde{\alpha},\\
			\varpi(\pi_\varepsilon(\delta))
			&=\delta^2\otimes_B(\delta-1)
			+q^2\tilde{\beta}\tilde{\delta}\otimes(q^2\beta-1)
			+q^2\tilde{\beta}\delta\otimes_B\tilde{\delta}
			+q^2\delta\tilde{\delta}\otimes_B\tilde{\beta}.
		\end{align*}
Analogous expressions of the Maurer-Cartan form on the projected generators $\tilde{\alpha},\tilde{\beta},\tilde{\gamma},\tilde{\delta}$ can be deduced the same way.
Non-trivial quotients of $\Omega=\mathcal{H}\otimes_B\mathcal{H}^+$ can then be obtained via Theorem \ref{prop:ES} by considering suitable quotients of the universal calculus on $A$. For example, one may consider the $3$-dimensional left covariant calculus on $A$ or one of the $4$-dimensional bicovariant calculi on $A$, as described, {\em e.g.}, in \cite[Exs.\ 2.32 \& 2.59]{BeggsMajid}.
\end{example}

\subsection{Covariant calculi on scalar extension Hopf algebroids}
\label{sec:scalar-calc}

Let us come back to the situation of \S\ref{adapter}, where we discussed scalar extension Hopf algebroids of the form $\mathcal{H} = A \# H$. Thanks to
Theorem \ref{thm:Wor}, we are now in a position to construct covariant noncommutative differential calculi on them.

\begin{proposition}
  \label{prop:smashCalc}
	Let $H$ be a Hopf algebra (with invertible antipode) over a field $\Bbbk$ and $A$ a braided commutative algebra in ${}_H\mathcal{YD}^H$. 
	Then for a left $A$-ideal and 
	left $H$-submodule 
	$I_A\subseteq A$ and a left $H$-ideal $I_H\subseteq H^+$,
		there exists a left covariant first order differential calculus $(\Omega,\dd )$ on $\mathcal{H}=A\#H$. More precisely, set 
		\begin{equation}
                  \label{Omegasmash}
			\Omega:=\mathcal{H}\otimes_A\mathcal{H}^+/(I_A\#I_H)
		\end{equation}
		along with the differential $\dd \colon A\#H\to\Omega$ defined by
		\begin{equation}
                  \label{dsmash}
			\dd (a\#h):=(\mathcal{H}\otimes_A\pi)((a\# h_{(1)})\otimes_A(1\#h_{(2)})-(a\#h)\otimes_A(1\#1)),
		\end{equation}
		where $\pi\colon\mathcal{H}^+\to\mathcal{H}^+/(I_A\#I_H)$ denotes the projection.
\end{proposition}

\begin{proof}
	Since the counit is given by $\varepsilon\colon A\#H\to A$, $\varepsilon(a\#h)=\varepsilon_H(h)a$ and $A$ is an algebra over the field $\Bbbk$, it follows that $\mathcal{H}^+=\ker\varepsilon=A\otimes\ker\varepsilon_H=A\otimes H^+$.
        Let $I_A\subseteq A$ be a left $A$-ideal and left $H$-ideal and let $I_H\subseteq H^+$ be a left $H$-ideal. Then $I_A\otimes I_H\subseteq\mathcal{H}^+$ is a left $\mathcal{H}$-ideal since for $b\#g\in I_A\otimes I_H\subseteq A\#H$ it follows that
	$$
	(a\#h)\mkern 1mu  (b\#g)
	=\underbracket{a(h_{(1)} \mkern 1mu  b)}_{\in I_A}\#\underbracket{h_{(2)}g}_{\in I_H}
	$$
	for all $a\#h\in A\#H$. Thus, $I_A\otimes I_H\subseteq\mathcal{H}^+$ is a left $\mathcal{H}$-ideal and Theorem \ref{thm:Wor} implies that \eqref{Omegasmash} and \eqref{dsmash} form a left covariant first order differential calculus on $\mathcal{H}$.
\end{proof}

	There is another construction of noncommutative differential calculi on smash product algebras and, more in general, on crossed product algebras in \cite{AndreaThomas}, where $A\#H$ is understood as a right $H$-comodule algebra: thus, $A$ is solely required to be a module algebra rather than a braided commutative algebra in ${}_H\mathcal{YD}^H$. Accordingly, right $H$-covariant calculi on $A\#H$ are considered in contrast to the left $A\#H$-covariant calculi above, which admit a larger symmetry. It would be interesting to investigate if the scalar extension Hopf algebroid construction and Proposition \ref{prop:smashCalc} admit generalisations to crossed product algebras, where weak actions and $2$-cocycles are taken into account, as in \cite{AndreaThomas}.

\section{Takeuchi-Schneider equivalence for Hopf algebroids}\label{sec:TakSchMain}

The Takeuchi-Schneider equivalence \cite{Tak:Rel,Schneider} is a generalisation of the fundamental theorem of Hopf modules to principal homogeneous spaces. In the following, we extend this theorem to the realm of Hopf algebroids.

\subsection{Hopf kernels for Hopf algebroids}
\label{schreibtinte}
Consider two left bialgebroids over the same base algebra $R$,  denoted here by $(\mathcal{G},R,s_\cG, t_\cG, \Delta_\cG,\gve_\cG)$ resp.\ $(\mathcal{H},R,s_\cH, t_\cH, \Delta_\mathcal{H},\varepsilon_\mathcal{H})$ for better distinction, and 
let $\mathcal{G} \stackrel{\pi}{\to} \mathcal{H} \to 0$ be an exact sequence of left bialgebroids (with the identity map $R \to R$ on the base algebra).
The map 
\begin{equation*}
	\rho_\mathcal{G} := (\mathcal{G} \otimes_R  \pi) \circ \Delta_\cG \colon \mathcal{G} \to \mathcal{G} \otimes_R \mathcal{H}
\end{equation*}
defines a right $\mathcal{H}$-coaction on $\mathcal{G}$ with respect to which $\mathcal{G}$ becomes a right $\mathcal{H}$-comodule algebra. Note that, as customary, the image of $\rho_\mathcal{G}$ lands in the Takeuchi subspace, that is,  $\rho_\mathcal{G} (\mathcal{G}) \subseteq {\mathcal{G}}_\ract \times_{\!R} \due {\mathcal{H}} \lact {}$, see \S\ref{mittagleffler}.

\begin{definition}[Left Hopf kernel]
	\label{boing}
	Consider the exact sequence $\mathcal{G} \stackrel{\pi}{\to} \mathcal{H} \to 0$ of left bialgebroids over $R$, with the identity map on this base algebra.
	The set
	\begin{equation*}
		B \coloneqq \mathcal{G}^{\mathrm{co}\mathcal{H}}
                \coloneqq
                \big\{g \in \mathcal{G} \mid \rho_\mathcal{G} (g) = g \otimes_R  1_\mathcal{H} \in \mathcal{G} \otimes_R \mathcal{H}\big\} 
	\end{equation*}
	is called the {\em left Hopf kernel} of $\pi$. Put $B^+:=B\cap\ker\gve_\cG$.
\end{definition}

\noindent This definition is a straightforward extension to left bialgebroids of the corresponding definition for Hopf algebras (resp.\ bialgebras) in \cite[Def.\  4.12]{BlattnerCohenMontgomery}.

\begin{remark}
\label{tobeornottobeakernel}
  The left Hopf kernel of $\pi\colon \cG \to \cH$ can be seen as the kernel of the difference map ({\em i.e.}, the equaliser)
  of 
  \begin{equation}
    \label{kartoffelsuppe}
\varphi \colon \xymatrix@C=25pt{ B \ar[r] & \due \cG \lact \ract \ar@<+0.7ex>[rr]^-{(\cG \otimes_R \pi) \circ \Delta_\cG } \ar@<-0.7ex>[rr]_-{\cG \, \otimes_R \, s_\cH} & & \, \due \cG {} \ract \! \times_{\!R}  \due \cH \lact {}}
\end{equation}
  in the category of left $R$-modules, and hence its name: therefore, $B$ is a left $R$-submodule of $\mathcal{G}$ with action induced by $ \due {\mathcal{G}} \lact {}$, that is, by the source map, but in general not a right $R$-submodule with action induced by $\cG_\ract$, that is, by the target map (in general, there is no reason why $t_\cG(r)b$ or $bt_\cG(r)$ should be elements in $B$).
  Moreover, observe that in the category of left bialgebroids (over the same base algebra $R$) the bialgebroid $\RRe = R \otimes R^\op$, by means of the map $\eta(\cdot \otimes \cdot\cdot) := s(\cdot)t(\cdot\cdot)$, is initial  but not terminal; that is to say, there is, in general, no projection $\pi \colon \cG \to \RRe$, which, in turn, implies that $B$ cannot equal $\cG$ itself, in striking contrast to the bialgebra case. This observation will become important in our later discussion of differential calculi, see Remark \ref{noworo}.
\end{remark}

The following is an analogue of \cite[Lem.~4.13]{BlattnerCohenMontgomery} and proven in \cite[Lem.~2.7]{BekKowSar:UEAOLRACPCAC}.

\begin{lemma}
	\label{wasserkessel}
	The left Hopf kernel $B \subseteq \mathcal{G}$ is an $R$-subring of $\mathcal{G}$ with unit map $s_\cG \colon R \to B$. 
	The elements of $B$ in $\mathcal{G}$ commute with the elements of $t_\cG(R)$, that is, 
	\begin{equation}
		\label{ohyes}
		b \ract r = 
		t_\cG(r)\, b = b\, t_\cG(r)
		= r \blact  b
	\end{equation}
	for all $b \in B$ and $r \in R$. 
\end{lemma}

Having introduced $B$ allows for one of the main notions we are going to use, that is, that of Hopf-Galois extension in the context of bialgebroid surjections, see also \cite{Boe:GTFHA}:

\begin{definition}
  \label{erbsensuppe}
	Let $(\pi, \id_R) \colon (\cG,R) \to (\cH,R)$ be a surjection of left bialgebroids and $B$ the corresponding (left) Hopf kernel. 
	We call $B\subseteq \cG$ a \textit{Hopf-Galois extension} if the linear map
	\begin{equation}
        		\label{ohno}
	\chi\colon \cG\otimes_B\cG\to \cG\otimes_R \cH,\quad g\otimes_B g'\mapsto g_{(1)}g'\otimes_R\pi(g_{(2)})
	\end{equation}
	is a bijection.
We denote by
    \begin{equation}
      \label{cuffie}
		\mathcal{T} \colon\mathcal{H}\to\mathcal{G}\otimes_B\mathcal{G},\qquad
			h\mapsto h^{\scriptscriptstyle \langle+\rangle}\otimes_Bh^{\scriptscriptstyle  \langle-\rangle}:=\chi^{-1}(1_\mathcal{G}\otimes_Rh)
	\end{equation}
the corresponding translation map.
\end{definition}

\begin{remark}
  The map \eqref{ohno} is not the direct analogue of the  
customary map \eqref{berlinale} in the Hopf algebra context, which one would choose if starting from {\em right} bialgebroids, but it rather corresponds to the second one mentioned in \eqref{Galois2}. Also, we refrain from formulating Definition \ref{erbsensuppe} in full generality by starting from an $\cH$-comodule algebra instead of $\cG$ as this shall not be used in the sequel.
 Nevertheless,   note that $\chi$ is $B$-balanced since 
	$$
	g_{(1)}b_{(1)}g'\otimes_R\pi(g_{(2)}b_{(2)})
	=g_{(1)}b_{(1)}g'\otimes_R\pi(g_{(2)})\pi(b_{(2)})
	=g_{(1)}bg'\otimes_R\pi(g_{(2)}),
	$$
	where we used the fact that $\pi$ is a ring morphism.
\end{remark}

\subsection{Surjections of scalar extension Hopf algebroids}\label{sec:HomogeneousScalarExt}
Let $\pi\colon G\to H$ be a Hopf algebra surjection. In such a situation, one can slightly extend the notion of (left-right) YD modules in the following sense: define the category ${}_H\mathcal{YD}^G$ of $H$-$G$-Yetter-Drinfel'd modules with objects $M$ being left $H$-modules and simultaneously right $G$-comodules, subject to the compatibility
\begin{equation*}
	h_{(1)} m_{(0)}\otimes h_{(2)}\pi(m_{(1)})
	=(h_{(2)}  m)_{(0)}\otimes\pi((h_{(2)}  m)_{(1)})h_{(1)}
\end{equation*}
for all $h\in H$ and $m\in M$. Morphisms in this category are simply left $H$-linear and right $G$-colinear maps.
The following lemma is then straightforward:

\begin{lemma}
  For any Hopf algebra surjection $\pi\colon G\to H$, the category ${}_H\mathcal{YD}^G$ is braided mon\-oidal, with monoidal structure defined on objects $M,N$ in ${}_H\mathcal{YD}^G$ by $M\otimes N$ with diagonal left $H$-action, right $G$-coaction $m\otimes n\mapsto m_{(0)}\otimes n_{(0)}\otimes n_{(1)}m_{(1)}$, and where
	\begin{equation*}
\sigma^\mathcal{YD}_{M,N} \colon M\otimes N \to N\otimes M, \quad		m\otimes n \mapsto n_{(0)}\otimes\pi(n_{(1)})  m
	\end{equation*}
        defines the corresponding braiding.
\end{lemma}

Note that there are functors ${}_H\mathcal{YD}^H\xleftarrow{F^\pi}{}_H\mathcal{YD}^G\xrightarrow{{}_\pi F}{}_G\mathcal{YD}^G$ with right $G$-coaction $\rho\colon M\to M\otimes G$ inducing a right $H$-coaction by $(M \otimes\pi)\circ\rho\colon M\to M\otimes H$ and left $H$-action %$\cdot\colon
$H\otimes M\to M$, inducing a left $G$-action by %$\cdot\circ(\pi\otimes M)\colon
$G\otimes M\to M$ by precomposing it with $(\pi\otimes M)$.

Let $A$ be an algebra object in ${}_H\mathcal{YD}^G$, that is, the multiplication $m\colon A\otimes A\to A$ and unit $\eta\colon\Bbbk\to A$ are morphisms in ${}_H\mathcal{YD}^G$. We say that $A$ is braided commutative if $m\circ\sigma^\mathcal{YD}_{A,A}=m$. On elements $a,b\in A$, the braided commutativity then reads $b_{(0)}(\pi(b_{(1)}) \mkern 1 mu a)=ab$.

\begin{corollary}
	If $A$ is a braided commutative algebra in ${}_H\mathcal{YD}^G$, then $F^\pi(A)$ is a braided commutative algebra in ${}_H\mathcal{YD}^H$ and ${}_\pi F(A)$ is a braided commutative algebra in ${}_G\mathcal{YD}^G$.
\end{corollary}

For simplicity, let $\Bbbk$ be a field in the next proposition (or, more generally, let all appearing Hopf algebras be flat over $\Bbbk$).

\begin{proposition}\label{prop:smashhomogeneous}
	Let $\pi\colon G\to H$ be a Hopf algebra surjection, assume that $G$ and $H$ have invertible antipodes, and let $A$ be a braided commutative algebra in ${}_H\mathcal{YD}^G$. Then:
	\begin{enumerate}
		\compactlist{99}
		\item
		There is a right Hopf algebroid surjection
		$$
		\pi_\#\colon{}_\pi F(A)\# G\to F^\pi(A)\# H, \quad a\# g \mapsto a\#\pi(g)
		$$
		between the scalar extension Hopf algebroids.
\end{enumerate}
        In the following, for better readability, we mostly omit the functors ${}_\pi F$ and $F^\pi$.
	\begin{enumerate}
\setcounter{enumi}{1}
	  \compactlist{99}
\item
		Viewing $A\#G$ as a right $A\#H$-comodule algebra via the coaction
		$$
		\rho_\#:=((A \# G)\otimes_A\pi_\#)\circ\Delta_{A\#G}\colon A\#G\to(A\#G)\otimes_A(A\#H),
		$$
		the left Hopf kernel of $A\#G$ is isomorphic to $A\otimes B$, where $B:=G^{\mathrm{co}H}$.
		\item
		If $B =G^{\mathrm{co}H} \subseteq G$ is a Hopf-Galois extension as in \S\ref{fuellfederhalter}, then $(A\#G)^{\mathrm{co}A\#H}\subseteq A\#G$ is a Hopf-Galois extension over the bialgebroid $A \# H$ in the sense of Definition \ref{erbsensuppe}.
	\end{enumerate}
\end{proposition}

\begin{proof}
  \
  \begin{enumerate}
    \compactlist{99}
\item
    The claimed surjectivity is immediate from that of $\pi$. We then have to show that $\pi_\#$ commutes with all bialgebroid structure maps. To start with, it is obvious that $\pi_\#\circ s_{{}_\pi F(A)\#G}=s_{F^\pi(A)\#H}$ and $\pi_\#\circ t_{{}_\pi F(A)\#G}=t_{F^\pi(A)\#H}$ as well as the compatibility of the units resp.\ counits. Next, 
		\begin{align*}
			\pi_\#((a\#g)(a'\#g'))
			&=\pi_\#(a(\pi(g_{(1)}) \mkern 1 mu  a')\# g_{(2)}g')\\
			&=a(\pi(g_{(1)}) \mkern 1 mu a')\#\pi(g_{(2)}g')\\
			&=a(\pi(g)_{(1)} \mkern 1 mu a')\#\pi(g)_{(2)}\pi(g')\\
			&=(a\#\pi(g))(a'\#\pi(g'))
			=\pi_\#(a\#g)\pi_\#(a'\#g')
		\end{align*}
proves compatibility with the multiplication, while
		\begin{align*}
			(\pi_\#\otimes_A\pi_\#)(\Delta_{A\#G}(a\#g))
			&=\pi_\#(a\#g_{(1)})\otimes_A\pi_\#(1\#g_{(2)})\\
			&=(a\#\pi(g_{(1)}))\otimes_A(1\#\pi(g_{(2)}))\\
			&=(a\#\pi(g)_{(1)})\otimes_A(1\#\pi(g)_{(2)})
			=\Delta_{A\#H}(\pi_\#(a\otimes g))
		\end{align*}
                that for the comultiplication.
\item
Using $(A \# G) \otimes_A (A \# H) \cong A \otimes G \otimes H$ for the codomain of the coaction $\rho_\#$ shows that, up to an isomorphism, the Hopf kernel is given by 
  $\ker (A \otimes \gvf)$, where $\gvf$ is the map described in \eqref{kartoffelsuppe}; as we assume here that $\Bbbk$ is a field, this is then equal to
  $A  \otimes  \ker \gvf = A \otimes B$.
  \item
  We claim that the
  inverse of the Galois map
		\begin{eqnarray*}
			\chi_\#\colon(A\#G)\otimes_{A\otimes B}(A\#G) &\to& (A\#G)\otimes_A(A\#H),
			\\
			(a\#g)\otimes_{A\otimes B}(a'\#g')
			&\mapsto& (a\#g_{(1)}) (a'\#g') \otimes_A(1\#\pi(g_{(2)}))
		\end{eqnarray*}
	as in Eq.~\eqref{ohno}	will be given by
\begin{equation}
\label{jour}
  \begin{array}{rcl}
			\chi^{-1}_\# \colon (A\#G)\otimes_A  (A\#H) &\to& (A\#G)\otimes_{A\otimes B}(A\#G),
			\\
			(a' \# g') \otimes_A	(a\#h)
			&\mapsto& (a_{(0)}\#S^{-1}(h)^{\langle 1\rangle})\otimes_{A\otimes B}(1\#S^{-1}(h)^{\langle 2\rangle}a_{(1)})(a' \# g'),  
\end{array}
                \end{equation}
		where $\tau\colon H\to A\otimes_BA$, $h \mapsto h^{\langle 1\rangle}\otimes_Bh^{\langle 2\rangle}$ denotes the translation map of \eqref{berlinale},
           that is, the respective translation map corresponding to $\chi_\#$ reads
	   \begin{equation}
             \label{nuit}
		A\#H \to (A\#G)\otimes_{A\otimes B}(A\#G),
		\quad
			a\#h
			\mapsto (a_{(0)}\#S^{-1}(h)^{\langle 1\rangle})\otimes_{A\otimes B}(1\#S^{-1}(h)^{\langle 2\rangle}a_{(1)}),
		\end{equation}
                which can be checked directly. Alternatively,
                let us show that $\chi_\#$ fits into the commutative diagram
\begin{equation*}
		\begin{tikzcd}
		  A \otimes (G \otimes_B G) 
\arrow{r}{A \otimes \tilde\chi_G}
\arrow{d}[swap]{\cong}
& 
		  A \otimes (G \otimes H) 
                  \arrow{d}{\cong}
                   \\
 (A \# G) \otimes_{A \otimes B} (A \# G)
 \arrow{r}{\chi_\#}
 &  
 (A \# G)_\ract \otimes_A \due {(A \# G)} \lact {},
		\end{tikzcd}
	\end{equation*}
where $\tilde\chi_G$ is the map from \eqref{Galois2} in case $A = G$.
Indeed, if we call the vertical isomorphisms from left to right $\phi$ and $\psi$ for the moment,
one has
\begin{equation*}
\begin{split}
& \big(\psi \circ (A \otimes \tilde\chi_G) \circ \phi^{-1}\big)\big((a \# g) \otimes_{A \otimes B} (a' \# g')\big)
  \\
  = \ &
  \big(\psi \circ (A \otimes \tilde\chi_G) \circ \phi^{-1}\big)\big((a \# g)(a' \# 1_G) \otimes_{A \otimes B} (1 \# g')\big)
  \\
  = \ &
  \big(\psi \circ (A \otimes \tilde\chi_G)\big)\big(a (g_{(1)} a') \otimes (g_{(2)} \otimes g')\big)
  \\
  = \ &
\big(a (g_{(1)} a') \otimes g_{(2)}g'\big) \otimes_A \big(1 \otimes \pi(g_{(3)})\big),
\end{split}
\end{equation*}
which is $\chi_\#$ given above. Hence, using the inverse \eqref{chi'inv} of $\tilde\chi_G$ and the target map \eqref{ratagnan},
\begin{equation*}
\begin{split}
  \chi_\#^{-1}\big((a' \# g') \otimes_A (a \# h)\big)
  & =
  \big(\phi \circ (A \otimes \tilde\chi_G)^{-1} \circ \psi^{-1}\big)\big((a' \# g') \otimes_A  (a \# h)\big)
  \\
&  = 
  \big(\phi \circ (A \otimes \tilde\chi_G)^{-1} \circ \psi^{-1}\big)\big((a_{(0)} \# a_{(1)})(a' \# g') \otimes_A  (1 \# h)\big)
  \\
&  = 
  \big(\phi \circ (A \otimes \tilde\chi^{-1}_G)\big)\big(a_{(0)} (a_{(1)}a') \otimes (a_{(2)} g' \otimes  h)\big)
  \\
&  = \big((a_{(0)} (a_{(1)}a') \# S^{-1}(h)^{\langle 1\rangle}\big) \otimes_{A \otimes B} \big(1 \# S^{-1}(h)^{\langle 2\rangle}  a_{(2)} g'\big),
\end{split}
\end{equation*}
which could be shown to be equal to \eqref{jour} by using the left $H$-colinearity of the canonical map \eqref{Galois2}, but even simpler is to observe that for $a' \# g' = 1 \# 1_G$ one obtains the translation map claimed in \eqref{nuit}. Then, by right $(A \# G)$-linearity of $\chi_\#$ resp.\ its inverse (as in the general case in Definition \ref{erbsensuppe}), one directly obtains the expression in \eqref{jour}.
\end{enumerate}
  This concludes the proof.
\end{proof}

\subsection{Classifying Hopf-Galois extensions}
Let us now pass to a classification of bialgebroid Hopf-Galois extensions in the spirit of what is known for Hopf algebras \cite[\S3]{Schneider}, \cite[\S1]{Tak:Rel}:

\begin{proposition}
	\label{prop:HG}
	Given a (right) Hopf algebroid surjection $\pi\colon \cG \to \cH $ with (left) Hopf kernel $B$ such that both $\mathcal{G}_\ract$ and $\due \cG \lact {}$ are $R$-flat, the following are equivalent:
	\begin{enumerate}
		\compactlist{99}
	      \item
                $\cG B^+=\ker\pi$;
	      \item
                $B\subseteq \cG $ is a Hopf-Galois extension.
	\end{enumerate}
\end{proposition}

\noindent For its proof we shall need a couple of technical details in (almost) complete analogy to the case of Hopf algebras, which will be dealt with first.

\begin{lemma}
	\label{spinat}
	Let $B = \cG^{\mathrm{co}\cH}$ be a left Hopf kernel such that $\mathcal{G}_\ract$ is $R$-flat.
	\begin{enumerate}
		\compactlist{99}
		\item 
		Then $B$ is a left $\mathcal{G}$-comodule algebra by means of $\lambda_B:=\Delta|_B\colon B\to \mathcal{G}\otimes_R B\subseteq \mathcal{G}\otimes_R \mathcal{G}$.
		\item
		The restriction $\varepsilon|_B\colon B\to R$ is an algebra morphism. In particular, $B^+$ is a two-sided ideal in $B$ and $\mathcal{G}B^+\subseteq\ker\pi$.
		\item
		The set $\mathcal{G}B^+\subseteq \mathcal{G}$ is a left $\mathcal{G}$-module and also a coring coideal, that is, $\mathcal{G}B^+\subseteq\ker\varepsilon$ as well as
		$
		\Delta(\mathcal{G}B^+)\subseteq \mathcal{G}B^+\otimes_R \mathcal{G} + \mathcal{G}\otimes_R \mathcal{G}B^+
		$
		holds.
	\end{enumerate}
\end{lemma}

\begin{proof}
	All of these are direct verifications that use the bialgebroid properties such as counitality of the coproduct, along with the properties of the left Hopf kernel in Lemma \ref{wasserkessel}. 
	\begin{enumerate}
		\compactlist{50}
		\item
		It suffices to prove that $\Delta(B)\subseteq \mathcal{G}\otimes_RB$. 
		Indeed,  since 
		$
		(\mathcal{G}\otimes_R\rho_\mathcal{G})(\Delta b)
		=b_{(1)}\otimes_R\rho_\cG (b_{(2)})
		=b_{(1)}\otimes_Rb_{(2)}\otimes_R\pi(b_{(3)})
		=(\Delta\otimes_R\mathcal{G})\rho_\mathcal{G}(b)
		=\Delta b \otimes_R1
		$
		by coassociativity of $\Delta$
and that $\cG \otimes_R \ker \varphi = \ker(\cG \otimes_R \varphi)$ thanks to the flatness of $\cG_\ract$ as a right $R$-module, where $\ker \varphi$ defines $B$ as in \eqref{kartoffelsuppe},
                                one deduces that 
		$\Delta b \in \mathcal{G}\otimes_RB$ for $b \in B$.
		\item
		  For $b,b'\in B$, one has $\varepsilon(bb')=\varepsilon(\varepsilon(b') \blact b)\overset{\eqref{ohyes}}{=}\varepsilon(b \ract \varepsilon(b'))=\varepsilon(b)\varepsilon(b')$, that is, on $B$ the counit behaves like an $R$-ring morphism, which on a general element in $\cG$ is not true.
                  Hence, $\varepsilon(bb^+)=\varepsilon(b^+b) =0$ for $b^+ \in B^+$.
		On the other hand, 
		\begin{align*}
			\pi(gb^+)
			&=\varepsilon(g_{(1)}b^+_{(1)}) \lact \pi(g_{(2)}b^+_{(2)})
			=\varepsilon(g_{(1)}b^+_{(1)}) \lact \pi(g_{(2)})\pi(b^+_{(2)})\\
			&=\varepsilon(g_{(1)}b^+) \lact \pi(g_{(2)})
			=\varepsilon\big(g_{(1)} \bract  \varepsilon(b^+)\big) \lact \pi(g_{(2)})
			=0
		\end{align*}
		for $g\in \mathcal{G}$, and thus $\mathcal{G}B^+\subseteq\ker\pi$.
		\item 
		Clearly, $\mathcal{G}B^+$ is a left $\mathcal{G}$-module and that $\mathcal{G}B^+\subseteq\ker\varepsilon$ has already been shown in part (ii). 
		To prove that $\Delta(\mathcal{G}B^+)\subseteq \mathcal{G}B^+\otimes_R \mathcal{G} + \mathcal{G}\otimes_R \mathcal{G}B^+$, let $g\in \mathcal{G}$ and $b^+\in B^+$. Then
		$$
		\Delta(gb^+)
		= g_{(1)}b^+_{(1)} \otimes_R g_{(2)}b^+_{(2)}
		= g_{(1)}b^+_{(1)} \otimes_R \underbracket{g_{(2)}(b^+_{(2)}-s(\varepsilon(b^+_{(2)})))}_{\in \mathcal{G}B^+}
		+ 
		\underbracket{g_{(1)}b^+}_{\in \mathcal{G}B^+}\otimes_Rg_{(2)}
		$$
		is an element of $\mathcal{G}B^+\otimes_R \mathcal{G} + \mathcal{G}\otimes_R \mathcal{G}B^+$, using $\Delta(B)\subseteq \mathcal{G}\otimes_RB$ from (i).
	\end{enumerate}
	This concludes the proof.
\end{proof}

The previous lemma implies that the quotient
$$
\overline{\cG }:=\cG /\cG B^+ \cong \cG \otimes_B R
$$
is a left $\cG $-module coring, the canonical projection
$\wp\colon \cG \to\overline{\cG }$ of which induces its structure maps: writing generic elements in $\overline \cG$ as $\wp(g)$, one sets   
\begin{equation}
	\label{pepe}
	g \cdot \wp(g'):=\wp(gg'),\qquad
	\Delta(\wp(g)):=\wp(g_{(1)})\otimes_R\wp(g_{(2)}),\qquad
	\varepsilon(\wp(g)):=\varepsilon(g).
\end{equation}
In particular, $\wp$ is an $\RRe$-module map, that is,
$
r \lact \wp(g):=\wp(r \lact g)$ as well as $\wp(g) \ract r := \wp(g \ract r).
$
Moreover, $\overline{\cG }$ is a right $\cH $-comodule with respect to the right $\cH $-coaction 
$$
\overline{\cG }\to\overline{\cG }\otimes_R\cH ,\quad
\wp(g)\mapsto\wp(g_{(1)})\otimes_R\pi(g_{(2)}).
$$
Note that this map is well-defined since $\cG B^+\subseteq\ker\pi$ and $\Delta(\cG B^+)\subseteq \cG B^+\otimes_R\cG +\cG \otimes_R\cG B^+$
as just proven in (ii) and (iii) of the previous lemma.

\begin{lemma}
	\label{Warschau}
	Let $\pi\colon \cG \to \cH$ be a surjection of right Hopf algebroids with Hopf kernel $B=\cG^{\mathrm{co}\cH}$, and let both $\mathcal{G}_\ract$ and $\due \cG \lact {}$ be $R$-flat.
Then
	\begin{equation}
		\label{Kleinseite}
		\gamma \colon \cG   \otimes_B \cG 
		\xrightarrow{\cong}
                \cG_\ract  \otimes_R \due {\overline{\cG}} \lact {}, 
		\quad g\otimes_B g'\mapsto g_{(1)}g' \otimes_R \wp(g_{(2)})
	\end{equation}
	is an isomorphism of the $R$-bimodules
        $
        \due {(\cG \otimes_B \cG)} \blact \ract :=
        \due \cG  \blact \ract \otimes_B \cG $
        and
        $
\due {(\cG  \otimes_R \overline{\cG})} \blact \ract :=
        \cG  \otimes_R \due {\overline{\cG}} \blact \ract,
$
        with inverse given by $g\otimes_R\wp(g')\mapsto g'_{\smap }\otimes_B g'_{\smam }g$. 
	This map is moreover a morphism of right $\cH$-comodules with respect to the coactions
	\begin{equation*}
		\begin{array}{rrrl}
			\rho_{\cG  \otimes_B \cG}\colon \!\!\!\!\! & \cG  \otimes_B \cG  \to (\cG  \otimes_B \cG ) \otimes_R \cH , 
			& g \otimes_B g' \!\!\! & \mapsto (g_{(1)} \otimes_B g') \otimes_R \pi(g_{(2)}), 
			\\
			\rho_{\cG \otimes_R\overline{\cG }}\colon \!\!\!\!\! & \cG  \otimes_R \overline{\cG } \to (\cG  \otimes_R \overline{\cG }) \otimes_R \cH , 
			& 
			g \otimes_B \wp(g') \!\!\! & \mapsto \big(g \otimes_B \wp(g'_{(1)})\big) \otimes_R \pi(g'_{(2)}). 
		\end{array}
	\end{equation*}
\end{lemma}

\begin{proof}
	We first prove that \eqref{Kleinseite} is a well-defined map over the balanced tensor product. For $b\in B$, one has
	$$
\wp(b)=\wp\big(b+\underbracket{(s\varepsilon(b)-b)}_{\in\ker\varepsilon}\big)
	=\wp(s \varepsilon(b)),
	$$
        which is also immediate from the identification
        $\overline \cG \cong \cG \otimes_B R$.
	Furthermore, 
	recall that $\Delta(B)\subseteq \cG \otimes_RB$, so that $b_{(1)}\otimes_R\wp(b_{(2)})=b_{(1)}\otimes_R\wp(s\varepsilon(b_{(2)}))=b\,\otimes_R\wp(1)$, and both statements
	together with the left $\cG$-linearity of $\wp$ defined in \eqref{pepe} imply
	\begin{equation*}
		\begin{split}
			g_{(1)}b_{(1)}g' \otimes_R \wp(g_{(2)}b_{(2)})
			&=g_{(1)}b_{(1)} g'\otimes_R g_{(2)}\cdot\wp(b_{(2)})
			\\
			&=g_{(1)}b g'\otimes_R g_{(2)}\cdot\wp(1)
			=g_{(1)}b g'\otimes_R\wp(g_{(2)}).
		\end{split}
	\end{equation*}
	This proves that the map \eqref{Kleinseite}
	is well-defined. 
	That the (what is going to be)
	inverse map $\cG \otimes_R\overline{\cG }\to \cG \otimes_B\cG $ given by  $g\otimes_R\wp(g')\mapsto g'_{\smap }\otimes_B g'_{\smam }g$ is well-defined follows from the fact that for all 
	$b^+\in B^+$, we obtain
	\begin{equation*}
		\begin{split}
			(g'b^+)_{\smap }\otimes_B(g'b^+)_{\smam }g
			&=g'_{\smap }b^+_{\smap }\otimes_Bb^+_{\smam }g'_{\smam }g
\\
			&
                  =
                  g'_{\smap }\otimes_B b^+_{\smap }b^+_{\smam } g'_{\smam } g
			\\
			&
                        =g'_{\smap } \otimes_B (g'_{\smam } \ract \varepsilon(b^+)) g 
			=0,
		\end{split}
	\end{equation*}
	where we used that
	$b_{\smap }\otimes_Rb_{\smam }\in B\otimes_R\cG $.
This, in turn, holds since
\begin{equation}
  \label{torino}
    b_{\smap (1)} \otimes_R \pi(b_{\smap (2)}) \otimes_R b_\smam
    \, \overset{\eqref{Tch4}}{=} \,
     b_{(1)\smap} \otimes_R \pi(b_{(2)}) \otimes_R b_{(1)\smam}
     \, = \, 
      b_{\smap} \otimes_R 1_\cH \otimes_R b_{\smam},
     \end{equation}
from which one deduces
$b_{\smap }\otimes_Rb_{\smam  } \in B \otimes_R \mathcal{G}$, using $\ker( \gvf \otimes_R \cG) = \ker \gvf \otimes_R \cG$ by the left $R$-flatness of $\due \cG \lact {}$,
and where $\ker \varphi$ defines $B$ as in \eqref{kartoffelsuppe}.

        That $\gamma$ is invertible with the mentioned inverse simply follows from the properties \eqref{Tch2} \& \eqref{Tch3} of a right Hopf algebroid.
	Moreover, $\gamma$ is right $\cG$-colinear since
	\begin{equation*}
		\begin{split}
			\big((\gamma\otimes_R\mathrm{id}) \circ \rho_{\cG \otimes_B\cG }\big)(g\otimes_B g')
			&=\gamma(g_{(1)}\otimes_B g')\otimes_R \pi(g_{(2)})
			\\
			&=g_{(1)}g'\otimes_R\wp(g_{(2)})\otimes_R\pi(g_{(3)})
			\\
			&=\rho_{\cG \otimes_R\overline{\cG }}\big(g_{(1)}g'\otimes_R\wp(g_{(2)})\big)
						=(\rho_{\cG \otimes_R\overline{\cG}} \circ \gamma)(g\otimes_B g'),
		\end{split}
	\end{equation*}
	which concludes the proof of the lemma.
\end{proof}

\begin{proof}[Proof of Proposition \ref{prop:HG}]
	Writing $\overline{\cG} = \cG /\cG B^+$ as well as $\cH  = \cG /\ker \pi$, consider the map  $f$ defined by the diagram
	\begin{equation*}
		\begin{tikzcd}
			&    \cG
			\arrow{rd}{\pi} 
			\arrow{ld}[swap]{\wp}
			& 
			\\
			\overline{\cG}    \arrow{rr}{f} & & \cH, 
		\end{tikzcd}
	\end{equation*}
	that is,
	$f\colon\overline{\cG}\to \cH$, $\wp(g)\mapsto\pi(g)$. This is a well-defined morphism of $R$-modules since $\cG B^+\subseteq\ker\pi$ by Lemma \ref{spinat}\,(iii). Now stare at the diagram 
	\begin{equation*}
		\begin{tikzcd}
			\cG \otimes_B\cG  \arrow{rrr}{\cong\text{ by Lemma \ref{Warschau}}}
			\arrow{ddrrr}{\chi} & & & \cG \otimes_R\overline{\cG } \arrow{dd}{\cG \otimes_Rf}\\
			& & & \\
			& & & \cG \otimes_R\cH 
		\end{tikzcd}
	\end{equation*}
	and accept its superior truth.
	If $\cG B^+=\ker\pi$, then $f$ is an isomorphism, which implies that $\chi$ is an isomorphism by the commutativity of the second diagram. This means that (i) implies~(ii). If, on the other hand, $B\subseteq \cG $ is a Hopf-Galois extension, by definition $\chi$ is an isomorphism. By  commutativity of the second diagram, this implies that $\cG \otimes_Rf$ is an isomorphism. In particular, the restriction $(\cG \otimes_Rf)|_{R\otimes_R\overline{\cG}}\colon R\otimes_R\overline{\cG}\to R\otimes_R\cH$ is injective. Since $R \otimes_R -$ is the trivial functor, we can identify the previous map with $f$ and deduce that $f$ is injective. By construction, $f$ is surjective, and thus (ii) implies (i), proving the stated equivalence.
	Moreover, $f$ is a morphism of left $\cG$-module corings, providing an isomorphism of such, in case (i) or (ii) are satisfied.
\end{proof}

As an immediate consequence of the now proven Proposition \ref{prop:HG}, we obtain:

\begin{corollary}
  \label{cor:overlineG}
  If one (and thus both) of the above conditions in Proposition \ref{prop:HG} is satisfied, then $\overline{\cG}\cong \cH$ as left $\cG$-module corings.
\end{corollary}

\begin{remark}
Let $B\subseteq \cG$ be a Hopf-Galois extension as in Definition \ref{erbsensuppe} such that both $\mathcal{G}_\ract$ and $\due \cG \lact {}$ are $R$-flat. Then we can express the corresponding translation map $\mathcal{T} \colon\mathcal{H}\to\mathcal{G}\otimes_B\mathcal{G}$, $h\mapsto h^{\scriptscriptstyle\langle+\rangle}\otimes_Bh^{\scriptscriptstyle\langle-\rangle}$ from Eq.~\eqref{cuffie} in terms of the translation map on $\mathcal{G}$, that is, for all $g\in\mathcal{G}$,
\begin{equation*}
	\mathcal{T}(\pi(g))=g_{\smap }\otimes_Bg_{\smam  }.
\end{equation*}
To show that this is well-defined, 
let $gb^+\in\ker\pi=\mathcal{G}B^+$ as in Proposition~\ref{prop:HG}.
Then
\begin{align*}
	(gb)_{\smap }\otimes_B(gb)_{\smam  }
	= g_{\smap }b_{\smap }\otimes_Bb_{\smam  }g_{\smam  }
	 \stackrel{\eqref{torino}}{=} g_{\smap }\otimes_Bb_{\smap }b_{\smam  }g_{\smam  }
	=g_{\smap }\otimes_Bt(\varepsilon(b))g_{\smam  }
	=0
\end{align*}
by \eqref{Tch6} along with \eqref{Tch7}.
From this, for all $h,h'\in\mathcal{H}$, and $g\in\mathcal{G}$, the identities
\begin{eqnarray*}
  {h^{\scriptscriptstyle \scriptscriptstyle\langle+\rangle}}_{\! (0)} h^{\scriptscriptstyle\langle-\rangle} \otimes_R {h^{\scriptscriptstyle\langle+\rangle}}_{\!(1)}  &=& 1 \otimes_R h \quad \in \mathcal{G}_{\ract} \! \otimes_R \! {}_\lact \mathcal{H},
  \\
	{g_{(1)}}^{\! \scriptscriptstyle\langle+\rangle} \otimes_B  {g_{(1)}}^{\! \scriptscriptstyle\langle-\rangle}g_{(0)}   &=& g \otimes_B 1
	\quad \in \cG 
	\otimes_B  \cG,
        \\
	{h^{\scriptscriptstyle\langle+\rangle}}_{\! (0)} \otimes_B h^{\scriptscriptstyle\langle-\rangle} \otimes_R {h^{\scriptscriptstyle\langle+\rangle}}_{\! (1)} &=& {h_{(1)}}^{\! \scriptscriptstyle\langle+\rangle} \otimes_B
	{h_{(1)}}^{\! \scriptscriptstyle\langle-\rangle} \otimes_R  h_{(2)},
        \\
	h^{\scriptscriptstyle\langle+\rangle \scriptscriptstyle\langle+\rangle} \otimes_B  h^{\scriptscriptstyle\langle+\rangle \scriptscriptstyle\langle-\rangle} \otimes_B h^{\scriptscriptstyle\langle-\rangle} &=&
	h^{\scriptscriptstyle\langle+\rangle} \otimes_B {h^{\scriptscriptstyle\langle-\rangle}}_{\! (1)} \otimes_B {h^{\scriptscriptstyle\langle-\rangle}}_{\! (2)},  \\
	(hh')^{\scriptscriptstyle\langle+\rangle} \otimes_B (hh')^{\scriptscriptstyle\langle-\rangle} &=& h^{\scriptscriptstyle\langle+\rangle} h'^{\mkern 1.5mu\scriptscriptstyle\langle+\rangle}
	\otimes_B h'^{\mkern 1.5mu \scriptscriptstyle\langle-\rangle}h^{\scriptscriptstyle\langle-\rangle},  \\
	h^{\scriptscriptstyle\langle+\rangle}h^{\scriptscriptstyle\langle-\rangle} &=& t_\cG \varepsilon (h),  \\
	h^{\scriptscriptstyle\langle+\rangle} \bract \varepsilon(h^{\scriptscriptstyle\langle-\rangle})  &=&  h,  \\
	(s (r) t (r'))^{\scriptscriptstyle\langle+\rangle} \otimes_B (s (r) t (r') )^{\scriptscriptstyle\langle-\rangle}
	&=& t(r') \otimes_B t(r),
\end{eqnarray*}
follow directly from the corresponding properties of the translation map of $\mathcal{G}$.
\end{remark}

\subsection{Takeuchi-Schneider equivalence}
\label{sec:TakSchSub}

In this section, we will prove an equivalence of categories for Hopf modules associated to Hopf algebroid quantum homogeneous spaces, generalising the well-known Takeuchi-Schneider equivalence in the Hopf algebra setup. For this, we first introduce the Takeuchi functors and show that they are mutually adjoint, continue with two technical lemmata, and then provide the aforementioned equivalence under additional (faithful) flatness assumptions. The latter will be encapsulated in the definition of a Hopf algebroid principal homogeneous space.

As a standing assumption for the rest of this article, we shall assume that all appearing bialgebroids $(\cG,R)$, resp.\ $(\cH,R)$, are flat with respect to the right $R$-action denoted $\ract$. This, in particular, guarantees that the category ${}^\cG\!\cM$ of left comodules is abelian and that the forgetful functor ${}^\cG \!\cM \to \cM_R$ is exact (see, for example, \cite[\S1.1.2]{Pos:HAOSMASCM}), which we tacitly use throughout. At times, we repeat this assumption to underline its importance.

Next, as in \S\ref{schreibtinte}, we assume that there is a projection $\pi \colon \cG \to \cH$ of right Hopf algebroids and we set $B:= \cG^{\mathrm{co}\cH}$ as in Definition \ref{boing}.
Then there is a functor ${}_B^{\mkern 3 mu \cG}\cM^{\phantom{\cG}}_{\mkern -2mu B}\to{}^\cH _B\mathcal{M}$, where we view an object $M$ in ${}_B^{\mkern 3 mu \cG}\cM^{\phantom{\cG}}_{\mkern -2mu B}$ as an object in ${}^\cH _B\mathcal{M}$ by forgetting the right $B$-action and replacing the left $\cG $-coaction $\lambda_M\colon M\to \cG \otimes_RM$ 
by the left $\cH $-coaction 
\begin{equation}
\label{Platzregen}
\overline{\lambda}_M:=(\pi\otimes_RM) \circ \lambda_M\colon M\to \cH \otimes_R M,
\end{equation}
that is, by the projection of $\lambda_M$. Using this, one can easily see:

\begin{lemma}
\label{lem:Phi}
Let $M$ be an object in ${}_B^{\mkern 3 mu \cG}\cM^{\phantom{\cG}}_{\mkern -2mu B}$.
The induced left $\cH $-coaction $\overline{\lambda}_M\colon M\to \cH \otimes_RM$ and the left $B$-action restrict to maps
$$
MB^+\to \cH \otimes_RMB^+
\qquad\text{and}\qquad
B\otimes_RMB^+\to MB^+,
$$
respectively. 
In particular, for $M = B$, one obtains that $B^+$ itself is a left $\cH $-comodule.
\end{lemma}

\begin{proof}
  Let $m\in M$ and $b\in B^+$. Then
  \vskip -.8cm
$$
\overline{\lambda}_M(m b)
=\pi(m_{(-1)}b_{(1)})\otimes_R m_{(0)} b_{(2)}
=\pi(m_{(-1)}b_{(1)})\otimes_R\overbracket{m_{(0)} \big(b_{(2)}-s \varepsilon(b_{(2)})\big)}^{\in MB^+}
$$
since $\Delta(B)\subseteq \cG \otimes_RB$ and $\pi(b)= \varepsilon(b) \lact \pi(1)=0$.
The statement that the left $B$-action restricts to a map $B\otimes_RMB^+\to MB^+$ is obvious.
\end{proof}

\noindent As in the Hopf algebraic case, let us now define two functors: the first one will be
$$
\Phi\colon{}_B^{\mkern 3 mu \cG}\cM^{\phantom{\cG}}_{\mkern -2mu B}\to{}^\cH _B\mathcal{M}, \quad M \mapsto \overline{M}:=M/MB^+
$$
on objects, and on morphisms $\phi\colon M\to M'$ in ${}_B^{\mkern 3 mu \cG}\cM^{\phantom{\cG}}_{\mkern -2mu B}$ as the induced map $\Phi(\phi)\colon\overline{M}\to\overline{M'}$ on the respective quotients.
On objects, this is well-defined by Lemma \ref{lem:Phi} and well-defined on morphisms by construction: a morphism $\phi\colon M\to M'$ clearly satisfies $\phi(MB^+)\subseteq M'B^+$, while left $B$-linearity and left $\cH $-colinearity of $\Phi(\phi)$ are induced from the left $B$-linearity and left $\cG $-colinearity of $\phi$.
The second functor
\begin{equation}
  \label{baldessen1}
\Psi\colon{}^\cH _B\mathcal{M}\to{}_B^{\mkern 3 mu \cG}\cM^{\phantom{\cG}}_{\mkern -2mu B}, \quad N \mapsto \cG \bx_\cH N 
\end{equation}
is given on objects by the cotensor product 
\begin{equation*}
  \label{baldessen2}
\cG \bx_\cH N =
\big\{g^i\otimes_Rn^i\in \cG \otimes_RN \mid g^i_{(1)}\otimes_R\pi(g^i_{(2)})\otimes_Rn^i = g^i\otimes_R n^i_{(-1)}\otimes_R n^i_{(0)} \big\}.
\end{equation*}
If $\cG_\ract$ is flat over $R$, then $\cG \bx_\cH N$ is a left $\cG$-comodule (see, for example, \cite[\S22.3]{BrzWis:CAC}) with left $\cG $-coaction simply given by the coproduct on $\cG $, while the $B$-bimodule structure is defined by
\begin{equation}
\label{Regenschirm}
    b(g^i\otimes_Rn^i)b'
    :=b_{(1)} g^i b'\otimes_Rb_{(2)} n^i
\end{equation}
for all $b,b'\in B$ and $g^i\otimes_R n^i\in \cG \bx_\cH N$. Note again that $\Delta(B)\subseteq \cG \otimes_RB$ is crucial for this definition to work. On morphisms $\psi\colon N\to N'$ in ${}^\cH _B\mathcal{M}$, let us define $\Psi(\psi):= \cG  \otimes_R\psi\colon \cG \bx_\cH N\to \cG \bx_\cH N'$. 
It is not difficult to see that $\Psi$ is well-defined, essentially by construction.

\begin{proposition}
\label{prop:Adj}
If $\cG_\ract$ is flat over $R$,
then the functors $\Phi$ and $\Psi$ are mutually adjoint, that is, there is a natural isomorphism
$$
{}_B^\cH \mathrm{Hom}(\Phi(M),N)\cong {}_B^{\,\cG}\mathrm{Hom}^{\phantom{\cG}}_B(M,\Psi(N))
$$
for all $M\in{}_B^{\mkern 3 mu \cG}\cM^{\phantom{\cG}}_{\mkern -2mu B}$ and $N\in{}^\cH _B\mathcal{M}$.
\end{proposition}
\begin{proof}
In this proof, we explicitly construct the above bijection of morphisms, following closely the strategy outlined in \cite[p.~455]{Tak:Rel} for the Hopf algebra case. Alternatively, one can describe the above adjunction as the composition of the following two adjunctions
$$
{}_B^{\,\mathcal{G}}\mathcal{M}_B \rightleftarrows {}_B^{\mathcal{H}}\mathcal{M}^{\phantom{\cH}}_B \rightleftarrows{}_B^{\mathcal{H}}\mathcal{M},
$$
the first adjunction being the corestriction of scalars together with the cotensor product functor $\mathcal{G}\bx_\cH-$, and the second the tensor product functor $-\otimes_BR$ together with the functor $\mathrm{Hom}_R(R,-)$, where we view $R$ as a left $B$-module via the algebra morphism $\varepsilon|_B\colon B\to R$.

Let us nevertheless give an explicit and detailed construction of these facts.
To this end, recall that $M\in{}_B^{\mkern 3 mu \cG}\cM^{\phantom{\cG}}_{\mkern -2mu B}$ can be understood as an object in ${}^\cH _B\mathcal{M}$ by forgetting the right $B$-action and using the coaction \eqref{Platzregen}. Let us assign to a morphism $f\colon M\to N$ in ${}^\cH _B\mathcal{M}$ the map
$$
F_f:=(\cG \otimes_Rf)\circ\lambda_M\colon M \to \cG \bx_\cH N\subseteq \cG \otimes_RN.
$$
This map is clearly left $\cG $-colinear and left $B$-linear with respect to the left $B$-action \eqref{Regenschirm}.
Let us show that $F_f$ is right $B$-linear as well if and only if $f(MB^+)=0$.
In fact, assuming right $B$-linearity, we obtain
\begin{align*}
    m_{(-1)}b_{(1)}\otimes_Rf(m_{(0)}b_{(2)})
    =F_f(mb)
    =F_f(m)b
    =m_{(-1)}b\otimes_Rf(m_{(0)})
\end{align*}
for all $b\in B$ and $m\in M$, and applying $\varepsilon\otimes_RM$ to the above yields $f(mb)=f(m\varepsilon(b))$; in particular, $f(MB^+)=0$.
Conversely, from $f(MB^+)=0$ we infer
\begin{align*}
    F_f(m b)
    &=m_{(-1)}b_{(1)}\otimes_Rf(m_{(0)}b_{(2)})\\
    &=m_{(-1)}b_{(1)}\otimes_Rf\big(m_{(0)}b_{(2)}-m_{(0)}(b_{(2)}-s \varepsilon(b_{(2)}))\big)\\
    &=m_{(-1)}b_{(1)}\otimes_Rf\big(m_{(0)}\varepsilon(b_{(2)})\big)\\
        &=m_{(-1)}t(\varepsilon(b_{(2)}))b_{(1)}\otimes_Rf(m_{(0)})\\
    &=m_{(-1)}b\otimes_Rf(m_{(0)})\\
    &=F_f(m)b,
\end{align*}
where we used \eqref{ha} in the fourth step, 
and which amounts to the right $B$-linearity of $F_f$. Thus, $F_f\colon M\to\Psi(N)$ is a morphism in ${}_B^{\mkern 3 mu \cG}\cM^{\phantom{\cG}}_{\mkern -2mu B}$ if and only if $f\colon M\to N$ descends to a map $\tilde{f}\colon\Phi(M)=M/MB^+\to N$ as a morphism in ${}^\cH _B\mathcal{M}$.
On the other hand, given a morphism $F\colon M\to \cG \bx_\cH N$, $m \mapsto g' \otimes_R n'$ in the category ${}_B^{\mkern 3 mu \cG}\cM^{\phantom{\cG}}_{\mkern -2mu B}$, define
$$
f_F:=(\varepsilon\otimes_RN)\circ F\colon M\to N.
$$
This map descends to the quotient as a map $f_F\colon\Phi(M)\to N$ since
$$
f_F(mb)
=(\varepsilon\otimes_RN)(F(mb))
=(\varepsilon\otimes_RN)(F(m)b)
=\varepsilon(g'b)n'
=\varepsilon\big(g' \bract \varepsilon(b)\big) n'
=0
$$
for all $m\in M$ and $b\in B^+$.
Furthermore, $f_F$ is a morphism in ${}^\cH _B\mathcal{M}$: it is clearly left $B$-linear but also left $\cH $-colinear since
\begin{align*}
    (f_F(m))_{(-1)}\otimes_R(f_F(m))_{(0)}
    &=\varepsilon(g') \lact n'_{(-1)}\otimes_Rn'_{(0)}\\
    &= \varepsilon(g'_{(1)}) \lact \pi(g'_{(2)})\otimes_Rn'\\
    &=\pi(g')\otimes_Rn'\\
    &=\pi(g'_{(1)})\otimes_R \varepsilon(g'_{(2)})n'\\
  &=\pi(F(m)_{(-1)})\otimes_R(\varepsilon\otimes_RN)(F(m)_{(0)})
  \\
  &=\pi(m_{(-1)})\otimes_R \big( (\varepsilon\otimes_RN) \circ F\big) (m_{(0)})
  %\\
  %&
  =\pi(m_{(-1)})\otimes_Rf_F(m_{0}),
\end{align*}
using that $F(m)=g'\otimes_Rn'$ lies in the cotensor product $\cG \bx_\cH  N$, that is, in the subspace in which  $g'\otimes_Rn'_{(-1)}\otimes_Rn'_{(0)}=g'_{(1)}\otimes_R\pi(g'_{(2)})\otimes_Rn'$ holds.

It is easy to see that the assignments $f\mapsto F_f$ and $F\mapsto f_F$ are inverse to each other: starting with a morphism $f\colon\Phi(M)\to N$ in the category ${}^\cH _B\mathcal{M}$, we have
$$
f_{(F_f)}([m])
=\big((\varepsilon\otimes_RN) \circ F_f \big) (m)
=\varepsilon(m_{(-1)})f(m_{(0)})
=f(m)
$$
for all $m\in M$. On the other hand, given a morphism $F\colon M\to\Psi(N)$ in ${}_B^{\mkern 3 mu \cG}\cM^{\phantom{\cG}}_{\mkern -2mu B}$, it follows that
\begin{align*}
    F_{(f_F)}(m)
    &=m_{(-1)}\otimes_Rf_F(m_{(0)})
    \\
    &=m_{(-1)}\otimes_R \big((\varepsilon \otimes_R N) \circ F \big)(m_{(0)})
    \\
    &=F(m)_{(-1)}\otimes_R(\varepsilon\otimes_RN)(F(m)_{(0)})\\
    &=g'_{(1)}\otimes_R\varepsilon(g'_{(2)})n'\\
    &=F(m),
\end{align*}
using the left $\cG $-colinearity of $F$ and writing $F(m)=g'\otimes_Rn'$.

To conclude, let us still show that this isomorphism is natural:
given a morphism $\phi\colon M'\to M$ in the category ${}_B^\cH \mathcal{M}$, the diagram
\begin{equation*}
\begin{tikzcd}
{}_B^\cH \mathrm{Hom}(\Phi(M),N) \arrow{r}{\cong}
\arrow{d}[swap] {{}_B^\cH\mathrm{Hom}(\Phi(\phi),N)}
& {}_B^{\mkern 2mu \cG} \mathrm{Hom}^{\phantom{\cG}}_B(M,\Psi(N))
\arrow{d}{{}_B^{\mkern 2mu \cG} \mathrm{Hom}^{\phantom{\cG}}_B(\phi,\Psi(N))}\\
{}_B^\cH \mathrm{Hom}(\Phi(M'),N) \arrow{r}{\cong}
& {}_B^{\mkern 2mu \cG} \mathrm{Hom}^{\phantom{\cG}}_B(M',\Psi(N))
\end{tikzcd}
\end{equation*}
commutes since $F_{f\circ\Phi(\phi)}=F_f\circ\phi$ for all $f\in{}_B^\cH \mathrm{Hom}(\Phi(M),N)$, which is
true because
\begin{align*}
    (F_f \circ \phi)(m')
  &=\big( (\cG \otimes_Rf) \circ (\cG \otimes_R\pi_{\overline{M}}) \circ \lambda_M \circ \phi\big)(m')
  \\
  &=
  \big((\cG \otimes_Rf)\circ (\cG \otimes_R\Phi(\phi)) \circ (\cG \otimes_R\pi_{\overline{M'}}) \circ \lambda_{M'}\big)(m')
  =F_{f\circ\Phi(\phi)}(m')
\end{align*}
for all $m'\in M'$.
Moreover, given  a morphism $\psi\colon N\to N'$ in the category ${}_B^{\mkern 3 mu \cG}\cM^{\phantom{\cG}}_{\mkern -2mu B}$, the diagram
\begin{equation*}
\begin{tikzcd}
{}_B^\cH \mathrm{Hom}(\Phi(M),N) \arrow{r}{\cong}
\arrow{d}[swap]{{}_B^\cH \mathrm{Hom}(\Phi(M),\psi)}
& {}_B^{\mkern 2mu \cG} \mathrm{Hom}^{\phantom{\cG}}_B(M,\Psi(N))
\arrow{d}{{}_B^{\mkern 2mu \cG} \mathrm{Hom}^{\phantom{\cG}}_B(M,\Psi(\psi))}
\\
{}_B^\cH \mathrm{Hom}(\Phi(M),N') \arrow{r}{\cong}
& {}_B^{ \mkern 2mu \cG} \mathrm{Hom}^{\phantom{\cG}}_B(M,\Psi(N'))
\end{tikzcd}
\end{equation*}
commutes since $F_{\psi\circ f}=\Psi(\psi)\circ F_f$ for all $f\in{}_B^\cH \mathrm{Hom}(\Phi(M),N)$, which
holds since
\begin{align*}
\big(    \Psi(\psi)\circ F_f \circ \pi_{\overline{M}}\big)(m)
    &=
    \big((\cG \bx_\cH \psi) \circ (\cG \otimes_Rf) \circ (\cG \otimes_R \pi_{\overline{M}}) \circ \lambda_M\big)(m)
       =( F_{\psi\circ f} \circ \pi_{\overline{M}})(m)
\end{align*}
for all $m\in M$.
\end{proof}

Let us continue with two technical lemmata that are crucial for the proof of the Takeuchi-Schneider equivalence below, which will constitute the main result of this section. 

\begin{lemma}
  \label{lem:Tak1}
  Let $(\cG,R)$ be a right Hopf algebroid such that both $\cG_\ract$ and
  $\due \cG \lact {}$ are $R$-flat.
For $M \in {}_B^{\mkern 3 mu \cG}\cM^{\phantom{\cG}}_{\mkern -2mu B}$, 
the map
\begin{equation*}
\begin{split}
  \zeta \colon M \otimes_B \cG & \to
  \cG_\ract \otimes_R \overline{M}, \quad
    m \otimes_B g  \mapsto m_{(-1)}g\otimes_R\pi(m_{(0)})
\end{split}
\end{equation*}
is an isomorphism of the left $\RRe$-modules
$ \due  {(M \otimes_B \cG)} \lact \bract : = \due  M \lact {} \otimes_B \cG_\bract$
and
$
\due {(\cG \otimes_R \overline{M})} \lact \bract := \due \cG \lact \bract \otimes_R \overline{M},
$
where $\pi\colon M\to M/MB^+$ denotes the projection. The map
\begin{equation*}
\begin{split}
    \cG_\ract \otimes_R \overline{M} & \to M \otimes_B \cG,
    \quad
    g\otimes_R\pi(m) \mapsto m_{\smap }\otimes_B m_{\smam }g
\end{split}
\end{equation*}
yields its inverse.
\end{lemma}

\begin{proof}
We only check here that the first map is indeed balanced over $B$, that is, that both $mb \otimes_B g$ as well as $m \otimes_B bg$ have the same image. Using $s(R)\subseteq B$ and therefore the relation $mr = ms(r)$ between the induced right $R$-action on $M$ as a left $\cG$-comodule and the induced right $R$-action on $M$ as a right $B$-module, one sees 
\begin{align*}
mb \otimes_B g \qquad \mapsto \qquad & m_{(-1)}b_{(1)}g\otimes_R\pi(m_{(0)} b_{(2)})
\\[-.5cm]
= \ & m_{(-1)}b_{(1)}g\otimes_R\pi\big(m_{(0)}b_{(2)}
+ m_{(0)}\overbracket{(s \varepsilon(b_{(2)})-b_{(2)})}^{\in B^+} \big)
    \\
 = \ & m_{(-1)}t(\varepsilon(b_{(2)}) b_{(1)}g\otimes_R\pi(m_{(0)})
    \\
    = \ & m_{(-1)}bg\otimes_R\pi(m_{(0)}),
\end{align*}
where in the penultimate step we used the Takeuchi property \eqref{ha}. The last line is clearly also the image of the element $m \otimes_B bg$, which proves $B$-balancedness.
    Let us still check that for an element in $MB^+$ the given inverse is the zero map, that is, $(mb^+)_\smap \otimes_B (mb^+)_\smam g = 0$ for $b^+ \in B^+$. Indeed,
as already shown in \eqref{torino}, 
one has     $b_\smap \otimes_R b_\smam \in B \otimes_R \cG$ for any $b \in B$,
     and therefore,
in case of $b^+ \in B^+$, one obtains
\begin{equation*}
  \begin{split}
     (mb^+)_\smap \otimes_B (mb^+)_\smam g
  & \,  = \,
     m_\smap b^+_\smap \otimes_B b^+_\smam m_\smam g
\\
     &   \, = \,
    m_\smap \otimes_B b^+_\smap b^+_\smam m_\smam g
           \,  \overset{\eqref{Tch7}}{=} \,
      m_\smap \otimes_B t \gve(b^+) m_\smam g = 0
  \end{split}
\end{equation*}
in analogy to the computation below \eqref{torino},
by means of Lemma \ref{minestrone} in the first step.
That the two maps are mutually inverse is a standard computation which we omit.
\end{proof}

\noindent As a preparatory remark for the next lemma,
given a right Hopf algebroid surjection $\pi\colon\mathcal{G}\to\mathcal{H}$ such that $B\subseteq\mathcal{G}$ is a Hopf-Galois extension, observe that the faithful right $R$-flatness of  $\mathcal{G}_\ract$ along with the left $B$-flatness of $\mathcal{G}$ induces the right $R$-flatness of $\mathcal{H}_\ract$: indeed, by right $R$-flatness of $\cG_\ract$, the functor $\cG_\ract \otimes_R - $ is exact, and therefore by left $B$-flatness
$(\cG_\ract \otimes_R -) \otimes_B \cG$ as well. Using Lemma \ref{Warschau}, from this we obtain
that
$$
(\cG_\ract \otimes_R -) \otimes_B \cG \cong (\cG_\ract \otimes_B \cG) \otimes_R - \cong (\cG_\ract \otimes_R \overline \cG)_\ract \otimes_R - \cong
\cG_\ract \otimes_R (\overline \cG_\ract \otimes_R - )
$$
is an exact functor, and by faithful flatness of $\cG_\ract$, the exactness of the functor $\overline\cG_\ract \otimes_R -$ follows. Finally, by Corollary \ref{cor:overlineG}, we conclude that $\overline\cG_\ract \cong \cH_\ract$ is right flat over $R$, as claimed.

\begin{lemma}
\label{lem:Tak2}
Let 
$B\subseteq \cG $ be a Hopf-Galois extension, let $\cG $ be flat as a 
left $B$-module, and let moreover $\cG_\ract$ be faithfully flat as a right $R$-module and $\due \cG \lact {}$ flat as a left $R$-module.
\reversemarginpar
Then one has a chain of ($R$-bimodule) isomorphisms given by
\begin{equation}
  \label{dochnoch}
(\cG \bx_\cH N)\otimes_B\cG 
\overset{\cong}{\longrightarrow}
(\cG \otimes_B\cG )\bx_\cH N
\overset{\cong}{\longrightarrow}
(\cG \otimes_R\overline{\cG })\bx_\cH N
\overset{\cong}{\longrightarrow}
\cG \otimes_RN
\end{equation}
for any $N \in{}^\cH _B\mathcal{M}$,
which amounts to an isomorphism
\begin{equation}
  \label{coflat}
\xi\colon(\cG \bx_\cH N)\otimes_B\cG 
\overset{\cong}{\longrightarrow} \cG \otimes_RN, \quad (g\otimes_R n)\otimes_B g'\mapsto gg'\otimes_Rn.
\end{equation}
In particular, if $\cG$ is faithfully left $B$-flat, then $\cG$ is faithfully $\cH$-coflat.
\end{lemma}

\begin{proof}
  The second isomorphism in the sequence \eqref{dochnoch} is clearly induced by the one from Lemma \ref{Warschau}, whereas
the third isomorphism in \eqref{dochnoch} results from the
standard identification
\begin{equation*}
(\cG \otimes_R\overline{\cG })\bx_\cH N
\overset{\cong}{\longrightarrow}
(\cG \otimes_R\cH )\bx_\cH N
\overset{\cong}{\longrightarrow}
\cG \otimes_R (\cH \bx_\cH N)
\overset{\cong}{\longrightarrow}
\cG \otimes_RN
\end{equation*}
given by
$$
(g^i\otimes_Rh^i)\otimes_Rn^i\mapsto g^i \ract \varepsilon(h^i)  \otimes_R n^i,
$$
and
with inverse $g\otimes_Rn\mapsto(g\otimes_Rn_{(-1)})\otimes_Rn_{(0)}$, obtained from the left coaction on $N$,
where we used that $\overline{\cG }\cong \cH $ as $R$-corings according to Corollary \ref{cor:overlineG}, the fact that the right $\cH $-coaction on $\cG \otimes_R\cH $ is free ({\em i.e.}, simply given by the coproduct on $\cH $), and the tensor-cotensor associativity in the middle that results from the $R$-flatness of $\cG_\ract$.

It remains to prove the first isomorphism. To give it a sense, observe first that the right $B$-action on $\cG \bx_\cH N$ is defined as
$(g^i \otimes_R n^i) b := g^i b \otimes_R n^i$. Then this isomorphism is essentially the tensor-cotensor associativity which holds by the assumed $B$-flatnesses of $\cG$.
In more detail, by definition of the cotensor product as an equaliser of a pair of maps, we can consider the exact sequence
$$
0 \longrightarrow \cG \bx_\cH N \longrightarrow \cG \otimes_RN\xrightarrow{\Delta_\cG \otimes_R N \, - \, \cG  \otimes_R \gl_N } \cG \otimes_R\cH \otimes_RN.
$$
Since $\cG $ is flat as a left $B$-module, we obtain from this an exact sequence
\begin{equation}
\label{Nudeln}
0 \to (\cG \bx_\cH N)\otimes_B \cG 
\to (\cG \otimes_RN)\otimes_B \cG 
\xrightarrow{(\Delta_\cG \otimes_R N)\, \otimes_B \cG  \, - \, ({\cG }\otimes_R \gl_N)\, \otimes_B \cG }
(\cG \otimes_R\cH \otimes_RN)\otimes_B\cG ,
\end{equation}
which, in turn, can be reorganised using two flip isomorphisms, respectively given by the right $R$-module map
$$
(\due \cG  {} \ract \otimes_RN)\otimes_B\cG \to(\due \cG  {} \ract \otimes_B\cG )\otimes_RN,
\quad
(g\otimes_R n)\otimes_B g'\mapsto (g\otimes_B g')\otimes_R n,
$$ 
which also is an isomorphism of left $R$-modules when considering the left $R$-action $\due \cG  \blact {}$ on the first tensor factor, as well as 
$$
(\due \cG  {} \ract \otimes_R \due \cH  \lact \ract \otimes_R N)\otimes_B \cG  \to (\cG _\ract \otimes_B \cG ) \otimes_R \due \cH  \lact \ract \otimes_R N,
\quad
(g\otimes_R h\otimes_R n)\otimes_B g'\mapsto (g\otimes_B g')\otimes_R h\otimes_R n,
$$
so as to obtain from the exact sequence \eqref{Nudeln} the new exact sequence
\begin{equation}
\label{Reis}
0 \to (\cG \bx_\cH N)\otimes_B\cG 
\to(\cG \otimes_B\cG )\otimes_RN
\xrightarrow{\ \rho_{\cG \otimes_B\cG }\otimes_R N \, - \, \cG \otimes_B\cG \otimes_R \gl_N \ }
(\cG \otimes_B\cG )\otimes_R\cH \otimes_RN,
\end{equation}
where we denoted
$$
\rho_{\cG \otimes_B\cG }\colon \cG \otimes_B\cG
\to
(\cG \otimes_B\cG)_\ract \otimes_R\cH ,\quad g\otimes_B g' \mapsto (g\otimes_B g'_{(1)})\otimes_R \pi(g'_{(2)})
$$
with respect to the right $R$-action $(\cG \otimes_B \cG )_{\ract} :=\cG  \otimes_B \cG_\ract$, and which is well-defined by the property \eqref{ohyes}.
Now, by definition, the kernel of  $\rho_{\cG \otimes_B \cG }\otimes_R N - \cG \otimes_B \cG \otimes_R \gl_N$ is given by $(\cG \otimes_B \cG )\bx_\cH N$. On the other hand, by exactness of the sequence \eqref{Reis}, we deduce that the kernel of $\rho_{\cG \otimes_B \cG }\otimes_R N - \cG \otimes_B \cG \otimes_R \gl_N $ is precisely given by $(\cG \bx_\cH  N)\otimes_B \cG $. Thus, $(\cG \otimes_B \cG )\bx_\cH N$ and $(\cG \bx_\cH  N)\otimes_B \cG $ are isomorphic, as desired.

The statement on coflatness follows along the lines of the standard argument in \cite[p.~391]{Tak:ANOGRG}:
by $R$-flatness, the functor $\cG_\ract \otimes_R -$ preserves exactness, hence by the isomorphism \eqref{coflat} the functor $(G \bx_\cH -) \otimes_B \cG$ does so as well. Since now $- \otimes_B \cG$ reflects exactness by faithful $B$-flatness, the functor $G \bx_\cH -$ is exact as well, which by definition means that $\cG$ is $\cH$-coflat. As $\cG_\ract \otimes_R-$ moreover reflects exactness if $\cG_\ract$ is faithfully $R$-flat, the last statement follows by an analogous argument.
\end{proof}

\begin{definition}
  \label{def:PHS}
  Let $\pi\colon\cG\to\cH$ be a (right) Hopf algebroid surjection (over left $R$-bialgebroids) such that
  $\due \cG \lact {}$ is $R$-flat.
  If $\cG_\ract$ is faithfully $R$-flat and $\cG$ faithfully left $B$-flat, we call $B$ a \emph{Hopf algebroid principal homogeneous space} or, slightly more cumbersome, a {\em principal quantum homogeneous space for a right Hopf algebroid
surjection}.
\end{definition}

After these technical preparations, we are able to state (and prove) the main result in this section, that is, a Takeuchi-Schneider equivalence for (right) Hopf algebroids as an extension of classical bialgebra theory, see \cite[Thm.~3.7]{Schneider}, building on \cite[Thm.~1]{Tak:Rel}, and with additional $B$-module structure as mentioned, \emph{e.g.}, in \cite[\S2.3]{UlKr}. See also \cite[Thm.~2.11]{KaoGhoSarVer:CTFHAWATAG} for an interesting contribution in this direction in the Hopf algebroid setting.

\begin{theorem}[Takeuchi-Schneider equivalence for Hopf algebroids]
  \label{Kuchen}
If $B\subseteq \cG$ is a Hopf algebroid principal homogeneous space, then there is an equivalence
$$
{}_B^{\mkern 3 mu \cG}\cM^{\phantom{\cG}}_{\mkern -2mu B} \cong {}^\cH_B \cM
$$
of categories.
\end{theorem}

\begin{proof}
Let $F\colon M\to \cG \bx_\cH N$ be a morphism in ${}_B^{\mkern 3 mu \cG}\cM^{\phantom{\cG}}_{\mkern -2mu B}$ and $f\colon\overline{M}\to N$ the corresponding morphism in ${}^\cH _B\mathcal{M}$ according to Proposition \ref{prop:Adj}.
The diagram
\begin{equation*}
\begin{tikzcd}
M\otimes_B\cG  \arrow{rrr}{\cong \text{ by Lem.~\ref{lem:Tak1}}}
\arrow{d}[swap]{F\otimes_B\cG }
& & & \cG \otimes_R\overline{M} \arrow{d}{\cG \otimes_Rf}\\
(\cG \bx_\cH N)\otimes_B\cG  \arrow{rrr}{\cong\text{ by Lem.~\ref{lem:Tak2}}}
& & & \cG \otimes_RN
\end{tikzcd}
\end{equation*}
commutes
because
\begin{align*}
	((\mathcal{G}\otimes_Rf)\circ\zeta)(m\otimes_Bg)
	&=(\mathcal{G}\otimes_Rf)(m_{(-1)}g\otimes_R\pi(m_{(0)}))\\
	&=m_{(-1)}g\otimes_Rf(\pi(m_{(0)}))\\
	&=\xi((m_{(-1)}\otimes_Rf(\pi(m_{(0)})))\otimes_Bg)\\
	&=(\xi\circ(F\otimes_B\mathcal{G}))(m\otimes_Bg).
\end{align*}
Then, by faithful flatness, $F$ is an isomorphism if and only if $f$ is an isomorphism. Consider the two special cases:
\begin{enumerate}
  \compactlist{99}
\item For $N=\overline{M}$ and $f=\mathrm{id}$, we have that
$$
F\colon M\xrightarrow{\cong}\cG \bx_\cH \overline{M}, \quad
m\mapsto m_{(-1)}\otimes_R\pi(m_{(0)})
$$
is an isomorphism, that is, $\Psi(\Phi(M))\cong M$.

\item For $M=\cG \bx_\cH N$ and $F=\mathrm{id}$, we have that
$$
f\colon\overline{\cG \bx_\cH N}\xrightarrow{\cong}N,\quad
\pi(g^i\otimes_Rn^i)\mapsto\varepsilon(g^i)n^i
$$
is an isomorphism, that is, $\Phi(\Psi(N))\cong N$.
\end{enumerate}
Together with Proposition \ref{prop:Adj}, this proves the desired equivalence.
\end{proof}

\section{Differential calculi for Hopf algebroid principal homogeneous spaces}

\subsection{The general case}
\label{sec:calcPHS}

Consider a right Hopf algebroid surjection $\pi\colon \cG \to \cH$ over the same base algebra $R$. As in Definition~\ref{boing}, denote the associated left Hopf kernel of $\pi$ by $B$ again and set $B^+=B\cap\ker\varepsilon$.
By Lem\-ma~\ref{spinat}\,(ii), the algebra $B$ is a left $\cG$-comodule algebra with $R$-bilinear coaction $\gl_B$. 
Recall that $B^+$ is an object in ${}^\cH_B\mathcal{M}$ via left multiplication and left $\cH$-coaction given by
$$
\gl_{B^+} :=(\pi\otimes_R B^+)\circ \Delta|_{B^+}\colon
   B^+ \to \cH \otimes_R B^+,
   $$
   where $\gD$ denotes the coproduct in $\cG$.
Analogously to the Hopf algebra case, we can then state:

\begin{proposition}
  Let $B$ be the 
  principal homogeneous space for a right Hopf algebroid surjection $\pi\colon \cG \to \cH$.
  For any subobject $I \subseteq B^+$ in ${}^\cH_B\mathcal{M}$, we obtain a left $\cG$-covariant first order differential calculus $(\Omega, \dd)$ on $B$.
  Explicitly, define
  $$
  \Omega:=\Psi(B^+/I)=\cG \bx_\cH B^+/I,
$$
  with $B$-bimodule and left $\cG$-comodule structure discussed before Proposition \ref{prop:Adj},
and where $\Psi$ denotes the functor introduced in Eq.~\eqref{baldessen1}.
The differential 
is given by
$$
\dd \colon B\to\Omega, \quad 
b \mapsto (\cG \otimes_R \pi_I)(\Delta b - b\otimes_R 1),
$$
where $\pi_I\colon B^+\to B^+/I$ denotes the canonical projection.
\end{proposition}

\begin{proof}
We begin by showing that for $I=0\subseteq B^+$, we obtain a left $\cG$-covariant first order differential calculus on $B$ given by $\gG = \Psi(B^+)=\cG\bx_\cH B^+$, with differential defined as $\dd b:=\Delta b -b\otimes_R 1$ for all $b\in B$. 
Explicitly, its $B$-bimodule structure reads
\begin{equation}
  \label{sandhamn}
b (g^i\otimes_R \tilde b^i) b'
:=b_{(1)}g^ib'\otimes_R b_{(2)} \tilde b^i
\end{equation}
for $\tilde b^i \in B^+$, while the left $\cG$-coaction is simply given by $\gl_\Omega :=\Delta\otimes_R B^+$, where as before $\gD$ denotes the coproduct in $\cG$.
The differential is well-defined since $\Delta b - b\otimes_R 1 \in \cG \otimes_R B^+$ for all $b\in B$ and moreover satisfies the Leibniz rule \eqref{wasserrutsche}
\begin{align*}
    (\dd b)b'+b \dd {\mkern 1mu} b'
  &=(b_{(1)}\otimes_R b_{(2)} - b\otimes_R 1) b'+b (b'_{(1)}\otimes_R b'_{(2)}-b'\otimes_R1)
  \\
  &=
  b_{(1)}b'\otimes_R b_{(2)} - bb'\otimes_R 1
  + b_{(1)}b'_{(1)}\otimes_R b_{(2)}b'_{(2)} - b_{(1)}b'\otimes_R b_{(2)}
  \\
  &=b_{(1)}b'_{(1)}\otimes_R b_{(2)}b'_{(2)} - bb'\otimes_R 1
  \\
    &=\dd (bb')
\end{align*}
as required, where the $B$-bimodule structure used in the second step is that from \eqref{sandhamn}.
The differential is also left $\cG$-colinear since
\begin{align*}
    \gl_\Omega(\dd b)
    =\gl_\Omega (b_{(1)}\otimes_R b_{(2)} - b\otimes_R 1)
    = b_{(1)}\otimes_R b_{(2)} \otimes_R b_{(3)} - b_{(1)}\otimes_R b_{(2)}\otimes_R 1
    =(\cG \otimes_R \dd) (\Delta b).
\end{align*}
It remains to prove surjectivity of $\mathrm{d}$ in case $\cG$ is a right Hopf algebroid (over a left bialgebroid):
let $g^i\otimes_R b^i \in \cG \bx_\cH B^+$, that is to say, 
\begin{equation}
  \label{immernochsoheisz}
  g^i_{(1)}\otimes_R \pi(g^i_{(2)}) \otimes_R b^i=g^i \otimes_R \pi(b^i_{(1)}) \otimes_R b^i_{(2)},
\end{equation}
and 
  let us show that $b^i_{\smap } \otimes_R b^i_{\smam } g^i \in  B_\bract \otimes_R \due B \lact {}$.
First of all, we show
\begin{align*}
    b^i_{\smap } \otimes_R  b^i_{\smam (1)}g^i_{(1)} \otimes_R  \pi(b^i_{\smam (2)} g^i_{(2)})
    &= b^i_{\smap \smap } \otimes_R  b^i_{\smap \smam }g^i_{(1)} \otimes_R  \pi(b^i_{\smam })\pi(g^i_{(2)})
    \\
    &= b^i_{(2)\smap \smap } \otimes_R  b^i_{(2)\smap \smam }g^i \otimes_R  \pi(b^i_{(2)\smam }b^i_{(1)})
    \\
    &= b^i_{\smap } \otimes_R  b^i_{\smam }g^i \otimes_R  1_\cH,
\end{align*}
where we used Eqs.~\eqref{Tch5}, \eqref{immernochsoheisz}, and \eqref{Tch3}, which means that $b^i_{\smap } \otimes_R  b^i_{\smam }g^i\in \cG \otimes_R  B$.
Then, as elements in
$(\cG_\ract \otimes_R \due \cH \lact {})_\bract \otimes_R \due B \lact {} :=
\cG_{\ract, \bract} \otimes_R \due \cH \lact {} \otimes_R \due B \lact {}$,
by using
\eqref{Tch4} one sees that
\begin{align*}
    b^i_{\smap (1)} \otimes_R  \pi(b^i_{\smap (2)}) \otimes_R  b^i_{\smam }g^i        
    = b^i_{(1)\smap } \otimes_R  \pi(b^i_{(2)}) \otimes_R  b^i_{(1)\smam }g^i
    = b^i_{\smap } \otimes_R  1_\cH \otimes_R   b^i_{\smam }g^i,
\end{align*}
 and hence
$b^i_{\smap } \otimes_R  b^i_{\smam }g^i \in \ker (\gvf \otimes_R B)$.
From the faithful left $B$-flatness of $\cG$ along with the left $R$-flatness of $\due \cG \lact {}$, we obtain the flatness of $B$ as a left $R$-module and therefore $\ker(\gvf \otimes_R B) = \ker \gvf \otimes_R B = B \otimes_R B$, and thus $b^i_{\smap } \otimes_R  b^i_{\smam }g^i \in B\otimes_R  B$, as desired.
Finally,
\begin{align*}
    b^i_{\smap } \dd (b^i_{\smam } g^i)
    &=b^i_{\smap }(b^i_{\smam (1)} g^i_{(1)}\otimes_R b^i_{\smam (2)} g^i_{(2)} - b^i_{\smam } g^i\otimes_R 1)\\
    &=b^i_{\smap (1)}b^i_{\smam (1)} g^i_{(1)}\otimes_R b^i_{\smap (2)}b^i_{\smam (2)} g^i_{(2)} - b^i_{\smap (1)}b^i_{\smam } g^i\otimes_R b^i_{\smap (2)}\\
    &= \varepsilon(b^i) \lact g^i_{(1)}\otimes_R g^i_{(2)} - g^i\otimes_R b^i\\
    &=- g^i\otimes_R b^i,
\end{align*}
using \eqref{Tch7} and $b^i \in B^+$, which shows surjectivity of the calculus. In summary, $(\cG \bx_\cH B^+, \dd)$ is a left $\cG$-covariant first order differential calculus on $B$.
Furthermore, $(\cG \bx_\cH B^+, \dd)$ is isomorphic to the universal calculus $(\gO^u_B, \dd_u)$ on $B$ by means of the isomorphism
\begin{align}
  \label{ichholjetztdochnochnKaffee}
  \xi\colon\Psi(B^+)=\cG \bx_\cH B^+  \xrightarrow{\cong}\gO^u_B, \quad g^i\otimes_R b^i  \mapsto b^i_{\smap }\otimes_R b^i_{\smam }g^i,
\end{align}
with inverse $\gO^u_B\to \cG \bx_\cH B^+$, $b^i\otimes_R  \tilde b^i\mapsto b^i_{(1)} \tilde b^i\otimes_R b^i_{(2)}$ for $b^i, \tilde b^i \in B$.
Note that $\xi$ is well-defined since $b^i_{\smap }\otimes_R b^i_{\smam }g^i\in B\otimes_R B$
and
$b^i_{\smap }b^i_{\smam }g^i = g^i \ract \varepsilon(b^i)=0$
for all $g^i\otimes_R b^i\in \cG\bx_\cH B^+$.
Moreover, $\xi$ is left $B$-linear since
\begin{align*}
    \xi(b(g^i\otimes_R b^i))
    &=\xi(b_{(1)}g^i\otimes_R b_{(2)}b^i)
    \\
    &=(b_{(2)}b^i)_{\smap }\otimes_R (b_{(2)}b^i)_{\smam }b_{(1)}g^i\\
    &=b_{(2)\smap } b^i_{\smap }\otimes_R b^i_{\smam }b_{(2)\smam }b_{(1)}g^i\\
    &=bb^i_{\smap }\otimes_R b^i_{\smam }g^i\\
    &=b\xi(g^i\otimes_R b^i),
\end{align*}
using Eqs.~\eqref{Tch6} \& \eqref{Tch3},
while $\xi$ is clearly right $B$-linear as well.
To show that $\xi$ is, moreover, left $\cG$-colinear, recall that $\cG\bx_\cH B^+$ is endowed 
with the left $\cG$-coaction $\Delta\otimes_R  B^+$ if $\cG_\ract$ is $R$-flat, while $\gO^u_B$ is a left $\cG$-comodule by means of the diagonal coaction $\gl_{\gG^u_B}$. Then, as elements in $\cG_\ract \otimes_R (B_\bract \otimes_R \due B \lact {})$, we have:
\begin{align*}
    \gl_{\Omega^u_B} (\xi(g^i\otimes_R b^i))
    &=\gl_{\Omega^u_B} (b^i_{\smap }\otimes_R b^i_{\smam }g^i)
    \\
    &=b^i_{\smap (1)}(b^i_{\smam }g^i)_{(1)}\otimes_R  b^i_{\smap (2)}\otimes_R (b^i_{\smam }g^i)_{(2)}
    \\
    &=b^i_{(1)\smap }(b^i_{(1)\smam }g^i)_{(1)}\otimes_R b^i_{(2)}\otimes_R (b^i_{(1)\smam }g^i)_{(2)}
    \\
    &=b^i_{(1)\smap \smap }b^i_{(1)\smap \smam }g^i_{(1)}\otimes_R b^i_{(2)}\otimes_R b^i_{(1)\smam }g^i_{(2)}
    \\
    &=g^i_{(1)} \ract \varepsilon(b^i_{(1)\smap }) \otimes_R b^i_{(2)}\otimes_R b^i_{(1)\smam }g^i_{(2)}
    \\
    &=g^i_{(1)}\otimes_R b^i_{\smap }\otimes_R b^i_{\smam }g^i_{(2)}
    \\
    &=g^i_{(1)}\otimes_R \xi(g^i_{(2)}\otimes_R b^i)
    \\
    &=(\cG \otimes_R \xi)\big(\gl_{\mathcal{G}{\scalebox{0.7}{$\bx$}}_\mathcal{H} B^+} (g^i\otimes_R b^i)\big),
\end{align*}
using, once more, the relations \eqref{Tch4}--\eqref{Tch7} for the translation map, which 
proves that $\xi$ is left $\cG$-colinear. That $\xi$ is invertible follows from
\begin{align*}
    \xi^{-1}(\xi(g^i\otimes_R b^i)) =\xi^{-1}(b^i_{\smap }\otimes_R b^i_{\smam }g^i) =b^i_{\smap (1)}b^i_{\smam }g^i\otimes_R b^i_{\smap (2)} =g^i\otimes_R b^i,
\end{align*}
and likewise the other way round.
Furthermore, the left $\cG$-colinear $B$-bimodule isomorphism $\xi$ intertwines the differentials,  
\begin{align*}
    \xi( \dd b)
    =\xi(\Delta b -b\otimes_R 1)
    =b_{(2)\smap }\otimes_R b_{(2)\smam }b_{(1)}-1_{\smap }\otimes_R 1_{\smam }b
    =b\otimes_R 1-1\otimes_R b
    = - \dd_u b,
\end{align*}
and, hence, also is
an isomorphism of left $\cG$-covariant first order differential calculi on $B$.
Finally, let $I\subseteq B^+$ be any subobject in ${}_B^\cH\mathcal{M}$; in particular, there is a short exact sequence
$$
0\to I\hookrightarrow B^+\twoheadrightarrow B^+/I\to 0
$$
in ${}_B^\cH\mathcal{M}$.
Since by the resulting $\cH$-coflatness of $\cG$, see Lemma \ref{lem:Tak2}, the functor $\Psi$ is exact, 
we therefore obtain a short exact sequence
$$
0\to \Psi(I)\hookrightarrow \Psi(B^+)\twoheadrightarrow \Psi(B^+/I)\to 0
$$
in
${}_B^{\mkern 3 mu \cG}\cM^{\phantom{\cG}}_{\mkern -2mu B}$.
Thus,
$$
\Omega:=\Psi(B^+/I)=\cG\bx_\cH B^+/I\cong\Psi(B^+)/\Psi(I)
$$
are isomorphic as left $\cG$-covariant $B$-bimodules.
As a quotient of the universal calculus $\gO^u_B\cong\Psi(B^+)$, the left $\cG$-covariant $B$-bimodule $\Omega:=\Psi(B^+/I)$ becomes a left $\cG$-covariant first order differential calculus on $B$ with differential induced by the universal differential $\dd_u$. Explicitly, $\dd_u b=(\cG \otimes_R \pi_I)(\Delta b -b\otimes_R 1)$, as claimed.
\end{proof}

\begin{proposition}
  Let $B$ be the
  principal homogeneous space for a right Hopf algebroid
surjection $\pi\colon \cG \to \cH $. Furthermore, let $(\Omega, \dd)$ be a left $\cG $-covariant first order differential calculus on $B$ and let $M\subseteq\gO^u_B$ be, as in Proposition \ref{schweineschmalz}, the subobject in ${}_B^{\mkern 3 mu \cG}\cM^{\phantom{\cG}}_{\mkern -2mu B}$ of the universal calculus on $B$ such that $\Omega\cong\gO^u_B/M$. Then we obtain a subobject $I\subseteq B^+$ in ${}^\cH _B\mathcal{M}$ such that
\begin{align*}
  \Phi(M)=M/MB^+\xrightarrow{\cong}I,
  \quad
[b {\mkern 1mu} \dd _u b'] \mapsto b b' - b \bract \gve(b')
\end{align*}
is an isomorphism.
\end{proposition}

\begin{proof}
In a first step, we show that $\Phi(\gO^u_B)$ with $\gO^u_B =\ker m_B$ is isomorphic to $B^+$ as a left $\cH $-covariant left $B$-module.
This isomorphism in ${}_B^\cH \mathcal{M}$ is given by
\begin{align}
\label{Kekse}
\zeta\colon\Phi(\gO^u_B)=\gO^u_B/\gO^u_BB^+ \to B^+,
\quad
[b {\mkern 1mu} \dd _u b'] \mapsto b b' - b \bract \gve(b')
\end{align}
with inverse $\zeta^{-1}\colon B^+\to\Phi(\gO^u_B)$, $b\mapsto[\dd _u b]$.
Since $\gve(bb') - \gve(b \bract \gve(b')) = \gve(bb') - \gve(bb') = 0$, 
the image of $\zeta$, seen as a map $\gO^u_B\to B$, lives in $B^+$. It is well-defined on the quotient, and thus well-defined as a map $\zeta\colon\Phi(\gO^u_B)\to B^+$ since for $b,b'\in B$ and $c\in B^+$, we have 
$$
\zeta \colon (b{\mkern 1mu} \dd_u b') c = b {\mkern 1mu} \dd_u (b'c) - bb' {\mkern 1mu} \dd_uc
\, \mapsto \, b\big(b'c - s\gve(b'c) - b'c + b' \bract \gve(c)\big) = 0, 
$$
again because $\gve$ here is an algebra morphism.
The map $\zeta^{-1}$ is well-defined on the nose, and let us show that it inverts $\zeta$, indeed: for all $b\in B^+$, we have $\zeta(\zeta^{-1}(b))=\zeta([\dd _u b])=b$,
while for all $b,b'\in B$, we obtain
\begin{equation*}
\begin{split}
    \zeta^{-1}\big(\zeta([b {\mkern 1mu} \dd _u b'])\big)
        &= \zeta^{-1}\big( b b' - b \bract \gve(b') \big)
        =     [\dd_u ( b b')] - [\dd_u b] \bract \gve(b') =
        [\dd_u ( b b')] - [(\dd_u b)b']
        =  [b {\mkern 1mu} \dd _u b'].
\end{split}
\end{equation*}
Hence, $\zeta$ and $\zeta^{-1}$ are mutually inverse.
Let us still show that $\zeta^{-1}$ is left $B$-linear and left $\cH $-colinear as well:
first, for $b\in B$ and $b^+\in B^+$, we obtain
$$
\zeta^{-1}(bb^+)
=[\dd _u(bb^+)]
=[\dd _u(b)b^++b {\mkern 1mu} \dd _u b^+]
=[\dd _u b ] \bract \varepsilon(b^+) + [b {\mkern 1mu} \dd _u b^+]
=0+b\zeta^{-1}(b^+),
$$
while for all $b^+\in B^+$, one has
\begin{align*}
    \gl_{\Phi(\gO^u_B)}(\zeta^{-1}(b^+))
    =\gl_{\Phi(\gO^u_B)}\big([\dd _u b^+]\big)
    =\pi(b^+_{(1)})\otimes_R [\dd _u b^+ _{(2)}]
    =(\cH \otimes_R \zeta^{-1})(\gl_{B^+}(b^+)),
\end{align*}
which means colinearity.
As for the more general statement, 
by Example \ref{freibadoderbadesee} and Proposition \ref{schweineschmalz}, for every left $\cG $-covariant first order differential calculus $(\Omega,\mathrm{d})$ on $B$ there is a short exact sequence
$$
0\to M\hookrightarrow\gO^u_B\twoheadrightarrow\Omega\to 0
$$
in ${}_B^{\mkern 3 mu \cG}\cM^{\phantom{\cG}}_{\mkern -2mu B}$. If now $\cG_\ract$ is faithfully $R$-flat, then $\cG$  is faithfully $\cH$-coflat by Lemma \ref{lem:Tak2}; hence,  $\Phi$ is an exact functor since $\Psi$ is so and both functors establish an equivalence of categories, see Theorem \ref{Kuchen}. We therefore
obtain a short exact sequence
$$
0\to \Phi(M)\hookrightarrow\Phi(\gO^u_B)\twoheadrightarrow\Phi(\Omega)\to 0
$$
in ${}_B^\cH \mathcal{M}$. By definition $I:=\zeta(\Phi(M))\subseteq B^+$ is a subobject in ${}_B^\cH \mathcal{M}$ and isomorphic to $\Phi(M)$, which is the isomorphism in the statement.
\end{proof}

\begin{theorem}
  \label{thm:herm}
  If $B$ is the principal homogeneous space for a right Hopf algebroid surjection $\pi\colon \cG \to \cH $,
  then
  $$
\begin{Bmatrix}
\text{left $\cG $-covariant first order}\\
\text{differential calculi on $B$}
\end{Bmatrix}
\overset{1:1}{\longleftrightarrow}
\begin{Bmatrix}
\text{subobjects $I\subseteq B^+$}\\
\text{in ${}^\cH _B\mathcal{M}$}
\end{Bmatrix}
$$
is a one-to-one correspondence.
\end{theorem}

\begin{proof}
Recall first from \eqref{ichholjetztdochnochnKaffee} the isomorphism 
\begin{align*}
  \xi\colon\Psi(B^+)=\cG \bx_\cH B^+ \xrightarrow{\cong}\gO^u_B,
  \quad g^i\otimes_R b^i&\mapsto b^i_{\smap }\otimes_R b^i_{\smam }g^i
\end{align*}
in ${}_B^{\mkern 3 mu \cG}\cM^{\phantom{\cG}}_{\mkern -2mu B}$,
with inverse $b^i\otimes_R g^i\mapsto b^i_{(1)}g^i\otimes_R b^i_{(2)}$, and second from \eqref{Kekse} the isomorphism  
\begin{align*}
  \zeta\colon\Phi(\gO^u_B)=\gO^u_B/\gO^u_BB^+  \xrightarrow{\cong} B^+, \quad
  [b {\mkern 1mu} \dd _u b'] \mapsto b b' - b \bract \gve(b')
\end{align*}
in ${}_B^\cH \mathcal{M}$,
with inverse $b\mapsto[\dd _u b]$.
Let $(\Omega, \dd )$ be a left $\cG $-covariant first order differential calculus on $B$. Then, as just seen right above, there exists a subobject $M\subseteq\gO^u_B$ in ${}_B^{\mkern 3 mu \cG}\cM^{\phantom{\cG}}_{\mkern -2mu B}$ such that
$$
0\to M\hookrightarrow\gO^u_B\twoheadrightarrow\Omega\to 0
$$
is a short exact sequence in ${}_B^{\mkern 3 mu \cG}\cM^{\phantom{\cG}}_{\mkern -2mu B}$. This gives a subobject $I:=\zeta(\Phi(M))\subseteq B^+$ in ${}_B^\cH \mathcal{M}$.
On the other hand, every subobject $I\subseteq B^+$ in ${}_B^\cH \mathcal{M}$ gives rise to a short exact sequence
$$
0\to I\hookrightarrow B^+\twoheadrightarrow B^+/I\to 0
$$
in ${}_B^\cH \mathcal{M}$ so as to obtain a subobject $M:=\xi(\Psi(I))\subseteq\gO^u_B$ in ${}_B^{\mkern 3 mu \cG}\cM^{\phantom{\cG}}_{\mkern -2mu B}$. To see that these assignments are mutually inverse,
one verifies that, given subobjects $I\subseteq B^+$ in ${}^\cH _B\mathcal{M}$ and $M\subseteq\Omega^u_B$ in ${}_B^{\mkern 3 mu \cG}\cM^{\phantom{\cG}}_{\mkern -2mu B}$, the maps
\begin{equation}
    \label{eq:2isos}
\begin{split}
  I & \to \zeta\big(\Phi(\xi(\Psi(I)))\big),
  \\
M & \to \xi \big(\Psi(\zeta(\Phi(M)))\big),
\end{split}
\quad
\begin{split}
  b & \mapsto \zeta\big([\xi(b_{(1)}\otimes_R b_{(2)}-b\otimes_R1)]\big),
  \\
	m & \mapsto \xi \big(m_{(-1)}\bx_\cH\zeta([m_{(0)}]) \big),
\end{split}
\end{equation}
are bijections in ${}^\cH _B\mathcal{M}$ and ${}_B^{\mkern 3 mu \cG}\cM^{\phantom{\cG}}_{\mkern -2mu B}$, respectively. In fact, for all $b\in I$ we obtain 
\begin{align*}
	\zeta([\xi(b_{(1)}\otimes_R b_{(2)}-b\otimes_R1)])
	&=\zeta(b_{(2)\smap}\otimes_Rb_{(2)\smam}b_{(1)}-1_{\smap}\otimes_R1_{\smam}b)\\
	&=\zeta([b\otimes_R1-1\otimes_Rb])\\
	&=-\zeta([\mathrm{d}_ub])\\
	&=-b,
\end{align*}
which implies injectivity, as well as left $B$-linearity and left $\mathcal{H}$-colinearity of the first map in \eqref{eq:2isos}. As for surjectivity, for arbitrary $g^i\otimes_Rb^i\in\mathcal{G}\bx_\cH I$, as just shown, $\varepsilon(g^i) \lact b^i\in I$ is mapped to $-\varepsilon(g^i) \lact b^i$, from which
\begin{align*}
	- \varepsilon(g^i) \lact b^i
	=b^i_{\smap }b^i_{\smam  }g^i-b^i_{\smap }\bract\varepsilon(b^i_{\smam  }g^i)
	&=\zeta\big([b^i_{\smap }\mathrm{d}_u(b^i_{\smam  }g^i)]\big) \\
	&=\zeta\big([b^i_{\smap }\otimes_Rb^i_{\smam  }g^i]\big) =\zeta\big([\xi(g^i\otimes_Rb^i)]\big)
\end{align*}
follows. As a non-zero multiple of the identity map, the first assignment in \eqref{eq:2isos} maps the $\Bbbk$-module $I$ to itself and thus gives the equality $I=\zeta\big(\Phi(\xi(\Psi(I)))\big)$.
To prove that the second map in \eqref{eq:2isos} is an isomorphism, we employ Theorem \ref{Kuchen}, which gives an isomorphism $M\to\Psi(\Phi(M))=\mathcal{G}\bx_\cH\overline{M}$ in ${}_B^{\mkern 3 mu \cG}\cM^{\phantom{\cG}}_{\mkern -2mu B}$, and compose it with 
\begin{equation}
  \label{eq:idzeta}
\cG\otimes_R\zeta \colon\mathcal{G}\bx_\cH\overline{M}\to\mathcal{G}\bx_\cH\zeta(\overline{M}).
\end{equation} 
Since $\Phi$ is exact, the isomorphism $\zeta\colon\Phi(\Omega^u_B)\to B^+$ (co)restricts to an isomorphism $\zeta\colon\Phi(M)\to\zeta(\Phi(M))$. Thus \eqref{eq:idzeta} is an isomorphism by the right $\mathcal{H}$-coflatness of $\mathcal{G}$, therefore $M$ is isomorphic to $\mathcal{G}\bx_\cH\zeta(\overline{M})$, and hence  also isomorphic to $\xi\big(\mathcal{G}\bx_\cH\zeta(\overline{M})\big)$. Explicitly, for $m=b\mkern 1.5mu\mathrm{d}_ub'\in M$ with $b,b'\in B$, we obtain
\begin{align*}
  m\mapsto&~\xi\big(m_{(-1)}\bx_\cH\zeta([m_{(0)}])\big)
  \\
	&=\xi\big(b_{(1)}b'_{(1)}\bx_\cH\zeta([b_{(2)}\mathrm{d}_ub'_{(2)}])\big)
\\
        &=\xi\big(b_{(1)}b'_{(1)}\bx_\cH( b_{(2)}b'_{(2)}-b_{(2)}\bract\varepsilon(b'_{(2)}))\big)
        \\
	&=(b_{(2)}b'_{(2)})_{\smap }\otimes_R(b_{(2)}b'_{(2)})_{\smam}b_{(1)}b'_{(1)}
        -(b_{(2)}\bract\varepsilon(b'_{(2)}))_{\smap }\otimes_R(b_{(2)}\bract\varepsilon(b'_{(2)}))_{\smam }b_{(1)}b'_{(1)}
\\
	&=bb'\otimes_R1
- b_{(2)\smap }\otimes_R b_{(2) \smam }b_{(1)}t \varepsilon(b'_{(2)}) b'_{(1)}\\
	&=bb'\otimes_R1
	-b\otimes_Rb'\\
	&=-m,
\end{align*}
which shows that the second map in \eqref{eq:2isos} is a non-zero multiple of the identity and thus $M=\xi \big(\Psi(\zeta(\Phi(M)))\big)$, as before.
\end{proof}

\begin{remark}
\label{noworo}
Comparing Theorems \ref{thm:Wor} \& \ref{thm:herm},
one might wonder whether the former, that is, the Woronowicz kind of classification cannot be seen as a special case of the latter, as is the case in classical Hopf algebra theory. However, in view of Remark \ref{tobeornottobeakernel}, there does not seem to be any Hopf algebroid surjection $\pi \colon \cG \to \RRe$ such that $B \cong \cG$, which would correspond to Theorem \ref{thm:Wor}, and which is why it has been proven here separately.
\end{remark}

\subsection{Calculi on the scalar extension principal homogeneous space}
\label{sec:QHScalc}

  With the help of Theorem \ref{thm:herm}, we can now construct examples of differential calculi on the previously discussed
  class of principal homogeneous spaces for scalar extension Hopf algebroids, see \S\ref{sec:HomogeneousScalarExt}.
  Let us recall the setup: let $\pi\colon G\to H$ be a surjection of Hopf algebras with invertible antipodes and $B:=G^{\mathrm{co}H}\subseteq G$ a faithfully flat Hopf-Galois extension. Then, given a (braided) commutative monoid $A$ in ${}_H\mathcal{YD}^G$,
by Proposition \ref{prop:smashhomogeneous}
the left Hopf kernel $A\otimes B$ is a Hopf algebroid principal homogeneous space for the right Hopf algebroid surjection
$$
\pi_\#\colon\underbracket{{}_\pi F(A)\# G}_{=:\mathcal{G}}\to\underbracket{F^\pi(A)\#H}_{=:\mathcal{H}}.
$$ 
Given subobjects $I_A\subseteq A$ in ${}_H\mathcal{M}\cap{}_A\mathcal{M}$ and $I_B\subseteq B^+$ in ${}_B^H\mathcal{M}$, we then obtain a subobject $I:=I_A\otimes I_B\subseteq A\otimes B$ in ${}_{A\otimes B}^{\mathcal{H}}\mathcal{M}$: in fact, $I$ is closed under the left $A\otimes B$-action since
$$
(a\otimes b) (x\otimes y)
=a(
\pi(b_{(1)}) \mkern 1 mu x
)
\otimes
b_{(2)}y
\in I_A\otimes I_B
$$
for all $a\in A$, $b\in B$, $x\in I_A$, $y\in I_B$. 
Moreover, $I$ is closed under the left $\mathcal{H}$-coaction, as seen by
\begin{align*}
 \overline{\lambda}_I
  (x\otimes y)
	&=\pi_\#(x\otimes y_{(1)})\otimes_A(1\otimes y_{(2)})\\
	&=\big(x\otimes\pi(y_{(1)})\big)\otimes_A(1\otimes y_{(2)})\\
 &=\big(1\otimes\pi(y_{(2)})\big)(\pi\big(S^{-1}(y_{(1)})\big) \mkern 1 mu x\otimes 1)\otimes_A(1\otimes y_{(3)})
 \\
 &=\big(1\otimes\pi(y_{(2)})\big)
 \otimes_A
 \big(\underbracket{\pi\big(S^{-1}(y_{(1)})\big) \mkern 1 mu  x\otimes y_{(3)}}_{\in I_A\otimes I_B}\big),
\end{align*}
where $x\in I_A$ and $y\in I_B$.
Then Theorem \ref{thm:herm} implies the following result:

\begin{proposition}
  \label{prop:CalcSmashHom}
  Consider a principal homogeneous space $B =G^{\mathrm{co}H}$ for a Hopf algebra surjection $\pi\colon G\to H$ and a braided commutative monoid $A$ in ${}_H\mathcal{YD}^G$. Let us abbreviate $\cG:=A\# G$ resp.\ $\cH:=A\#H$.
Then
  for every left $H$-ideal and left $A$-ideal $I_A\subseteq A$ and any subobject $I_B\subseteq B^+$ in ${}_B^H\mathcal{M}$, we obtain a left $\cG$-covariant first order differential calculus $(\Omega,\dd )$ on the Hopf algebroid principal homogeneous space $A\otimes B=\mathcal{G}^{\mathrm{co}\mathcal{H}}$. Explicitly,
$$
\Omega:=\mathcal{G}\bx_\mathcal{H}(A\otimes B^+)/(I_A\otimes I_B)
$$
along with
$$
\dd \colon A\otimes B\to\Omega, \quad
a\otimes b \mapsto (\mathcal{G}\otimes_A\pi_I)\big(\Delta_\#(a\#b)-(a\#b)\otimes_A(1\#1)\big)
$$
for all $a\in A$ and $b\in B$.
\end{proposition}

\begin{example}
As a major class of examples, let us recall the covariant calculi of Heckenberger and Kolb \cite{HeckKolb} that are constructed on irreducible quantum flag manifolds. Fix a finite-dimensional complex semisimple Lie algebra $\mathfrak{g}$ of rank $r$, with Cartan subalgebra $\mathfrak{h}$ and Cartan matrix $(a_{ij})$. For a nonzero number $q\in\mathbb{R}$ that is not a root of unity, the \textit{Drinfel'd-Jimbo quantum group} $U_q(\mathfrak{g})$ is generated by $E_i,F_i,K^{\pm1}_i$, for $i=1,\ldots,r$, subject to the relations
\begin{align*}
  K_iE_j=q^{a_{ij}}E_jK_i,
  \quad
  K_iF_j=q^{-a_{ij}}F_jK_i,
  \quad
  K_iK_j=K_jK_i,
  \quad
	E_iF_j-F_jE_i=\delta_{ij}\frac{K_i-K_i^{-1}}{q_i-q_i^{-1}},
\end{align*}
and the so-called {\em quantum Serre relations}. This quantised enveloping algebra is a Hopf algebra with comultiplication and antipode determined by
\begin{equation*}
\begin{split}
	\Delta(K_i)
	&=K_i\otimes K_i,\\
	\Delta(E_i)
	&=E_i\otimes K_i+1\otimes E_i,\\
	\Delta(F_i)
	&=F_i\otimes 1+K_i^{-1}\otimes F_i,
\end{split}\qquad
\begin{split}
	\varepsilon(K_i)
	&=1,\\
	\varepsilon(E_i)
	&=0,\\\varepsilon(F_i)
	&=0,
\end{split}\qquad
\begin{split}
	S(K_i)
	&=K_i^{-1},\\
	S(E_i)
	&=-E_iK_i^{-1},\\
	S(F_i)
	&=-K_iF_i,
\end{split}
\end{equation*}
see \cite[\S6]{KlSch} for more details. Choosing a subset $\mathfrak{S}$ of simple roots, one defines the \textit{Levi subalgebra} 
$$
U_q(\mathfrak{l}_\mathfrak{S}):=\big\langle K^\pm_i,E_j,F_j \mid i=1,\ldots,r\text{ and }j\in\mathfrak{S} \big\rangle
$$
as the subalgebra of $U_q(\mathfrak{g})$ generated by all $K_i^\pm$ but only the $E_j,F_j$ corresponding to the chosen set of simple roots.
Then $U_q(\mathfrak{l}_\mathfrak{S})$ turns out to be a Hopf subalgebra and one obtains a left $U_q(\mathfrak{l}_\mathfrak{S})$-action on the dual coordinate algebra $\mathcal{O}_q(G)$ of $U_q(\mathfrak{g})$, of which the invariant subalgebra is denoted by
$$
\mathcal{O}_q(G/L_\mathfrak{S}):=\big\{ f\in\mathcal{O}_q(G) \mid x\cdot f=\varepsilon(x)f\text{ for all }x\in U_q(\mathfrak{l}_\mathfrak{S}) \big\}.
$$
Considering the dual coordinate algebra $\mathcal{O}_q(L_\mathfrak{S})$ of $U_q(\mathfrak{l}_\mathfrak{S})$, we can describe $\mathcal{O}_q(G/L_\mathfrak{S})$ equivalently as the Hopf kernel of the Hopf algebra surjection $\pi\colon\mathcal{O}_q(G)\to\mathcal{O}_q(L_\mathfrak{S})$, that is, $\mathcal{O}_q(G/L_\mathfrak{S})\cong\mathcal{O}_q(G)^{\mathrm{co}\,\mathcal{O}_q(L_\mathfrak{S})}$. It is well-known, see, for example, \cite[\S5]{AleRe}, that $\mathcal{O}_q(G/L_\mathfrak{S})$ is a principal homogeneous space, referred to as the \textit{quantum flag manifold} associated to $\mathfrak{S}$,
and called \textit{irreducible} if $\mathfrak{S}$ is obtained by crossing out one node in the Dynkin diagram of which
the corresponding simple root has coefficient $1$ in the expansion of the highest root of $\mathfrak{g}$.

It was shown in \cite{HeckKolb} that on every irreducible quantum flag manifold $\mathcal{O}_q(G/L_\mathfrak{S})$, there exists an (essentially unique) left $\mathcal{O}_q(L_\mathfrak{S})$-covariant first order differential calculus of classical dimension. By \cite{Hermisson}, this calculus then corresponds to a subobject $I_\mathfrak{S}\subseteq\mathcal{O}_q(G/L_\mathfrak{S})^+$ in ${}_{\mathcal{O}_q(G/L_\mathfrak{S})}^{\mathcal{O}_q(L_\mathfrak{S})}\mathcal{M}$. Using the fact that $\mathcal{O}_q(G)$ is a coquasitriangular Hopf algebra \cite[\S10.1.2]{KlSch}, one can construct a braided commutative algebra $A_\mathfrak{S}$ in ${}_{\mathcal{O}_q(L_\mathfrak{S})}\mathcal{YD}^{\mathcal{O}_q(G)}$: denoting the coquasitriangular structure by $\mathcal{R}\colon\mathcal{O}_q(G)\otimes\mathcal{O}_q(G)\to\mathbb{C}$ and the generators of $\mathcal{O}_q(G)$ by $u_i^j$ such that $\Delta(u_i^j)=u^j_k\otimes u^k_i$ and $\varepsilon(u_i^j)=\delta_i^j$, see {\em op.~cit.}, \S9, we define
$A_\mathfrak{S}$ as the free algebra generated by $\xi_1,\ldots,\xi_r$ modulo the relations
\begin{equation}
  \label{eq:bc}
	\xi_j\xi_i-\xi_k\xi_\ell\mathcal{R}(u^\ell_j\otimes u^k_i).
\end{equation}
One can verify that $A_\mathfrak{S}$ is a monoid in ${}_{\mathcal{O}_q(G)}\mathcal{YD}^{\mathcal{O}_q(G)}$ with respect to the right $\mathcal{O}_q(G)$-coaction and left $\mathcal{O}_q(G)$-action determined by
$$
A_\mathfrak{S} \to A_\mathfrak{S} \otimes \mathcal{O}_q(G), \ \
\xi_i\mapsto\xi_j\otimes u^j_i,
\qquad\quad
\mathcal{O}_q(G) \otimes A_\mathfrak{S} \to
A_\mathfrak{S}, \ \
u_i^j \otimes \xi_k \mapsto \xi_\ell\mathcal{R}(u_k^\ell\otimes u_i^j),
$$
respectively. By construction \eqref{eq:bc}, the monoid $A_\mathfrak{S}$ is braided commutative and thus $A_\mathfrak{S}$ is a braided commutative algebra in ${}_{\mathcal{O}_q(L_\mathfrak{S})}\mathcal{YD}^{\mathcal{O}_q(G)}$, as claimed. This discussion, combined with Proposition \ref{prop:CalcSmashHom}, implies the following result:
\begin{proposition}
Let $\mathfrak{g}$ be a finite-dimensional semisimple Lie algebra over $\C$ and $\mathcal{O}_q(G/L_\mathfrak{S})\cong\mathcal{O}_q(G)^{\mathrm{co}\mathcal{O}_q(L_\mathfrak{S})}$ an irreducible quantum flag manifold with respect to a subset $\mathfrak{S}$ of simple roots. Consider 
\begin{itemize}
  \compactlist{70}
\item
the subobject $I_\mathfrak{S}\subseteq\mathcal{O}_q(G/L_\mathfrak{S})^+$ in ${}_{\mathcal{O}_q(G/L_\mathfrak{S})}^{\mathcal{O}_q(L_\mathfrak{S})}\mathcal{M}$, which determines the Heckenberger-Kolb calculus on $\mathcal{O}_q(G/L_\mathfrak{S})$,
along with
\item
  a left $A_\mathfrak{S}$-ideal and left $\mathcal{O}_q(L_\mathfrak{S})$-ideal $I_{A_\mathfrak{S}}\subseteq A_\mathfrak{S}$.
\end{itemize} 
Then
$$
\Omega:=\big(A_\mathfrak{S}\#\mathcal{O}_q(G)\big)
\bx_{A_\mathfrak{S}\#\mathcal{O}_q(L_\mathfrak{S})}
\big(A_\mathfrak{S}\otimes\mathcal{O}_q(G/L_\mathfrak{S})^+\big)/
(I_{A_\mathfrak{S}}\otimes I_\mathfrak{S})
$$
and
\begin{align*}
  \dd \colon A_\mathfrak{S}\otimes\mathcal{O}_q(G/L_\mathfrak{S})&\to\Omega,
  \\
	\xi_i\otimes u_j^k &\mapsto \pig(\big(A_\mathfrak{S}\#\mathcal{O}_q(G)\big)\otimes_{A_\mathfrak{S}}\pi_{I_\mathfrak{S}}\pig)
        \pig((\xi_i\#u_\ell^k)\otimes_{A_\mathfrak{S}}(1\#u^\ell_j)-(\xi_i\otimes u_j^k)\otimes_{A_\mathfrak{S}}(1\#1)\pig)
\end{align*}
describe a left $A_\mathfrak{S}\#\mathcal{O}_q(G)$-covariant first order differential calculus on the Hopf algebroid principal homogeneous space $A_\mathfrak{S}\otimes\mathcal{O}_q(G/L_\mathfrak{S})\cong(A_\mathfrak{S}\#\mathcal{O}_q(G))^{\mathrm{co}A_\mathfrak{S}\#\mathcal{O}_q(L_\mathfrak{S})}$.
\end{proposition}
\end{example}

\addtocontents{toc}{\SkipTocEntry}
\section*{Declarations}

\addtocontents{toc}{\SkipTocEntry}
\subsection*{Data availability}
Data sharing is not applicable as no data sets were generated or analysed during the current study.

\addtocontents{toc}{\SkipTocEntry}
\subsection*{Conflict of Interest}
On behalf of all authors, the corresponding author states that there is no conflict of interest.

\end{document}